\documentclass[11pt]{amsart}
\usepackage{epsfig}
\usepackage{graphicx, subfigure}
\usepackage{amsmath, amssymb,mathabx,mathrsfs}
\usepackage{color}
\usepackage[toc,page]{appendix}
\usepackage{comment}

\newtheorem{thm}{Theorem}[section]
\newtheorem{cor}[thm]{Corollary}
\newtheorem{lem}[thm]{Lemma}
\newtheorem{prop}[thm]{Proposition}
\theoremstyle{definition}
\newtheorem{defn}[thm]{Definition}
\theoremstyle{remark}
\newtheorem{rem}[thm]{Remark}
\newtheorem{ex}[thm]{Example}
\numberwithin{equation}{section}

\theoremstyle{definition}

\begin{document}

\newcommand{\vc}[1]{\boldsymbol{#1}}
\newcommand{\di}{\mathrm{d}}
\newcommand{\de}{\partial}
\newcommand{\ru}{r_1}
\newcommand{\tu}{\theta_1}
\newcommand{\rd}{r}
\newcommand{\td}{\theta}
\addtolength{\hoffset}{-0.8cm}
\addtolength{\textwidth}{1.8cm}
\newcommand{\ds}{\displaystyle}
\newcommand{\eps}{\varepsilon}
\newcommand{\xt}{\textrm{exit}}
\newcommand{\nt}{\textrm{entry}}
\newcommand{\tF}{\tilde{F}}
\newcommand{\tsi}{\tilde{\sigma}}
\newcommand{\ty}{\tilde{y}}
\newcommand{\new}{\textcolor{red}}
\newcommand{\newB}{\textcolor{blue}}
\newcommand{\N}{\mathbb{N}}


\title[A General Mechanism  of Diffusion in Hamiltonian Systems]{A General Mechanism  of Diffusion in Hamiltonian Systems: Qualitative Results}

\author[Marian Gidea]{Marian\ Gidea$^\dag$}
\address{Yeshiva University, Department of Mathematical Sciences, New York, NY 10016, USA }
\email{Marian.Gidea@yu.edu}
\thanks{$^\dag$ Research of M.G. was partially supported by NSF grant  DMS-0635607, and by  the  Alfred P. Sloan Foundation grant G-2016-7320}
\author[Rafael de la Llave]{Rafael de la Llave${^\ddag}$}
\address{School of Mathematics, Georgia Institute of Technology, Atlanta, GA 30332, USA}
\email{rafael.delallave@math.gatech.edu}
\thanks{$^\ddag$ Research of R.L. was partially supported by NSF grant  DMS-1500943}
\author[Tere Searaf]{Tere Seara${^\flat}$}
\address{Departament de Matem\`{a}tica Aplicada I, Universitat Polit\`{e}cnica de Catalunya, Diagonal 647, 08028 Barcelona, Spain}
\email{Tere.M-Seara@upc.edu}
\thanks{$^\flat$ Research of T.S. was partially supported by Russian Scientific Foundation grant 14-41-00044.}

\begin{abstract}

We present a general mechanism to establish the existence of diffusing orbits in a large class of nearly integrable  Hamiltonian systems. Our approach relies on  successive applications of the `outer dynamics'  along homoclinic orbits to a normally hyperbolic invariant manifold.
The information on the outer dynamics is encoded by a geometrically defined map, referred to as the  `scattering map'.

We find pseudo-orbits of the scattering map that keep advancing in some  privileged  direction.
Then we use the recurrence property of the `inner dynamics', restricted to the normally hyperbolic invariant manifold,
to return to those pseudo-orbits.
Finally, we apply topological methods  to show the existence of true orbits that follow the successive applications of the  two dynamics.

This method  differs, in several crucial aspects,  from earlier works.
Unlike the well known `two-dynamics' approach, the method we present
relies on the outer dynamics alone.
There are virtually no assumptions on the inner dynamics, as its
invariant objects  (e.g., primary and secondary tori, lower dimensional hyperbolic tori and their
stable/unstable manifolds, Aubry-Mather sets) are not used at all.

The method applies to unperturbed  Hamiltonians of arbitrary degrees of freedom  that are not necessarily convex.
In addition, this mechanism is easy to verify (analytically or numerically)  in concrete examples, as well as to establish diffusion in generic systems.

We  include several applications, such as bridging large gaps in a priori unstable models in any dimension, and establishing diffusion in cases when the inner
dynamics is a non-twist map.
\end{abstract}


\maketitle

\section{Introduction}

In this paper we develop a general but simple method to show the existence of diffusing orbits in nearly integrable Hamiltonian systems in any dimension.
The main  requirement for the system is to have a  normally hyperbolic invariant manifold whose stable and unstable manifolds intersect transversally along
a transverse homoclinic manifold.
In this setting, one can geometrically define   a map  on the normally hyperbolic invariant manifold,  referred to as the scattering map \cite{DelshamsLS08a},
which accounts for the `outer' dynamics along homoclinic orbits.
The scattering map assigns  to the foot-point of an unstable fiber  the foot-point of a stable fiber, provided the two fibers meet at a unique point in the
homoclinic manifold.
On the  `inner dynamics', defined by the restriction  to the normally hyperbolic invariant manifold, we only require to satisfy  Poincar\'e recurrence.

The main results of this paper can be summarized as follows:

(i) For every pseudo-orbit\footnote{In this paper we use  the term pseudo-orbit  in the sense of an
orbit  of an iterated function system or poly-system; see  \cite{Marco2008}.} obtained by successively applying the scattering map,
there exists a true orbit  that shadows it (its intermediate points follow the pseudo-orbit).
The statement  is given in Theorem \ref{lem:main1}.

(ii) For a large class of nearly integrable Hamiltonian systems, which satisfy some verifiable conditions that occur generically, there  exist diffusing orbits
that travel a distance that is independent of the size of the perturbation.
More precisely, there exists a vector field whose integral curves approximate the pseudo-orbits of the scattering map.
If this vector field is non-trivial at some point, then  there exists an  integral curve that travels a distance of
order $1$ within the normally hyperbolic invariant manifold, and   hence a  pseudo-orbit of the scattering map that goes along that curve.
Assuming the inner dynamics satisfies Poincar\'{e} recurrence, there  also exists a true orbit that shadows it.
The statement is given in Theorem \ref{thm:main}.

(iii) A general type of the shadowing lemma, needed to establish the above results.
This lemma says  that for every infinite pseudo-orbit generated by alternatively applying the scattering map and the inner dynamics for sufficiently long time,
there exists a true orbit that shadows that pseudo-orbit.
The statement is given  in Lemma \ref{lem:key1-alternative}.

The above results remain valid if one considers several transverse homoclinic manifolds rather than a single one, and  hence  several scattering maps.
They also remain valid if one considers a sequence of manifolds (which may be of different topologies) chained via different heteroclinic connections,
which can also be described via scattering maps.

For the purpose of establishing the existence of diffusing orbits,  the assumption that the inner dynamics satisfies  Poincar\'e recurrence
on some bounded domain can be eliminated.
If there is no such domain, then there exist diffusing orbits determined just by the inner dynamics.

As a  concrete application of this method we obtain a qualitative result  on the existence of diffusing orbits in \emph{a priori unstable Hamiltonian
systems}  (see \cite{ChierchiaG94})  of any dimension, under verifiable conditions on the perturbation  that are  generically satisfied,
and under some mild conditions on  the unperturbed system.
In particular, the unperturbed Hamiltonian does not need to be convex.
The main requirement on the Hamiltonian  system is that we can compute perturbatively the scattering map.

The salient features  of the mechanism outlined above are the following:
\begin{enumerate}
\item
We do not  require any information on the inner dynamics.
In particular, we can obtain diffusing orbits whose action variable crosses resonant surfaces of any multiplicity.

This is a significant departure from previous approaches which rely on a detailed analysis of the invariant objects for the inner dynamics: primary KAM tori,
secondary tori,  lower dimensional hyperbolic tori and their stable and unstable manifolds, Aubry-Mather sets, etc.
In fact, we do not  need the inner dynamics to satisfy a twist condition, which is a key assumption  in previous geometric and  variational approaches.
In particular, the present mechanism does not present the
\emph{large gap problem}.
\item
The normally hyperbolic invariant manifold as well as its stable and unstable manifolds can be of arbitrary dimensions.
\item  We can take advantage of the existence of several scattering maps.
\item
Our method can be applied to concrete systems --  e.g,  the planar elliptic restricted three-body problem,  the spatial circular restricted three-body problem --,
and, further, can be implemented in computer assisted proofs.
See the  related papers \cite{CapinskiGL14,DelshamsGR2016}.
\item
Although the main application in this paper is on diffusion in a priori unstable systems, we  expect that this method can be useful when applied to \emph{a
priori stable systems}, as well as to   \emph{infinite dimensional systems}, once the existence of suitable normally hyperbolic invariant manifolds
(called normally hyperbolic cylinders in \cite{KaloshinZ12a, KaloshinZ12b, KaloshinBZ11, Marco13}) and their homoclinic channels  is established. See Remark \ref{rem:muepsilon}.
\end{enumerate}

We compare our method here with some previous approaches to the diffusion problem,  for different types of Hamiltonian systems.

It is customary to distinguish between  geometric methods and  variational methods.
The method in this paper is geometric, so we first compare it with some related approaches.

For nearly integrable Hamiltonian systems of two-and-a-half degrees of freedom, the existence of diffusion has  been established  via geometric methods in
\cite{DelshamsLS00,DelshamsLS2006}, by using the existence of KAM tori, primary and secondary, along the normally hyperbolic invariant manifold.
The perturbation in \cite{DelshamsLS2006} is assumed to be a trigonometric polynomial in the angle variable, but \cite{DelshamsH2009}  eliminates this assumption.
The integrable Hamiltonian is not assumed to be convex, which seems to be the standard assumption in many variational  approaches.
Similar type of results have been  obtained in  \cite{GideaL06,GideaL06b}  with the use of the method of correctly-aligned windows.
This allows to simplify the proofs, and to obtain explicit estimates on the diffusion speed.

The case of higher dimensional Hamiltonian systems poses a difficulty that is not present in the case of two-and-a-half degrees of freedom:
there are points where the resonances have higher multiplicity.
The technique involved in \cite{DelshamsLS00,DelshamsLS2006}  uses heavily that
in the neighborhood of resonances of multiplicity $1$ one can  introduce a normal form
which is integrable and  can be analyzed with great accuracy to obtain secondary tori.
Unfortunately, it is well known that multiple resonances, that is,  resonances of
multiplicity greater or equal than $2$, lead to normal forms that  are not
integrable and require  other techniques to be analyzed (see \cite{Haller97,Haller99}).

In \cite{DelshamsLS16} the authors adapted the methods used in two-and-a-half degrees of freedom
to show instability in higher dimensions.
Their approach relies on the basic fact that  multiple resonances happen in
subsets of codimension greater than $1$ in the space of actions  and therefore the diffusing trajectories can
contour them.

Now we mention some other types of approaches to the diffusion problem.

Geometric  methods  based on normally hyperbolic invariant manifolds that use  the separatrix map rather than the scattering map
appear in \cite{BolotinTreschev1999,Treschev02a,Treschev02b,Treschev02c,
Treschev04,Treschev12,Piftankin2006,PiftankinT07}.
Other geometric methods have  been applied in
\cite{Bounemoura12,DLS03,Kaloshin03,Llave04,GelfreichT2008,DGLS08,Marco13,Zhang11}.

Several authors have used  variational methods (either
local variational methods or global variational methods)
alone or in  combination with geometric methods,  to obtain results on diffusion.
This is the case, for example, in \cite{Bessi96,Bessi97,Mather04,Mather12,ChengY2004,ChengY2009,Bernard08, Cheng12,ChengX2015,KaloshinBZ11,KaloshinZ12a,KaloshinZ12b,KaloshinZ14, BessiCV01}.
We mention the  paper \cite{Tennyson82} who suggests several other mechanisms that
should be at play. It seems to be a  very challenging problem to make
rigorous the heuristic discussions on statistical and quantitative properties
of different instability mechanisms in the heuristic literature
\cite{Chirikov79,Lieberman83,Tennyson82}.

We also acknowledge that   many of the methods and ideas that appear in the works on the Arnold diffusion problem are owed to John Mather,
whose influence to the field cannot be overstated \cite{Mather04,Mather10,Mather12}.

The structure of this paper is as follows.
In Section 2 we review some background on normally hyperbolic invariant manifolds, scattering map, and recurrent dynamics.
In Section 3 we provide two  general results on the existence of diffusing orbits --
 Theorem \ref{lem:main1} and Theorem \ref{thm:main} --, as well as some corollaries.
We also provide a general shadowing lemma -- Lemma \ref{lem:key1-alternative} -- that is used in proving these results.
An application to establish the existence of diffusing orbits in a class of nearly integrable a priori unstable Hamiltonian systems that are
multi-dimensional both in the center and in the hyperbolic directions is given in Section 4.
A novelty is that the unperturbed  system corresponds to a Hamiltonian which is not necessarily convex, and that the inner dynamics does not need to satisfy  a twist condition.
Section 5 contains the proofs of the results stated in Section 3.
An Appendix with  definitions and tools that are utilized in the paper is included for the convenience of the reader.

\section{Background}

In this section, we cover some standard material
that will be used in the statement of the results. All the material
will be well known to experts.

\subsection{Normally hyperbolic invariant manifolds and scattering maps}\label{subsection:background}

Consider a discrete-time dynamical system given by the action of a $C^r$-smooth map $f$ on a   $C^r$-smooth manifold $M$, of dimension $m$, where $r\geq 1$.

Assume  that $\Lambda$ is a  normally hyperbolic invariant manifold (NHIM) in $M$, of dimension $n_c$:
this means that the tangent bundle of $M$ restricted to $\Lambda$ splits as a Whitney sum of sub-bundles
$TM_{\mid \Lambda}=T\Lambda\oplus E^u\oplus E^s$ which are invariant under $Df$, and that $(Df)_{\mid^u}$
expands more than $(Df)_{\mid T\Lambda}$, while $(Df)_{\mid E^s}$  contracts more than  $(Df)_{\mid T\Lambda}$.
We  also assume that $\Lambda$ is compact or that $f$ is uniformly $C^r$ in a neighborhood of $\Lambda$.
The rather standard definition is given in Appendix  \ref{subsection:NHIM}.

In the sequel we assume that the stable and unstable bundles associated to the normally hyperbolic structure have dimensions $n_u, n_s>0$,
respectively, where $n_c+n_u+n_s=m$.
(In many applications concerning diffusion in  nearly integrable Hamiltonian systems
we have $n_u=n_s =n$ and $n_c=$ even number, hence $m=$ even number.)

\begin{rem}
In the general theory of normally hyperbolic
manifolds one does not have the above restriction  on  dimensions, but for
symplectic systems, this is natural.  We also note that in the
symplectic case, it is natural to assume that the
stable and unstable rates \cite{Fenichel74} and that the
forward rates in the tangent direction are the same.
In such a case, one has automatically that the invariant
manifold is symplectic. See \cite{DelshamsLS08a}.
\end{rem}

The normal hyperbolicity  of $\Lambda$ implies that there exist stable and unstable invariant manifolds $W^s(\Lambda)$, $W^u(\Lambda)$ of $\Lambda$.
The exponential contraction and expansion rates  of $Df$ along the stable and unstable bundles, and on $T\Lambda$,
determine an integer $\ell$ with $0<\ell\leq r$, such that $\Lambda$ is $C^{\ell}$-smooth, and $W^s(\Lambda)$,
$W^u(\Lambda)$ are $C^{\ell-1}$-smooth. The stable and unstable manifolds $W^s(\Lambda)$, $W^u(\Lambda)$ are foliated
by stable and unstable fibers $W^s(x)$, $W^u(x)$, respectively, with $x\in\Lambda$,
which are $C^{r}$-smooth $1$-dimensional manifolds. The corresponding foliations are however only $C^{\ell-1}$-smooth. See Appendix  ~\ref{subsection:NHIM}.

From now on we assume that $r$ and the normally hyperbolic structure are so that $\ell\geq 2$.

Let $\Gamma\subseteq W^s(\Lambda)\cap W^u(\Lambda)$ be a $C^{\ell-1}$-smooth homoclinic manifold.
Consider the  wave maps
\begin{align}
\Omega^{-}: \Gamma \subset W^u(\Lambda)\to \Omega^{-}(\Gamma)\subseteq \Lambda,\, \Omega^{-}(x)=x^-,\\
\Omega^{+}: \Gamma \subset W^s(\Lambda)\to \Omega^{+}(\Gamma)\subseteq \Lambda,\, \Omega^{+}(x)=x^+,
\end{align}
where $x^-$ is the unique point in $\Lambda$ such that $x\in W^u(x^-)$, and $x^+$ is the unique point in $\Lambda$ such that $x\in W^s(x^+)$.

Under certain restrictions  on $\Gamma$, which are given explicitly in Appendix \ref{subsection:NHIM}, the wave maps $\Omega^\pm$ are
$C^{\ell-1}$-diffeomorphisms from $\Gamma$ to their images.
Such  a homoclinic manifold is referred to as a \emph{homoclinic channel}.

Assuming that $\Gamma$ is a homoclinic channel, one can define a $C^{\ell-1}$ diffeomorphism
$$
\sigma:\Omega^{-}(\Gamma)\to\Omega^{+}(\Gamma), \ \mbox{given by} \ \sigma=\Omega^{+} \circ (\Omega^{-})^{-1},
$$
where $\Omega^{-}(\Gamma), \Omega^{+}(\Gamma)$ are open sets in  $\Lambda$.
That is, $\sigma(x^-)=x^+$, for every $x^-\in \Omega^{-}(\Gamma)$ defined as above.
The mapping $\sigma$ is referred to as the scattering map associated to the homoclinic channel~$\Gamma$.
For details on this set-up and general properties of the scattering map  see Appendix \ref{subsection:NHIM}.

We shall note that there is no actual orbit of the system that goes from $x^-$ to $\sigma(x^-)=x^+$.
Rather,   the geometric object that corresponds to $\sigma(x^-)=x^+$ is the heteroclinic orbit $\{f^n(x)\}_{n\in\mathbb{Z}}$ of $x$,
which approaches asymptotically  $f^{n}(x^-)$ backwards in time,  as $n\to-\infty$, and respectively
$f^{n}(x^+)$ forward in time, as $n\to+\infty$.
We remark that, if  we denote by $\sigma^\Gamma$ the scattering map associated to  the  homoclinic channel~$\Gamma$,
then for each $k\in\mathbb{Z}$, $f^k(\Gamma)$ is also a homoclinic channel, and the corresponding scattering map
$\sigma ^{f^k(\Gamma)}$ is related to $\sigma^\Gamma$ by the invariance relation
\begin{equation}\label{eqn:scatteringinv}
\sigma ^{f^k(\Gamma)}=f^k\circ \sigma ^{\Gamma}\circ f^{-k}.
\end{equation}
While $\sigma^\Gamma$ and $\sigma^{f^k(\Gamma)}$ are technically different scattering maps,
they are geometrically the same, as they are defined via the same homoclinic channel (up to iterations by the map $f$).
Of course, homoclinic channels  that are not obtained from one another via iteration  yield, in general, to scattering maps that are geometrically different.

In many examples, the scattering map can be computed explicitly via perturbation theory \cite{DelshamsL2000,DelshamsLS2006,DelshamsLS2006b},
or numerically \cite{CanaliasDMR06,DelshamsGR13,DelshamsGR2016,CapinskiGL14}.

\subsection{Normally hyperbolic invariant manifolds and scattering maps in a symplectic perturbative setting}
\label{subsection:perturbative}

Assume now that $(M,\omega)$ is a symplectic manifold, and $f_\eps:M\to M$ is a $C^r$-family of  symplectic maps,
where $\eps\in(\eps_0,\eps_0)$, for some $\eps_0>0$.

For example, one can think of $f_\eps$ as being the time-$1$ map  associated to the Hamiltonian flow $\phi_{t,\eps}$ corresponding to a Hamiltonian
$H_\eps:M\to \mathbb{R}$ of the form
\begin{equation}\label{Heps}
H_\eps=H_0+\eps H_1,
\end{equation}
where $H_0$ is an integrable Hamiltonian;
in this case the maps $f_\eps$ with $\eps\neq 0$ can be viewed as $\eps$-perturbations of the map $f_0$,
which is the time-$1$ map for the unperturbed Hamiltonian flow of $H_0$.

Assume that there exists a normally hyperbolic invariant manifold $\Lambda_\eps\subseteq M$   for $f_\eps$, of even dimension $n_c$,
for all $\eps\in(-\eps_0,\eps_0)$, and that $\dim W^u(\Lambda_\eps)=n_c+n_u=\dim W^s(\Lambda_\eps)=n_c+n_s$.
Assume that $\Lambda_\eps$ is symplectic and  denote by $J$ the linear operator associated to $\omega_{\mid\Lambda_{\eps}}$ by the metric.
Then the  map  $f_\eps$ is also symplectic on $\Lambda_\eps$.

Assume that  for each  $\eps\in(-\eps_0,\eps_0)$ there exists  a homoclinic channel  $\Gamma_\eps$   for $f_\eps$ that depends $C^{\ell-1}$-smoothly on $\eps$.
Then the scattering map
$\sigma_\eps:\Omega^{-}(\Gamma_\eps)\to \Omega^{-}(\Gamma_\eps)$ is also symplectic and $C^{\ell-1}$ (see \cite{DelshamsLS08a,DelshamsGLS2008}).

Now we assume that $\Lambda_\eps$ can be parametrized
via a $C^{\ell}$-diffeomorphism $k_\eps:\Lambda_0\to\Lambda_\eps$,
for $\eps\in(-\eps_0,\eps_0)$, where  $\Lambda_0$  is the normally hyperbolic manifold for the unperturbed map $f_0$, and $k_0=\textrm{Id}_{\Lambda_0}$.
This happens, for example, when   the $\Lambda_\eps$'s are obtained by the persistence of normal hyperbolicity under sufficiently small perturbations
(see \cite{DelshamsLS2006}).

Via the  parametrizations $k_\eps$, each map $f_\eps$ induces a map $\tilde f_\eps$ on $\Lambda_0$ by
\[
\tilde f_\eps= k_\eps^{-1}\circ (f_\eps)_{\mid\Lambda_\eps}\circ  k_\eps.
\]

The scattering map $\sigma_\eps:\Omega^{-}(\Gamma_\eps)\subset \Lambda_\eps \to \Omega^{+}(\Gamma_\eps)\subset \Lambda_\eps $ can also be expressed
in terms of the reference manifold $\Lambda_0$ by
$$
\tilde \sigma_\eps:   k_\eps^{-1}(\Omega^{-}(\Gamma_\eps))\subset\Lambda_0 \to k_\eps^{-1}(\Omega^{+}(\Gamma_\eps))\subset \Lambda_0
$$
given by
\[
\tilde \sigma_\eps=k_\eps^{-1}\circ\sigma_\eps\circ k_\eps.
\]
We will refer to the map $\tilde \sigma_\eps$ also as the `scattering map'.

In this  setting, one also has an unperturbed scattering map  $\sigma_0$ on  the unperturbed manifold $\Lambda_0$, associated to the  homoclinic channel
$\Gamma_0$ contained in the intersection between the stable and unstable manifolds of $\Lambda_0$.
Of course, in the unperturbed case one  has   $\tilde \sigma_0=k_0^{-1}\circ \sigma_0\circ k_0=  \sigma_0$.
Expressing both the perturbed and the unperturbed scattering map  as maps on the same (unperturbed) manifold is quite advantageous,
as one can compare them relative to the same coordinate system.


For a  Hamiltonian system $H_\eps$ as in \eqref{Heps},   \cite{DelshamsLS08a}
provides  a perturbative formula for the scattering map:
\begin{equation}\label{eqn:scattering_perturb}
\tilde\sigma_\eps=\tilde \sigma _0+\eps J\nabla S \circ \tilde \sigma _0+O(\eps ^2)
\end{equation}
where
$S$ is a real valued $C^\ell$-function on $\Lambda_0$
that can be computed explicitly in terms of convergent integrals of the perturbation evaluated along homoclinic trajectories of the  unperturbed system
(see \cite{DelshamsLS08a,GideaL16}):
 \begin{align}\label{eq:S-def}
S\left( x\right) & =\lim_{T\rightarrow+\infty}\int
_{-T}^{0}\left [\frac{dH_\eps}{d\eps}_{\mid\eps=0}\circ \phi_{t}\circ\left( \Omega_{-}^{\Gamma_0} \right) ^{-1}\circ \sigma_0^{-1}(x)  \right. \\
&\left. \qquad\qquad\qquad-\frac{dH_\eps}{d\eps}_{\mid\eps=0}\circ \phi_{t} \circ \sigma_0 ^{-1}(x)\right]dt \notag \\
& +\lim_{T\rightarrow+\infty}\int_{0}^{T}\left [\frac{dH_\eps}{d\eps}_{\mid\eps=0}\circ \phi_{t}\circ\left( \Omega_{+}^{\Gamma_0} \right) ^{-1}(x)\right.  \notag \\
& \qquad\qquad\qquad\left.-\frac{dH_\eps}{d\eps}_{\mid\eps=0}\circ \phi_{t}(x)\right]dt .  \notag
\end{align}
Here $\phi_t=\phi_{t,0}$ is the flow corresponding to the unperturbed Hamiltonian~$H_0$.

Note that,  by definition, there exists $z\in \Gamma_0$ such that $\phi_t(z)$ is a heteroclinic orbit, or, equivalently,
$z\in W^u((\sigma_0)^{-1}(x))\cap W^s(x)$.
Therefore, the formula \eqref{eq:S-def} can also be written as:
 \begin{align}\label{eq:S-def2}
S\left( x\right) & =\lim_{T\rightarrow+\infty}\int
_{-T}^{0}\left [\frac{dH_\eps}{d\eps}_{\mid\eps=0}\circ \phi_{t}( z) -\frac{dH_\eps}{d\eps}_{\mid\eps=0}\circ \phi_{t} (\sigma_0 ^{-1}(x))\right]dt  \\
& +\lim_{T\rightarrow+\infty}\int_{0}^{T}\left [\frac{dH_\eps}{d\eps}_{\mid\eps=0}\circ \phi_{t}(z)-
\frac{dH_\eps}{d\eps}_{\mid\eps=0}\circ \phi_{t}(x)\right]dt .  \notag
\end{align}
The normal hyperbolicity of $\Lambda_0$ ensures that $\phi_{t}( z) - \phi_{t} (\sigma_0 ^{-1}(x))$
and $\phi_{t}( z) - \phi_{t} (x)$ converge to zero exponentially fast as $t\to \mp\infty$ respectively.
This makes the integral in \eqref{eq:S-def2} absolutely convergent with its derivatives.

In some cases  it is possible that, when $\eps=0$, the stable and unstable manifolds of $\Lambda_0$ coincide, i.e., $W^u(\Lambda_0)=W^s(\Lambda_0)$.
In these  cases, one usually uses  first order perturbation theory to establish the splitting of the manifolds.
Using an adapted Melnikov method, in \cite{DelshamsLS2006,GideaL16} it is shown that, under appropriate conditions,
for $0 < |\eps| \ll 1$, one can find a transverse intersection of $W^u(\Lambda_\eps)$ with $W^s(\Lambda_\eps)$
along a manifold $\Gamma_\eps$, which extends smoothly to a homoclinic manifold $\Gamma_0$ as $\eps\to 0$.
While the limiting manifold $\Gamma_0$ is not a transverse intersection, the scattering map $\sigma_\eps$
depends smoothly on $\eps$, and thus extends smoothly to a well defined map  $\sigma_0$ on $\Gamma_0$.

The special case when $\sigma_0=\mathrm{Id}$, which occurs in many examples,  will be considered in Section~\ref{subsection:aprioriunstable},
where a more explicit formula for the function $S(x)$ is given in terms of the so-called Melnikov potential.

\subsection{Recurrence}\label{recurrence}
We briefly recall here the definition of recurrent points and the Poincar\'e Recurrence Theorem, which will be needed later.

\begin{defn}
A point $x\in\Lambda$ is said to be recurrent for a map $f$ on $\Lambda$, if for every open neighborhood
$U\subseteq \Lambda$ of $x$,  $f^{k}(x)\in U$ for some $k>0$ large enough.
\end{defn}

\begin{thm}[Poincar\'e Recurrence Theorem]
Suppose that $\mu$ is a measure on $\Lambda$ that is preserved by $f$, and $D\subset \Lambda$ is $f$-invariant with $\mu(D)<\infty$.
Then $\mu$-almost every point of $D$ is recurrent.
\end{thm}

Instead of the Poincar\'e Recurrence Theorem, in the arguments below we can use the following weaker statement on recurrence.

\begin{prop}\label{prop:weakrecurrence}
Suppose that $\mu$ is a measure on $\Lambda$ that is preserved by $f$, and $D\subset \Lambda$ is $f$-invariant with $\mu(D)<\infty$. Then for every open set $U\subseteq D$, there exists $n\geq 1$ such that $f^n(U)\cap U\neq \emptyset$;
moreover, $n$ can be chosen arbitrarily large.
\end{prop}

\section{Main results}\label{section:shadowinglemmas}

The aim of this section is to provide a master theorem -- Theorem \ref{lem:main1} -- that guarantees the existence of diffusing orbits in a general framework.

\subsection{Shadowing of pseudo-orbits of the scattering map}

The main result of this section is:

\begin{thm}[Shadowing Lemma for  Orbits of   the Scattering Map]\label{lem:main1}
Assume that $f:M\to M$ is a sufficiently smooth map, $\Lambda\subseteq M$  is a normally hyperbolic invariant manifold
with stable and unstable manifolds which intersect transversally along an homoclinic channel $\Gamma\subseteq M$,
and $\sigma$ is the scattering map associated to~$\Gamma$.

Assume that  $f$ preserves a measure   absolutely continuous with respect to the Lebesgue measure on $\Lambda$,
and that $\sigma$ sends positive measure sets to positive measure sets.

Let $\{x_i\}_{i=0,\ldots,n}$ be a finite pseudo-orbit of the
scattering map in $\Lambda$, i.e., $x_{i+1}=\sigma(x_i)$, $i=0,\ldots,n-1$, $n\geq 1$,
that is contained in some open set $\mathcal{U}\subseteq \Lambda$  with almost every point  of $\mathcal{U}$ recurrent
for $f_{\mid\Lambda}$. (The points $\{x_i\}_{i=0,\ldots,n}$ do not have to be themselves recurrent.)

Then, for every $\delta>0$ there exists     an orbit $\{z_i\}_{i=0,\ldots,n}$ of $f$ in $M$,
with $z_{i+1}=f^{k_i}(z_i)$ for some $k_i>0$,  such that $d(z_i,x_i)<\delta$ for all $i=0,\ldots,n$.
\end{thm}

The proof of this result, given in Subsection \ref{prooflem:main1}, uses the given pseudo-orbit of the scattering map, and
the recurrence property of the inner dynamics, to produce another pseudo-orbit that intertwines the scattering map and the inner dynamics.
Then a shadowing lemma type of result -- Lemma \ref{lem:key1-alternative} below -- yields a true orbit of the system.

To apply Theorem \ref{lem:main1}, one needs to find orbits of the scattering map that follow desired itineraries.
For example, one may wish to  find  a pseudo-orbit  of the scattering map that travels  a `long distance' in $\Lambda$.
If such a pseudo-orbit is found,  Theorem \ref{lem:main1} yields a true orbit   that also travels the same large distance.

We emphasize that Theorem \ref{lem:main1} is very general, as the requirements on the scattering map and on the inner dynamics are
automatically satisfied in many situations.
If $M$ is endowed with a symplectic form $\omega$, $\omega_{\mid\Lambda}$ is symplectic,  and $f$ is also symplectic,
then $f_{\mid \Lambda}$ is symplectic and the scattering map $\sigma$ is also symplectic (see \cite{DelshamsLS08a}).
Thus, $f$ and $\sigma$ are volume preserving, and Theorem \ref{lem:main1} applies.

We  have the following remarkable dichotomy.
Then either:
\begin{itemize}
\item [I.]
The inner map $f_{\mid \Lambda}$ has an invariant open set $\mathcal{U}$ containing   the domain of
the scattering map, and on which there is Poincar\'e Recurrence.
Under generic conditions,   the scattering map has a pseudo-orbit  that travels a long distance within  $\mathcal{U}$.
Applying Theorem \ref{lem:main1} yields the existence of a true orbit which travels a long distance as well.
Therefore we obtain diffusion by intertwining the inner and outer dynamics.

\item[II.]
There is no  open set of finite measure   in $\Lambda$ that is invariant under $f_{\mid \Lambda}$.
Hence there are orbits  of $f$ that leave every open set  in $\Lambda$, thus traveling long distances.
Therefore we obtain diffusion by  the inner map $ f_{\mid \Lambda}$  alone.
\end{itemize}

In both alternatives we obtain  diffusing orbits.

A precise formulation of this dichotomy is given in Corollary \ref{cor:main}.

Note that in Theorem~\ref{lem:main1} we do not require that $f$ satisfies a twist condition,
which seem  to be essential  in many other works.
In general,  non-twist maps of the annulus have regions where   standard methods such as  KAM theory and Aubry-Mather theory do not apply
(see \cite{MorrisonCNG1997,DelshamsL2000}).

Indeed, in Theorem~\ref{lem:main1}, we do not need to make
any qualitative assumption for the map $f$. In particular, we do
not care of whether the map has KAM tori that are close enough.
That is, the mechanism presented here  does not present the
\emph{large gap problem}.

Theorem \ref{lem:main1} extends naturally to the case of finitely many scattering maps rather than a single one.
Suppose that there exists a finite collection of homoclinic channels    $\Gamma_j\subseteq M$, for $j\in\{1,\ldots,L\}$,
for some positive integer $L$.  Let $\sigma_j:\Omega^{-}(\Gamma_j)\to \Omega^{+}(\Gamma_j)$ be the  scattering map associated to
$\Gamma_j$, for $j=1,\ldots,L$.

Using many scattering maps in arbitrary order
rather than just one is very advantageous to
prove diffusion. Iterating a single map has obstructions for
large scale motions (e.g., KAM tori). Having several maps, it is very
hard to find objects that are invariant for all of them.
See \cite{DelshamsLS00,GideaL06,GideaR07,GideaL13, Bolotin06,BolotinM06,Marco2008}.

\begin{thm}\label{lem:mainw}
Assume that $f:M\to M$, $\Lambda\subseteq M$, $\Gamma_j\subseteq M$ and $\sigma_j$, $j=1,\ldots,L$, are as above.
Assume that $f$ preserves  a measure
absolutely continuous with respect to the Lebesgue measure on $\Lambda$ and that each $\sigma_j$
send positive measure sets to positive measure sets.
Let $\{x_i\}_{i=0,\ldots,n}$ be a finite sequence of points of the form $x_{i+1}=\sigma_{\alpha_i}(x_i)$ in $\Lambda$,
where $\alpha_i\in \{1,\ldots,L\}$ for  $i=0,\ldots,n-1$, which is contained in some open set  $\mathcal{U}\subseteq \Lambda$
with the property that almost every point of $\mathcal{U}$ is recurrent  for $f_{\mid\Lambda}$.
Then, for every $\delta>0$ there exists     an orbit $\{z_i\}_{i=0,\ldots,n}$ of $f$ in $M$, with $z_{i+1}=f^{k_i}(z_i)$ for some $k_i>0$,  such that
$d(z_i,x_i)<\delta$ for all $i=0,\ldots,n$.
\end{thm}

\begin{rem}
In general situations, one has an abundance of homoclinic orbits.
By the Smale-Birkhoff Homoclinic Orbit Theorem  the existence of a single transverse homoclinic orbit implies the
existence of infinitely many transverse homoclinic orbits that are geometrically distinct.
Thus one is able to define many scattering maps.

In applications, using several scattering maps rather than a single one can be very advantageous.
In astrodynamics,  for example, the existence of multiple homoclinic intersections
can be exploited to obtain diffusion \cite{DelshamsKRS14, FejozGKR11}
and to increase the versatility of space missions.
See, e.g., \cite{CapinskiGL14,DelshamsGR2016}.
\end{rem}

\begin{rem}
Using several scattering maps   can also be useful to prove diffusion in generic systems.
In some perturbative problems, e.g., as in Section \ref{subsection:aprioriunstable}, the scattering map can be computed in terms of
convergent integrals of the perturbation evaluated along a homoclinic of the unperturbed system.
One can ensure that the scattering map has non-trivial effects by verifying that such an integral is non-zero.
Thus, given a perturbation, one can slightly modify it, using a bump function supported  in some tubular neighborhood of the homoclinic,
to obtain a nearby perturbation for which the corresponding scattering map exhibits  the desired non-trivial effects.
Having available multiple homoclinics, one can use bump functions supported in disjoint tubular neighborhoods of each of these  homoclinics to obtain multiple
scattering maps that exhibit different types of non-trivial behaviors. See, e.g., \cite{ChengY2004,ChengY2009,GideaL13,GideaL16}.
\end{rem}

\begin{rem}
The results above also generalize to the case of several NHIM's. If
$$
\Gamma_{1,2}\subseteq W^u_{\Lambda_1}\cap W^s_{\Lambda_2}
$$
is a heteroclinic channel  between two NHIM's
$\Lambda_1,\Lambda_2$, we can define a scattering map
$$
\sigma_{1,2}:\Omega^{-} (\Gamma_{1,2})\subseteq \Lambda_1 \to \Lambda_2
$$
in a similar fashion to the case of a single NHIM.
If we are given a chain of manifolds $\Lambda_i$, $i=1,\ldots,n$, and scattering maps
$$
\sigma_{i,i+1}:\Omega^{-}(\Gamma_{i,i+1})\subseteq \Lambda_i \to \Lambda_{i+1}, \ i=1,\ldots,n-1,
$$
then we can shadow orbits of the form
$y_{i+1}= \sigma_{i,i+1} (y_i)$, with
$y_i\in\Lambda_i$ and $y_{i+1}\in\Lambda_{i+1}$, for $i=1,\ldots,n-1$.
Such scattering maps appear in the study of double resonances \cite{Mather12,KaloshinBZ11,KaloshinZ12a,KaloshinZ12b}.
We hope to come back to this problem.

Another problem where one has scattering maps between two different normally hyperbolic invariant manifolds is the problem of two rocking blocks under periodic
forcing \cite{GranadosHS14}.
\end{rem}

\subsection{A qualitative mechanism of diffusion in nearly integrable Hamiltonian systems}\label{subsection:qualitative-diffusion}

We now describe several situations when we can construct  pseudo-orbits of the scattering map that travel a significant distance within the normally
hyperbolic invariant manifold,
and so Theorem \ref{lem:main1} can be applied to obtain true orbits nearby. More concrete conditions that yield such orbits
in some concrete examples appear in Section~\ref{subsection:aprioriunstable}.

We consider the perturbative setting described in Section \ref{subsection:perturbative},
where $f_\eps:M\to M$  is a  symplectic map,
$\Lambda_\eps\subseteq M$  is a normally hyperbolic invariant manifold (not necessarily compact) for $f_\eps$,   $\Gamma_\eps$ is a homoclinic channel
for $f_\eps$, and $\sigma_\eps~:~\Omega^{-}(\Gamma_\eps)\to \Omega^{+}(\Gamma_\eps)$ is the  corresponding scattering map, for $\eps\in(\eps_0,\eps_0)$.
We assume that  $\Lambda_\eps$ is described via a  parametrization
$k_\eps:\Lambda_0\to\Lambda_\eps$,
and let
$(\tilde f_\eps)_{\mid\Lambda_0}=k_\eps^{-1}\circ (f_\eps)_{\mid\Lambda_\eps}\circ k_\eps$, $\tilde \sigma_\eps=k_\eps^{-1}\circ\sigma_\eps\circ k_\eps$.
We also assume that
$\Lambda_0=B^d\times\mathbb{T}^d$,
and that  we have a system of  action-angle coordinates
$(I,\phi)$ on $\Lambda_0$ with $I\in B^d$ and $\phi\in\mathbb{T}^d$,
where
$B^d\subseteq \mathbb{R}^d$ is a disk in $\mathbb{R}^d$ or $B^d=\mathbb{R}^d$.
Here $\mathbb{T}^d=\mathbb{R}^d/\mathbb{Z}^d$.

Below, in Theorem \ref{thm:main}, we will use the perturbative formula for the scattering map \eqref{eqn:scattering_perturb} with  $\tilde \sigma_0=\textrm{Id}$,
and with a slightly more general first order perturbation term of the scattering map.
This allows to apply the result of Theorem \ref{thm:main} to more degenerate cases,
where second order perturbation theory is necessary to detect the transversality between the stable and unstable manifolds, or to
the so-called `a priori stable' case, where the Melnikov potential can be  exponentially small in $\varepsilon$.
See Remark \ref{rem:muepsilon}.

\begin{thm}\label{thm:main}
Assume that for all $\eps\in(-\eps_0,\eps_0)$, there exists a  scattering map $\sigma_\eps$, defined in a domain
$U:=k_\eps^{-1}(\Omega^{-}(\Gamma_\eps))\subset \Lambda_0$, such that

\begin{equation}\label{eqn:seps}
\tilde\sigma_\eps=\mathrm{Id} +\mu(\eps) J\nabla S +g(\mu(\eps)), \ g(\mu(\eps))=o(\mu(\eps)),
\end{equation}
where
$S$ is some real valued $C^\ell$-function on $U\subset \Lambda_0$, and $g(\mu(\eps))$,
$\mu(\eps)$ are some
$C^\ell$-functions, being  defined on $(-\eps_0,\eps_0)$ with $\mu(0)=0$;
by $g(\mu(\eps))=o(\mu(\eps))$ we mean that $\lim_{\eps\to 0} {g(\mu(\eps))}/{\mu(\eps)} =0.$

Suppose that $J\nabla S(\tilde x_0)\neq 0$ at some point $\tilde x_0\in U\subset \Lambda_0$.

Let $\tilde{\gamma}:[0,1]\to \Lambda_0$ be  an integral curve through $\tilde x_0$ for the vector field $J\nabla S $.
Suppose  that there exists a  neighborhood
$\mathcal{U}_{\tilde{\gamma}}\subset U$ of $\tilde{\gamma} ([0,1])$ in $\Lambda_0$
such that a.e. point in  $\mathcal{U}_{\tilde{\gamma}}$ is recurrent for $\tilde f_{\eps \mid\Lambda_0}$.
Let $\gamma_\eps=k_\eps\circ\tilde{\gamma}$ be the corresponding curve in $\Lambda_\eps$.

There exists $\eps_1>0$ sufficiently small, and a constant $K>0$,
such that  for every $\eps\in(-\eps_1,\eps_1)\setminus\{0\}$ and every  $\delta >0$,
there exists an orbit $\{z_i\}_{i=0,\ldots,n}$ of $f_\eps$ in $M$, with $n=O(\mu(\eps)^{-1})$, such that for all $i=0,\ldots,n-1$,
\[
z_{i+1}=f^{k_i}_{\eps}(z_i),\quad \textrm{ for some }k_i>0,\]  and for all $i=0,\ldots,n$, we have
\[
d(z_i,\gamma_\eps(t_i))<\delta +K(\mu (\eps)+ |g(\mu(\eps))/\mu(\eps)|), \textrm{ for  }t_i=i\cdot \mu(\eps),
\]
where
$0= t_0<t_1<\ldots <t_n\le 1$.
\end{thm}

The proof of this theorem is given in Subsection \ref{proofthm:main}.

We will refer to a solution curve $\tilde{\gamma}$ in $\Lambda_0$ as in the statement of Theorem \ref{thm:main},
or to its corresponding curve  $\gamma_\eps=k_\eps(\tilde{\gamma})$ in $\Lambda_\eps$, as a `scattering path',
as it represents an approximation of an orbit of the scattering map. See Fig. \ref{fig:scattering_path}.
So the previous result can be stated that, given any scattering path, there exits a true orbit of the system that shadows it.
Since one can typically find a scattering path for which the action variable changes by some positive distance independent of $\eps$,
implicitly one can  find a true orbit for which the action variable changes by $O(1)$; this is stated precisely in the following corollary.

There exists a sufficiently small neighborhood $V_{\Lambda_\eps}$ of $\Lambda_\eps$ in $M$ such that for every point $z\in V_{\Lambda_\eps}$
there exists a unique point $z'\in\Lambda_\eps$ which is the closest point to $z$.
The point $z'$ is the image of some unique point $\tilde z\in\Lambda_0$ via $k_\eps$, i.e., $z'=k_\eps(\tilde z)$.
 We denote by $I(z)$ the $I$-coordinate of the corresponding point  $\tilde z\in\Lambda_0$, i.e., $I(z):=I(\tilde z)$.

\begin{cor}\label{cor:main}
Assume that a scattering map $\sigma_\eps$   as in Theorem \ref{thm:main} is given.
If $J\nabla S$ is transverse to some level set $\{I=I_*\}$ in $\Lambda_0$ at some point $(I_*,\phi_*)\subset U$, then there exist
$0<\eps_1<\eps_0$ and $\rho>0$,
such that for every $0<\eps <\eps_1$ there exists an orbit $\{z_i\}_{i=0,\ldots,n}$ of $f_\eps$, such that \[\|I(z_n)-I( z_0)\|>\rho.\]
\end{cor}

The proof of Corollary~\ref{cor:main} is given in Subsection \ref{proofcor:main}.

\begin{rem}
Let us note that the diffusion orbit $\{z_i\}_{i=0,\ldots,n}$ obtained in Corollary \ref{cor:main} does not necessarily follow
a given pseudo-orbit of the scattering map.
If the dynamics given by $f_\eps$ has diffusing orbits, these are the ones obtained in the corollary.
In case the dynamics of $f_\eps$ remains in a bounded set,
we need to follow  the pseudo-orbits of the scattering map $\sigma_\eps$ to obtain the diffusing ones.
\end{rem}

\begin{rem}
We note that, in order to obtain a trajectory that achieves a change in the $I$-variable of order $O(1)$, the scattering map needs to be applied
$n=O( \mu(\eps)^{-1})$ times.
However, the true orbit that achieves the $O(1)$-change in the $I$-variable follows
not only the scattering map but also some recursive orbit segments  of the  inner dynamics, as in the proof of  Theorem  \ref{lem:main1}.
Since these recursive orbit segments of the inner dynamics are obtained by invoking the Poincar\'e recurrence theorem,
the above result does not yield an estimate for the time required to  follow the inner dynamics, hence does not directly lead to an estimate on the diffusion time.
\end{rem}

\begin{rem}\label{rem:muepsilon}
The condition that the unperturbed scattering map is the identity, i.e.,
$\tilde \sigma_0=\textrm{Id}$, is naturally satisfied in some examples,   e.g., in the a priori unstable system in Section \ref{subsection:aprioriunstable}.
The function $\mu(\eps)$ is associated to the size of the splitting of $W^u(\Lambda_\eps)$, $W^s(\Lambda_\eps)$.

In the example   in Section \ref{subsection:aprioriunstable}, we have $\mu(\eps)=\eps$  and $g(\mu(\eps))= O(\eps ^2)$ in the generic case.
Nevertheless, in some degenerate cases, it can happen that, up to first order in $\varepsilon$,
the perturbed stable and unstable invariant manifolds of $\Lambda_\varepsilon$ coincide.
In these cases it is necessary to go to second order perturbation theory  to distinguish them
and therefore $\mu(\eps)=\eps ^2$ and $S$ in \eqref{eq:S-def2} has a different expression (not given here) in terms of the second order variationals along the unperturbed homoclinic orbit.

Another special situation occurs in the so-called `a priori stable' systems, where the unperturbed system is completely integrable without any hyperbolic structure.
In those cases,  the a priori unstable structure appears after some first order partial averaging near simple resonances, giving rise to a system of the form
$\tilde H_\eps ^0+\tilde H_\eps^1$.
Therefore the  analogue of the unpertubed homoclinic orbit $\phi_t(z)$,  which  appears in the formulas of  the scattering map
\eqref{eq:S-def2}, is $\eps$-dependent, i.e., $\phi_{t,\eps}(z)$. The splitting between the stable and unstable manifolds  behaves differently from the a priori unstable case with respect to the
perturbation parameter.
Concretely, we  have:
$$
\tilde\sigma_\eps=\mathrm{Id} + J\nabla S +g(\mu(\eps)),
$$
where:
\begin{align}
S\left( x,\eps\right) & =\lim_{T\rightarrow+\infty}\int
_{-T}^{0}\left [\tilde H^1_\eps\circ \phi_{t,\eps}( z) -\tilde H^1_\eps\circ \phi_{t,\eps} (x)\right]dt \notag \\
& +\lim_{T\rightarrow+\infty}\int_{0}^{T}\left [\tilde H^1_\eps\circ \phi_{t,\eps}(z)-
\tilde H^1_\eps\circ \phi_{t,\eps}(x)\right]dt .  \notag
\end{align}
and $S$ and $g$ satisfy:
$$
S\left( x,\eps\right)=O(\mu(\eps)), \quad
g(\mu(\eps))=o(\mu(\eps)).
$$
If the system is analytic   there is an  exponentially small splitting of the separatrices and therefore
$\mu(\eps)=O(\eps^p\exp(-q\eps^{-r}))$, for some $p,q,r \in \mathbb{Q}$,
as in  \cite{BaldomaFGS12}.
Nevertheless, to obtain the behaviour of the error function $g(\mu(\eps))$ in general analytic a priori stable systems
is still an open and difficult question.
If the system is only smooth one usually has $\mu(\eps)=\eps^p$, for $p\geq 2$.

Besides the above comments, we want to stress that, once a formula like \eqref{eqn:seps} is stablished, the results of Theorem \ref{thm:main} remain true.
\end{rem}

\subsection{Shadowing of pseudo-orbits obtained by interspersing the inner dynamics with scattering maps}\label{subsection:keylemmas1}

In this section we provide a rather general shadowing lemma-type of result  that is needed for the proof of  Theorem \ref{lem:main1}.

Let $\Lambda$ be a NHIM as in  the Subsection \ref{subsection:background}. There are two maps defined acting on~$\Lambda$:
the scattering map $\sigma$ -- the outer dynamics --, which is typically defined on some sub-domain of $\Lambda$,
called $\Omega ^-(\Gamma)$ in Section \ref{subsection:perturbative},
and the restriction of $f$ to $\Lambda$  -- the inner dynamics.
In principle, one can act on $\Lambda$ by applying either map in any succession, however this does not yield true orbits of the system but only pseudo-orbits.

The shadowing lemma below  says that for every pseudo-orbit  obtained by alternately applying a  single scattering map
and some power of the inner map, there exists a true
orbit of the system that shadows that pseudo-orbit.
The pseudo-orbits that we consider are of the form $y_{i+1}=f^{m_i}\circ \sigma \circ f^{n_i}(y_i)$.
The resulting shadowing orbits  are of the form $z_{i+1}=f^{m_i+n_i}(z_i)$, where $z_i$ is $\delta$-close to $y_i$ for all $i$.
We point out that we do not claim that  all points of the orbit $\{f^n(z_0)\}_{n\geq 0}$ are close to those of the pseudo-orbit,
but only some points corresponding to some intermediate times, and this is the sense in which we understand  shadowing orbits here.

The orders of the iterates $n_i$ and $m_i$ are required to satisfy certain conditions.
Each power $n_i$ is required to be larger than some threshold value $n^*$, which depends on  $\delta$,
and each power $m_i$ is required to be larger than some threshold value $m^*_i$,
which depends on the  history of the pseudo-orbit up to that point, that is, on all previous  powers $n_0,n_1,\ldots,n_{i-1},n_{i}$, $m_{0}, \ldots, m_{i-1}$
that were utilized in the previous segments of the pseudo-orbit from $y_0$ to $y_{i}$.
Intuitively, $m_i,n_i$ quantify
the lengths of time  for which we follow a  homoclinic trajectory associated to the scattering map,  forward, and respectively backwards,
in time, from  $\Gamma$ to a neighborhood of $\Lambda$.

\begin{lem}[Shadowing Lemma for Pseudo-Orbits of   the Scattering Map and  the Inner Dynamics] \label{lem:key1-alternative}

Assume that $f:M\to M$ is a $C^r$-map, $r\geq r_0$, $\Lambda\subseteq M$  is a normally hyperbolic invariant manifold,
$\Gamma\subseteq M$ is a  homoclinic channel, and $\sigma^\Gamma:\Omega^-(\Gamma)\to\Omega^+(\Gamma)$ is the scattering map associated to
$\Gamma$.
Assume that $\Lambda$ and $\Gamma$ are compact.

Then, for every $\delta>0$ there exists $n^*\in\mathbb{N}$ depending on $\delta$, and a family of functions
$m^*_i:\mathbb{N}^{2i+1}\to\mathbb{N}$, $i\geq 0$, depending on $\delta$, such that,
for every pseudo-orbit $\{y_i\}_{i\geq 0}$ in $\Lambda$ of the form
\begin{equation}\label{pseudorbit}
y_{i+1}=f^{m_i}\circ \sigma^{\Gamma}\circ f^{n_i}(y_i),
\end{equation}
for all $i\geq 0$, with $n_i\geq n^*$  and $m_i\geq m^*_i(n_0, \ldots, n_{i-1}, n_i, m_0, \ldots, m_{i-1})$,
there exists an orbit $\{z_i\}_{i\geq 0}$ of $f$ in $M$
such that, for all $i\geq 0$,
\[
z_{i+1}=f^{m_i+n_i}(z_i),\]
and \[d(z_{i},y_{i})<\delta.\]

\end{lem}

The proof of Lemma~\ref{lem:key1-alternative} is given in Subsection \ref{prooflem:key1-alternative}.

Notice that, of course, the functions $n^*$, $m^*_i$ are defined
only after we choose $\delta$, so, they depend on $\delta$.
We emphasize that the sequence $y_i$ in \eqref{pseudorbit} is
contained in $\Lambda$ so that the map $f$ that appears in the
definition of  $y_i$ can be taken to be $f|_\Lambda$.
The reason why we refer to the sequence $\{y_i\}$ in
\eqref{pseudorbit} as a `pseudo-orbit' is that
$y_{i+1}, y_i$ are close to the end points of
a segment orbit of the full map.

Indeed if we consider the point $p_i = (\Omega_-^\Gamma)^{-1} f^{n_i} (y_i)$
we see that
$f^{-n_i}(p_i)$ and $f^{-n_i}( f^{n_i}(y_i))=y_i$ would be close since they
are in the same unstable fiber and $n_i$ is large. We also see
that $\sigma^\Gamma f^{n_i}(y_i) = \Omega_+^\Gamma(p_i)$. Therefore,
$f^{m_i} (p_i)$ and $f^{m_i} \circ \sigma^\Gamma\circ f^{n_i}(y_i)$
will be close since they are in the same stable fiber.

Therefore, the sequence $\{y_i\}$ is approximated by a concatenation of
 segments of orbits $\mathcal{O}_i = \{ f^j(p_i)\}_{j = -n_i}^{m_i}$.
The mismatches at the ends of  these segments of orbits are clearly small.

It would be natural to try to use a hyperbolic
shadowing theorem to close this pseudo-orbit. Unfortunately, with the present
hypothesis, we do not have any information on the expanding or
contracting properties of the map along the directions tangent to
$\Lambda$, and standard hyperbolic shadowing theorems do not seem to apply.
We have to give a different proof
and introduce the condition that the $m_i$'s grow.

The above Lemma \ref{lem:key1-alternative} can be immediately extended to the case of countably many scattering maps.
Suppose that there exists an infinite collection of homoclinic channels    $\Gamma_j\subseteq M$, for $j\in\N$,
and let
$$
\sigma_j:\Omega^{-}(\Gamma_j)\to \Omega^{+}(\Gamma_j)
$$
be the  scattering map associated to $\Gamma_j$, for $j\in \N$.

\begin{lem}\label{lem:key2}
Assume that $f:M\to M$, $\Lambda\subseteq M$, $\Gamma_j\subseteq M$ and   $\sigma_j$, are as above,  for $j\in \mathbb{N}$.
Assume that $\Lambda$ and $\Gamma_j$ are compact.

Then, for every $\delta>0$ there exist  two families of functions,    $n^*_i:\mathbb{N}^{i}\to\mathbb{N}$ and
$m^*_i:\mathbb{N}^{2i+1}\times\mathbb{N}^{i+1}\to\mathbb{N}$, both  depending on $\delta$, for $i\geq 0$, such that,
for every pseudo-orbit $\{y_i\}_{i\geq 0}$ in $\Lambda$ of the form
\[
y_{i+1}=f^{m_i}\circ \sigma_{\alpha_i}\circ f^{n_i}(y_i),
\]
where $n_i\geq n^*(\alpha_0,\ldots,\alpha_{i-1})$, $m_i\geq m^*(n_0, \ldots,  n_i, m_0, \ldots, m_{i-1}, \alpha_0,\ldots,\alpha_{i})$ for all $i\geq 0$,
there exists an orbit $\{z_i\}_{i\geq 0}$ of $f$ in $M$
such that, for all $i\geq 0$,
\[
z_{i+1}=f^{m_i+n_i}(z_i),\] and  \[d(z_{i},y_{i})<\delta.
\]
\end{lem}

\begin{rem}\label{rem:angle}
Even if it is not explicitly written in Lemma \ref{lem:key1-alternative},  $n^*$ and $m_i^*$ also depend
on the hyperbolic structure, and in particular on the angle of intersection between
$W^u(\Lambda)$ with $W^s(\Lambda)$ along $\Gamma$.
\end{rem}

\begin{rem}
Note that Lemma \ref{lem:key1-alternative} does not use any symplectic structure.
It is valid for general maps.
Hence, the results obtained from it remain valid for dissipative perturbations of Hamiltonian systems.
Of course, when the perturbations are Hamiltonian we can obtain stronger results.
\end{rem}

\begin{rem}
We note that   results  related to Lemma \ref{lem:key1-alternative} appear in \cite{GideaR12,DelshamsGR13,GelfreichT2014}.
In comparison to our lemma, \cite{GideaR12,DelshamsGR13}
make some geometric assumptions on the inner dynamics, and \cite{GelfreichT2014}
considers only finite pseudo-orbits.
\end{rem}

\section{Existence of diffusing trajectories in nearly integrable a priori unstable Hamiltonian systems}
\label{subsection:aprioriunstable}
As an application, we show the existence of diffusing orbits in a large class of nearly integrable a priori unstable Hamiltonian systems
that are multi-dimensional both in the center and in the hyperbolic directions.
The model below is an extension of those  considered in \cite{DelshamsLS2006,DelshamsH2009,DelshamsLS16}.

Let
\begin{eqnarray}\label{eqn:hamiltonian}{\textrm{$ $}}\qquad
H_\eps(p, q, I,\phi,t)&=&h_0(I)+\sum_{i=1}^{n}\pm\left(\frac{1}{2}p^2_i+V_i(q_i)\right)+\eps H_1(p, q,I,\phi, t; \eps).
\end{eqnarray}
where $(p, q, I,\phi, t)\in  \mathbb{R}^n\times
\mathbb{T}^n\times\mathbb{R}^{d}\times\mathbb{T}^{d}\times \mathbb{T}^1$.

We make the following assumptions:

\smallskip

\emph{(A1.)
The functions $h_0$, $H_1$  and $V_i$, $i=1,\ldots,n$, are uniformly $C^r$ for  $r\geq r_0$.}

\smallskip

\emph{(A2.)
Each potential $V_i:\mathbb{T}^n\to \mathbb{R}$, $i=1,\ldots,n$, is $1$-periodic in $q_i$ and has a
non-degenerate global maximum at $0$, and hence each `pendulum' $\pm \left(\frac{1}{2}p^2_i+V_i(q_i)\right)$ has a homoclinic orbit to $(0,0)$,
parametrized  by $(p^0_i(t),q^0_i(t))$, $t\in\mathbb{R}$.}
\smallskip

To formulate the next assumption (A3), which has two parts (A3.a) and (A3.b),  we need to introduce some other tools.
\begin{itemize}
\item
Let $\tilde\Lambda_0=\{(p,q,I,\phi,t)\,|\, p=q=0 \}$.
By (A2) there is a family of homoclinic orbits for the whole system of penduli given by
\begin{equation*}
\begin{split}
(p^0(\tau+t\bar 1),q^0(\tau+t\bar 1))=&\left(p^0_1(\tau_1+t), \ldots, p^0_n(\tau_n+t),\right.\\&{}\qquad\left.q^0_1(\tau_1+t),
\ldots, q^0_n(\tau_n+t)\right),
\end{split}
\end{equation*}
where
$\tau=(\tau_1,\ldots,\tau_n)\in\mathbb{R}^n$, and $\bar 1=(1,\ldots,1)\in\mathbb{R}^n$.

\item Let
$\tilde\Gamma_0\subseteq \{(p^0(\tau),q^0(\tau),I,\phi,t)\,|\,\tau\in\mathbb{R}^n,\,  I\in\mathbb{R}^d,\,\phi\in\mathbb{T}^d,\, t\in\mathbb{T}^1\}$
be a homoclinic channel for which we can define a  scattering map  $\tilde\sigma_0$  on $\tilde\Lambda_0$.

\item Let the
Poincar\'e function (or Melnikov potential) associated to
the homoclinic manifold  $\tilde\Gamma_0$ be:
\begin{equation} \label{MelnikovpotG}
\begin{split}
L(\tau, I, \phi, s)  & =- \int_{-\infty}^{\infty}  \big[ H_1(p^0( \tau+t\bar 1),q^0( \tau+ t\bar 1),I, \phi + \omega(I)t,  s+t ;0 )  \\
&  \qquad\qquad -
H_1(0,0,I, \phi + \omega(I)t, s+t  ;0)   \big] \, dt . \\
\end{split}
\end{equation}
where $\omega(I)=\partial{h_0}/\partial{I}$.
\end{itemize}

The first part of the  assumption (A3) is:

\smallskip

\emph{(A3.a)
The perturbation $H_1$ is $1$-periodic in $t$ and satisfies some explicit non-degeneracy conditions as described  below.
Assume that there exists a set $U^-:=\mathcal{I}\times \mathcal {J} \subset \mathbb{R}^{d}\times \mathbb{T}^{d+1}$,
such that $\mathcal{I}$ is a ball in $\mathbb{R}^{d}$, and for any values  $(I,\phi,s)\in U^-$,
the map
\[\tau\in\mathbb{R}^n \to
L(\tau,I, \phi, s) \in\mathbb{R}\]
has a non-degenerate
critical point $\tau^*$, which is locally given,  by the implicit function theorem,
by
\[\tau^*=\tau^*(I,\phi,s).\]}

\smallskip

To formulate the next assumption we need to introduce some other tools.
\begin{itemize}
\item
Let the auxiliary functions
\begin{equation} \label{generator}
\mathcal{L}(I, \phi, s) =  L( \tau^*(I,\phi,s), I, \phi, s), \quad
\mathcal{L}^{*}(I,\theta)= \mathcal{L}(I, \theta, 0).
\end{equation}
We regard $\mathcal{L}^{*}(I,\theta)$ as a function on the set
\begin{equation*}
\textrm{Dom}(\mathcal{L}^{*})=\{ (I,\theta)\in\mathbb{R}^d\times\mathbb{T}^d\,|\,\exists s\in\mathbb{T}^1\textrm { s.t. }(I,\theta+\omega(I)s,s)\in U^-\}.
\end{equation*}
\end{itemize}

The second part of the  assumption (A3) is:

\smallskip

\emph{(A3.b) Assume that the reduced Poincar\'{e} function $\mathcal{L}^{*}(I, \theta)$ satisfies   that  $J\nabla \mathcal{L}^*(I,\theta)$
is transverse, relative to $\mathbb{R}^d\times\mathbb{T}^d$, to the level set $\{I=I_*\}$ at some point
$(I_*,\theta_*)=(I_*, \phi_*-\omega(I_*)s)$, with $(I_*, \phi_*,s)\in U^-$.
That is:
\begin{equation} \label{transversalitycond}
\frac{\partial \mathcal{L}^*}{\partial \theta}(I_*,\theta_*)\ne 0.
\end{equation}}

\smallskip

We note that the integral in \eqref{MelnikovpotG} is similar to that in  \eqref{eq:S-def} and \eqref{eq:S-def2}, as it concerns the average
effect of the perturbation $H_1$ on a homoclinic orbit of the unperturbed system.

The result below states that, for all small enough regular perturbations satisfying \eqref{transversalitycond}, there exist
trajectories that travel $O(1)$ with respect to the $I$-coordinate, that is, they travel a distance relative to the $I$-coordinate
that is independent of the size of the perturbation. This phenomenon is  referred to as Arnold diffusion.

\begin{thm}\label{thm:diffusion}
Assuming the conditions A1-A3, there exists $\eps_0>0$, and $\rho>0$ such that, for each $\eps\in(0,\eps_0)$,
there exists a trajectory $x(t)$ of the Hamiltonian flow of Hamiltonian \eqref{eqn:hamiltonian} and  $T>0$ such that
\[\|I(x(T))-I(x(0))\|>\rho.\]
\end{thm}

\begin{rem}
We emphasize some advantages of Theorem \ref{thm:diffusion} in comparison to  the main results of
\cite{DelshamsLS2006,DelshamsH2009,DelshamsLS16,GideaR12}:
\begin{itemize}
\item
Both the phase space of $h_0$ and that of the system of penduli are multi-dimensional.
\item
We do not assume a convexity condition on the unperturbed Hamiltonian $H_0(I,\phi,p,q)=h_0(I)+\sum_{i=1,\ldots,n}\pm(p_i^2/2+V_i(q_i))$,
which is typically required when using variational methods.
\item
We do not assume that $h_0$ satisfies a non-degeneracy condition that
$I\mapsto\partial{h_0}/\partial{I}$ is a diffeomorphism, or a convexity condition that $\partial^2h_0/\partial I_i\partial I_j$
is strictly positive/negative definite. In the lack of such conditions, one cannot apply the KAM theorem, hence cannot construct transition
chains of KAM tori. Also, Aubry-Mather theory cannot be  applied.
\item
We do not assume that $H_1$ is a trigonometric polynomial.
Moreover, we note that  condition (A3) is satisfied by a $C^r$ open and dense set of perturbations~$H_1$.

In the method of \cite{DelshamsLS2013} one needs to check a
different condition (which is clearly generic)
around every first order resonance. In concrete systems,
when one is interested in a practical problem (e.g., in the three-body problem)
and not in generic statements, the verification of
the mechanism of \cite{DelshamsLS2013} is possible, albeit
tedious.
With the present
method, the verification in concrete systems of interest
is much more straightforward, see e.g.,  \cite{CapinskiGL14}.

\end{itemize}
\end{rem}

From now on, we make the following notation convention.  When we
 say that some error term is bounded by a constant, or by $O(\eps^a)$, or by $O(\eps^a\ln(\eps^b))$
we   mean uniformly on some compact set.

\begin{proof}[Proof of Theorem \ref{thm:diffusion}.]
We describe the  geometric structures that organize the dynamics,  following \cite{DelshamsLS2006,DelshamsLS16}.
We emphasize that, once the geometric set-up is laid out, the dynamics argument to show the existence of diffusing orbits is very different.

The time-dependent Hamiltonian in \eqref{eqn:hamiltonian} is transformed into an autonomous Hamiltonian by introducing a new variable $A$,
symplectically conjugate with $t$, obtaining the $(n+d+1)$ degrees of freedom Hamiltonian system
\begin{equation}\label{eqn:hamiltonianex}
\tilde H_\eps(p, q, I,\phi,A,t)=h_0(I)+\sum_{i=1}^{n}\pm\left(\frac{1}{2}p^2_i+V_i(q_i)\right)+A+\eps H_1(p, q,I,\phi, t; \eps).
\end{equation}
The variable  $A$ does not play any dynamical role, as it does not appear in any of the Hamiltonian equations for any of the variables, including itself.

With an abuse of notation, we denote
\[
{\tilde\Lambda}_0:=\{(p,q,I,\phi,A,t)\,|\, p=q=0, \, I\in\mathcal{I},\, A\in\mathbb{R}, \,(\phi,t) \in \mathbb{T}^{d+1}  \}.
\]
This is a normally hyperbolic invariant manifold for the extended Hamiltonian flow, which is diffeomorphic to
$(\mathbb{R}^d\times\mathbb{T}^d) \times(\mathbb{R}\times\mathbb{T})$.

We  fix an energy manifold $\{\tilde H_\eps=\tilde h\}$ for some $\tilde h$, and restrict to a Poincar\'e section
$\{t=s\}$ for the Hamiltonian flow.
The resulting  manifold is a $(2n+2d)$-dimensional manifold which we denote $M_\epsilon$.
The first return map to $M_\eps$ of the Hamiltonian flow is a $C^{r}$-differentiable map denoted $f_\epsilon$.

 The manifold
\[
\Lambda_0:=\{(p,q,I,\phi)\,|\,  p=q=0 , \, I\in\mathcal{I}, \, \phi \in \mathbb{T}^d  \}\subseteq M_0,
\]
is a normally hyperbolic invariant manifold  for $f_0$, which is independent of the section $\{t=s\}$.
Note that $\Lambda_0$ is diffeomorphic to $\mathbb{R}^d\times\mathbb{T}^d$.

Thus, both $\tilde{\Lambda}_0$ and $\Lambda_0$ are non-compact.

Note that the restriction of $f_0$ to  $\Lambda_0$ is an integrable map, as  $f_0(0,0,I,\phi)=(0,0,I,\phi+\omega(I))$, and $\Lambda_0$ is foliated by
invariant $d$-dimensional tori given by $\{I=ct\}$.

Choose a closed ball $\bar B_{R}(I_*)$ in the action space $\mathbb{R}^d$,     such that  $J\nabla \mathcal{L}^*(I,\theta)$ is transverse,
relative to $\textrm{Dom}(\mathcal{L}^{*})$, to each action level set $\{I=I_0\}$ -- which is an invariant torus --,
 with $I_0\in \bar B_{R}(I_*)$.

Denote \begin{eqnarray*}
\tilde{\Lambda}'_0:&=&\{(p,q,I,\phi,A,t)\,|\, p=q=0, \, I\in \bar B_{R}(I_*), \, A\in\mathbb{R}, \,(\phi,t) \in \mathbb{T}^{d+1}\},\\
{\Lambda}'_0:&=&\{(p,q,I,\phi)\,|\, p=q=0, \, I\in \bar B_{R}(I_*), \, \phi \in \mathbb{T}^d\},
\end{eqnarray*}
which are normally hyperbolic invariant manifolds with boundary for the  flow, and respectively for the map,  corresponding  to $I\in \bar B_{R}(I_*)$.

Consider now the perturbed Hamiltonian system.
Using a $C^r$-differentiable bump function we can modify the Hamiltonian
$\tilde H_\eps$ to another
Hamiltonian $\tilde {\mathcal{ H}}_\eps$ that coincides with the original one  for all $(p,q,I,\phi,A,t)$ with  $I\in  \bar B_{R}(I_*)$,
and coincides with  $H_0$  for all $(p,q,I,\phi,A,t)$ with  $I$ outside of some open ball
$B_{R'}(I_*)\supseteq \bar B_{R}(I_*)$, with $R'>R$.
For all  $\eps$ sufficiently small there exists a normally hyperbolic invariant manifold $ {\tilde{\Lambda}}_\eps$ for the  flow of the modified Hamiltonian
$\tilde {\mathcal{ H}}_\eps$.
The manifold $\tilde{\Lambda}_\eps$ is diffeomorphic to $\tilde{\Lambda}_0$ via  a $C^\ell$-smooth parametrization
$\tilde k_\eps: {\tilde{\Lambda}}_0\to\tilde{\Lambda}_\eps$, with $\tilde k_0=\textrm{Id}$.
Using this parametrization, we can describe $ {\tilde{\Lambda}}_\eps$ in terms of the coordinates $(I,\phi,A,t)\in \tilde{\Lambda}_0$.
Similarly, there exists a $C^\ell$ smooth parametrization $k_\eps:\Lambda_0\to\Lambda_\eps$, with $k_0=\textrm{Id}$.

The manifold ${\tilde\Lambda}_\eps$ is not unique, as it depends on the modificated  Hamiltonian vector field of Hamiltonian
$\tilde {\mathcal{ H}}_\eps$,
but what is important for us is that the extended Hamiltonian
$\tilde {\mathcal{ H}}_\eps$, coincides with $\tilde H_\eps$ at
the  points with $I\in\bar B_{R}(I_*)$.
Therefore, as we will find an orbit of $\tilde {\mathcal{ H}}_\eps$ whose action will  stay in $I\in\bar B_{R}(I_*)$,
this orbit will also be a real orbit of $\tilde H_\eps$.

Let
 \begin{eqnarray*}
 {\tilde\Lambda}'_\eps:&=&\{\tilde k_\eps(p,q,I,\phi,A,t)\,|\, p=q=0, \, I\in \bar B_{R}(I_*), \, A\in\mathbb{R}, \, (\phi,t) \in \mathbb{T}^{d+1}\},\\
{\Lambda}'_\eps:&=&\{k_\eps(p,q,I,\phi)\,|\, p=q=0, \,  I\in \bar B_{R}(I_*), \, \phi \in \mathbb{T}^d\},
\end{eqnarray*}
be the normally hyperbolic manifolds  for the perturbed flow, and respectively for the perturbed map,  corresponding  to
$I\in \bar B_{R}(I_*)$.
They are not invariant, but only locally invariant.
The local invariance means, e.g., in the case  of $\Lambda'_\eps$,   that there exists a neighborhood $\mathcal{V}$ of $\Lambda'_\eps$ in $M_\eps$,
such that any orbit of $f_\eps$ that stays in $\mathcal{V}$ for all time is actually contained in $\Lambda'_\eps$.
The neighborhood $\mathcal{V}$ can be chosen independent of $\eps$.
The manifold $\Lambda'_\eps$ is compact and   symplectic (see~\cite{DelshamsLS08a}).

Condition (A3)   allows one to define a scattering map
$\sigma_\eps:\Omega^{-}(\Gamma_\eps)\to \Omega^{+}(\Gamma_\eps)$, with $\Omega^{-}(\Gamma_\eps),\Omega^{+}(\Gamma_\eps)\subseteq\Lambda'_\eps$.
We will restrict to a homoclinic channel $\Gamma_\eps$ that is compact.

As mentioned before, it is more convenient to express the scattering map $\sigma_\eps$ as a map on $\Lambda_0$, via
$\tilde\sigma_\eps=k_\eps^{-1}\circ \sigma_\eps\circ k_\eps$.
By hypothesis (A.3), we have
$U^- \subseteq \textrm{dom}(\tilde\sigma_\eps)= k_\eps^{-1}(\Omega^{-}(\Gamma_\eps))$.
In a similar fashion, we consider
$\tilde f_\eps=k_\eps^{-1}\circ {f_\eps}_{\mid\Lambda_\eps}\circ k_\eps$ on $\Lambda_0$.

The papers \cite{DelshamsLS08a,DelshamsGLS2008,GideaL06b} show that condition  (A3.a) implies   that the scattering map can be expressed as follows
\[
\tilde\sigma_\eps (I,\phi)=(I,\phi)+\eps J\nabla   \mathcal{L}^*(I,\phi -\omega(I)s) +O(\eps^2),
\]
which is of the form \eqref{eqn:seps} with $\mu(\eps)=\eps$, and $g(\mu(\eps))=\eps ^2$.
 Of course, both the scattering map $\tilde\sigma_\eps$ and the Poincar\'{e} map $\tilde f_\eps$ depend on the chosen section $t=s$.
Therefore we can apply
Theorem \ref{thm:main} for the normally hyperbolic invariant manifold $\Lambda_\eps$ and the scattering map $\tilde\sigma_\eps$.
Since we are actually restricting ourselves to the locally invariant manifold $\Lambda'_\eps\subseteq\Lambda_\eps$,
which is contained in the domain where the modified Hamiltonian $\tilde {\mathcal{ H}}_\eps$ coincides with $\tilde H_\eps$,
the diffusing orbits that we obtain correspond to diffusing orbits of $\tilde H_\eps$.

The function $\mathcal {L}^*$ involved in condition (A3) plays the role of the function $S$ in  Theorem \ref{thm:main}.
Condition \eqref{transversalitycond}  amounts to $J\nabla \mathcal{L}^*$ being  transverse to one level set of the variable $I$,
and hence the results of Theorem \ref{thm:main}, and specially the results of Corollary \ref{cor:main} give us the existence
of a real
orbit which  satisfies the required inequality for a suitable value $\rho>0$ independent of $\eps$.
\end{proof}

\begin{rem} For the  above result, we do not require the non-degeneracy condition that $I\mapsto\omega(I)=\partial{h_0}/\partial{I}$ is a diffeomorphism.
Note that in the case when $d=1$ such a   non-degeneracy condition   implies that $\tilde f_\eps$ is a  monotone twist  map relative to the $(I,\phi)$ coordinates.
In our case, we allow $\tilde f_\eps$ to be a non-twist map, which happens, for instance if $h_0(I)=I^n$ with $n\geq 3$ odd.
It is well known that non-twist maps arise in many concrete models, e.g.,  in magnetic fields of toroidal plasma devices
(such as tokamaks, which have reversed magnetic shear), models   of transport by traveling waves in shear flows with zonal flow,
and models of satellite orbits near critical inclination.
Unlike twist maps, non-twist maps have regions where the KAM theorem and the Aubry-Mather theory do not apply;
see \cite{MorrisonCNG1997,DelshamsL2000} and the references listed therein.
\end{rem}

\section{Proofs of the Main Results}\label{proofs}
\subsection{Proof of Theorem \ref{lem:main1}}\label{prooflem:main1}
Denote by $\mu$ the measure referred in the statement of the theorem,  which is absolutely continuous with respect to the Lebesgue measure on $\Lambda$.
Then $f$ preserves $\mu$, and, $\sigma$
takes positive   measure sets onto positive  measure sets.

Choose a small open disk $B_0$ of $x_0$ in $\Lambda$, with $B_0\subseteq \mathcal{U}$ such that  $B_i:=\sigma^{i}(B_0)\subseteq \mathcal{U}$, and
$\textrm{diam}(B_i)\leq \delta/2$, for all $i=0,\ldots,n$.
For the  given pseudo-orbit $\{x_i\}$ of $\sigma$, with $x_{i+1}=\sigma(x_i)$, we have that $x_i\in B_i$ for all $i$.
We will use Poincar\'e recurrence to produce a new pseudo-orbit $\{y_{i}\}$, with $y_{i+1}=f^{m_i}\circ\sigma\circ f^{n_i}(y_i)$,
where $m_i,n_i$ are as in Lemma \ref{lem:key1-alternative}, such that $y_i\in B_i$ for all $i$, and hence  $d(y_i,x_i)\leq \delta/2$.
Invoking Lemma \ref{lem:key1-alternative} will provide us with a true orbit $\{z_i\}$ with $z_{i+1}=f^{m_i+n_i}(z_i)$,
such that $d(z_i,y_i)\leq \delta/2$, hence $d(z_i,x_i)<\delta$.

We first establish some basic facts about recurrent points.
\subsubsection{First recurrence property.}
For  an open set $B\subseteq \mathcal{U}\subseteq \Lambda$, a subset $A\subseteq B$ of positive measure in $B$,  and $k^*> 0$, we define
\begin{equation*}\begin{split}
P^{k^*}(A,B)&=\{x\in A\,|\, (f^{k^*})^t(x)\in B \textrm{ for  some } t\geq 1\}.
\end{split}\end{equation*}

The set  $P^{k^*}(A,B)\subset A$ consists of the recurrent points of $A$ that return to $B$ under some positive   iteration  of $f^{k^*}$.
Since $\mu$-a.e. point  in $\mathcal{U}$ is recurrent, and  $B\subseteq \mathcal{U}$,
Poincar\'e recurrence  for the map $f^{k^*}$ implies that
$P^{k^*}(A,B)\subseteq A$ has full measure in $A$, hence is of positive measure itself.

For each $x\in P^{k^*}(A,B)$ we denote by $t_{min}(x)$  the smallest positive integer $t\geq 1$ with $(f^{k^*})^t(x)\in B$. Let $$
\Theta=\{\tau\geq 1\,|\, \exists x\in P^{k^*}(A,B) \textrm { s.t. }t_{min}(y)=\tau\}
$$
be the set of the return times to $B$.
For each $\tau\in\Theta$,  let
\begin{equation}\label{eq:P}
P^{k^*}_\tau(A,B)=\{x\in P^{k^*}(A,B)\,|\, t_{min}(y)=\tau\}
\end{equation}
be the set of points with a prescribed return time $\tau\in\Theta$, under $f^{k^*}$.
Since $P_{k^*}(A,B)=\bigcup_{\tau\geq 1} P^{k^*}_\tau(A,B)$, with the sets $P^{k^*}_\tau(A,B)$ mutually disjoint,
it follows that there exists $\tau^*\geq 1$ such that $\mu(P^{k^*}_{\tau^*}(A,B))>0$.
Since $f^{k^*}$ is area preserving, $\mu(f^{k^*\tau^*}(P^{k^*}_{\tau^*}(A,B))=\mu(P^{k^*}_{\tau^*}(A,B))>0$.

Thus, every point in $P^{k^*}_{\tau^*}(A,B)\subseteq A\subseteq B$ will
return to a point in $B$ under $f^{k^*\tau^*}$.
The set
\begin{equation}\label{eq:Q}
Q^{k^*}_{\tau^*}(B,A):=f^{k^*\tau^*}(P^{k^*}_{\tau^*}(A,B))\subseteq B
\end{equation}

has positive measure in $B$.
In terms of $f$, every point in $P^{k^*}_{\tau^*}(A,B)\subseteq A\subseteq B$ will return to a point in $Q^{k^*}_{\tau^*}(A,B)\subseteq B$ in exactly
$k^*\tau^*\geq k^*$ iterates.

\subsubsection{Second recurrence property.}

Consider now two open sets $B\subseteq \mathcal{U}$ and   $B'=\sigma(B)\subseteq \mathcal{U}$.
Let $A$ be a subset of $B$ of positive measure.
By the above, $P^{k^*}_{\tau^*}(A,B)$ and $Q^{k^*}_{\tau^*}(A,B)$ are   positive measure subsets of $B$.
Since the scattering map $\sigma$
sends positive measure sets onto positive measure sets, it follows that
\begin{equation}\label{eq:Aprime}
A':=\sigma(Q^{k^*}_{\tau^*}(A,B)) \subset B'
\end{equation}
is a positive measure subset of $B'$.

\subsubsection{Inductive construction of pseudo-orbits.}

Starting with $B_0$, we construct inductively a nested sequence of   subsets $\Sigma_i\subset B_0$ of positive measure of $B_0$,
such that each set is carried onto a positive measure subset of $B_i$,
$i=1, \ldots, n$, via successive applications of some large powers of  $f$ interspersed with applications of $\sigma$.

Use Lemma \ref{lem:key1-alternative} for $\delta/2$, and consider the value   $n^*$ depending on $\delta/2$ as  provided by this lemma.
Let $A_0:=B_0$, let $\tau_0\geq 1$ such that $P^{n^*}_{\tau_0}(A_0,B_0)\subset A_0$ (see \eqref{eq:P}) has positive measure, and
\[
\Sigma_0:=P^{n^*}_{\tau_0}(A_0,B_0)\subseteq A_0.
\]
Consider the set $Q^{n^*}_{\tau_0}(A_0,B_0)\subseteq B_0$ (see \eqref{eq:Q}), which has positive measure.
Then consider the set $A'_1:=\sigma(Q^{n^*}_{\tau_0}(A_0,B_0))\subseteq B_1$ (see \eqref{eq:Aprime}), which has positive measure in $B_1$.
Let $n_0:=n^*\tau_0$ and consider the value   $m^*_0=m^*_0(n_0)$  given by Lemma \ref{lem:key1-alternative} for $\delta/2$.
There exists $\tau'_0\geq 1$ such that the set
$P^{m^*_0}_{\tau'_0}(A'_1,B_1)\subseteq A'_1\subseteq B_1$ (see \eqref{eq:P}) has positive measure.
Then the set $Q^{m^*_0}_{\tau'_0}(A'_1,B_1)\subseteq B_1$ (see \eqref{eq:Q}) also has positive measure in $B_1$.

Each point $y_1\in  Q^{m^*_0}_{\tau'_0}(A'_1,B_1)$ is of the form
$y_1=f^{m^*_0\tau'_0}(x')$, for some  $x'\in  P^{m^*}_{\tau'_0}(A'_1,B_1)$ and $\tau'_0\geq 1$;
each such $x'$ is of the form $x'=\sigma(x)$ for some $x\in Q^{n^*}_{\tau_0}(A_0,B_0)$; and each such $x$ is of the form $x=f^{n^*\tau_0}(y_0)$ for some
$y_0\in P^{n^*}_{\tau_0}(A_0,B_0)=\Sigma_0$ and $\tau_0\geq 1$.
Denote $m_0:=m^*_0\tau'_0$ and $A_1:=Q^{m^*_0}_{\tau'_0}(A'_1,B_1)\subseteq B_1$.
Thus, each $y_1\in  A_1$ can be written as
\begin{equation}\label{eqn:main01}
y_1=f^{m_0}\circ \sigma\circ f^{n_0}(y_0)
\end{equation}
for some $y_0\in \Sigma_0$, $n_0\geq n^*$ and $m_0\geq m^*$,  where $m_0 =m^*_0  \tau'_0$ and $n_0 =n^*  \tau_0$.

Denote by $\Sigma_1$ the set of points $y_0\in \Sigma_0$ which correspond, via   \eqref{eqn:main01}, to some point $y_1\in A_1$.
We obviously have $\Sigma_1\subseteq \Sigma_0$.
The preliminary facts established above show that  $\Sigma_1$ is a positive measure subset of $B_0$.

Assume that at the $j$-th step we have constructed a subset $A_j\subseteq B_j$, which has positive measure in $B_j$,
such that each point
$y_j\in A_j$
is of the form
\begin{equation}\label{eqn:main02}
y_j=f^{m_{j-1}}\circ \sigma\circ f^{n_{j-1}}\circ \ldots \circ
f^{m_{0}}\circ \sigma\circ f^{n_{0}}(y_{0}),
\end{equation}
some $y_0\in A_0 \subset B_0$, with $n_0\geq n^*, \ldots,n_{j-1}\geq n^*$,  and $m_0\geq m^*_0,\ldots, m_{j-1}\geq m_{j-1}^*$,
where $n^*$ and the $m^*_k$'s are  as in Lemma \ref{lem:key1-alternative}.
Let $\Sigma_j$ be the set of points $y_0$ for which the corresponding $y_j$ given by \eqref{eqn:main02}  is in $A_j$.
We assume that $\Sigma_j\subseteq\Sigma_{j-1}\subseteq \ldots\subseteq \Sigma_0$, and that $\Sigma_j$ is a positive measure subset of $B_0$.

It follows from the above preliminaries that,  for some $\tau_j\geq 1$, the sets
$$
P^{n^*}_{\tau_j}(A_j,B_j)\subseteq B_j\,  \mbox{and} \  Q^{n^*}_{\tau_j}(A_j,B_j)\subseteq B_j=f^{n^*\tau_j}(P^{n^*}_{\tau_j}(A_j,B_j))
$$
have positive measure.
Each point $y\in P^{n^*}_{\tau_j}(A_j,B_j)$ returns to a point in $Q^{n^*}_{\tau_j}(A_j,B_j)\subseteq B_j$ after exactly $n^*\tau_{j}$ iterates of $f$.
Denote $n_j:=n^*\tau_j$.
Since $\sigma$ is measure preserving, the set $A'_{j+1}:=\sigma(Q^{n^*}_{\tau_j}(A_j,B_j))\subset B_{j+1}$ has positive measure  in $B_{j+1}$.
Let $m^*_j$, depending on $\delta/2$ and on $n_0,\ldots, n_{j},m_0,\ldots,m_{j-1}$,  be as  in the Lemma \ref{lem:key1-alternative}.
There exists $\tau'_j\geq 1$ such that $P^{m^*_j}_{\tau'_j}(A'_{j+1},B_{j+1})\subseteq A'_{j+1}\subseteq B_{j+1}$ and
$Q^{m^*_j}_{\tau'_j}(A'_{j+1},B_{j+1})\subseteq B_{j+1}$ have positive measure.
Each point $y\in P^{m^*_j}_{\tau'_j}(A'_{j+1},B_{j+1})$ returns to a point in
$Q^{m^*_j}_{\tau'_j}(A'_{j+1},B_{j+1})\subseteq B_{j+1}$ after exactly $m^*_j\tau'_{j}$ iterates of $f$.
Denote $A_{j+1}=Q^{m^*_j}_{\tau'_j}(A'_{j+1},B_{j+1})$, which is of positive measure.
Then each point $y_{j+1}\in A_{j+1}$ is of the form
\begin{equation}\label{eqn:main03}
y_{j+1}=f^{m_j}\circ \sigma\circ f^{n_j}(y_j)
\end{equation}
for some $y_j\in A_j$, where $n_j=n^* \tau_j\geq n^*$ and $m_j= m^*_j \tau'_j\geq m^*_j$, with $\tau_j,\tau'_j\geq 1$.

Since $y_j$ is of the form \eqref{eqn:main02}, then
\begin{equation}\label{eqn:main03a}
y_{j+1}=f^{m_{j}}\circ \sigma\circ f^{n_{j}}\circ \ldots \circ
f^{m_{0}}\circ \sigma\circ f^{n_{0}}(y_{0}),
\end{equation}
for some $y_0\in \Sigma_0$, with $n_0\geq n^*, \ldots,n_{j-1}\geq n^*$,  and $m_0\geq m^*_0,\ldots, m_{j}\geq m_{j}^*$.
Denoting by $\Sigma_{j+1}$ the  set  of points $y_0\in \Sigma_0$  that yield points $y_{j+1}$  given by \eqref{eqn:main04},
we obtain that $\Sigma_{j+1}\subseteq \Sigma_{j}$ is of positive measure.
 This completes the induction step.

\subsubsection{Shadowing of pseudo-orbits}
At the $n$-th step we obtain a nested sequence of sets $\Sigma_0\supseteq \Sigma_1\supseteq\cdots\supseteq\Sigma_n$, such that each set $\Sigma_j$, $j=0,\ldots,n$,
has positive measure in $B_0$.
Each  point $y_0\in \Sigma _n$ generates  a pseudo-orbit
of the form
\begin{equation}\label{eqn:main04}
y_{j+1}=f^{m_j}\circ \sigma\circ f^{n_j}(y_j),
\end{equation}
for $j=0,\ldots,n-1$, where $n_j,m_j$ are as in Lemma \ref{lem:key1-alternative}.
By construction, each point $y_j$ is inside $B_j$ hence    $d(y_j,x_j)<\delta/2$.
Then Lemma \ref{lem:key1-alternative} provides the existence of an orbit $\{z_j\}_{j=0,\ldots,n}$  with $z_{j+1}=f^{m_{j}+n_{j}}(z_j)$,
such that $d(z_j,y_j)<\delta/2$.  Hence $d(z_j,x_j)<\delta$   for all $j$.
\qed

\begin{rem}
\label{rem:reccurencetime}
In the proof of Theorem \ref{lem:main1},
instead of using the Poincar\'e recurrence theorem we can use the weak recurrence property given by Proposition \ref{prop:weakrecurrence}.
Starting with $B'_0=B_0$, there exists $n_0\geq n^*$ such that\break $f^{n_0}(B'_0)\cap B'_0\neq\emptyset$.
The set $\tilde B'_0=f^{n_0}(B'_0)\cap B'_0$ is an open set in $B'_0$, and $\sigma(\tilde B'_0)\subseteq B_1$.
There exists $m_0\geq m^*_0$ such that $f^{m_0}(\sigma(\tilde B'_0))\cap \sigma(\tilde B'_0)\neq\emptyset$.
The set $B'_1:=\sigma(f^{m_0}(\sigma(\tilde B'_0))\cap \sigma(\tilde B'_0))$ is an open set in $B_1$.
The construction can be continued recursively as before.
Given the open set $B'_{j}\subseteq B_{j}$ obtained at the end of the $(j-1)$-th step, at step $j$ we construct $\tilde B'_{j}=f^{n_{j}}(B'_{j})\cap B'_{j}$
for $n_{j}\geq n^*$,   $\sigma(\tilde B'_{j})\subseteq B_{j+1}$, and $B'_{j+1}:=f^{m_{j}}(\sigma(\tilde B'_{j}))\cap \sigma(\tilde B'_{j})\neq\emptyset$
for $m_{j}\geq m^*_{j}$.
The initial points $y_0\in B_0$ which generate   pseudo-orbits of the form \eqref{eqn:main04}, for $j=0,\ldots, n-1$, form an open set $\Sigma_n\subseteq B_0$.

This approach yields explicit estimates of the return times to $B'_{j}$ and $\sigma(\tilde B'_{j})$, given by  $O(1/\mu(B'_{j}))$ and
$O(1/\mu(\sigma(\tilde B'_{j})))$, respectively.
These estimates on the return time, together with the data on the  hyperbolic expansion/contraction rates and on the  angle of intersection between
the stable and unstable manifolds (see Remark \ref{rem:angle}), can be used to obtain explicit -- but far from optimal -- estimates on the diffusion time.
\end{rem}

\subsection{Proof of Theorem \ref{thm:main}}\label{proofthm:main}

We notice that \eqref{eqn:seps} is reminiscent of the forward Euler method with step $\mu(\eps)$ for ordinary differential equations.

As  $J\nabla S(\tilde x_0)\neq 0$ at some point $\tilde x_0\in U\subset \Lambda_0$, we know that the solution
\begin{equation}\label{eq:flowscattering}
\frac{d}{dt} \tilde\gamma(t)=J\nabla S\circ \tilde\gamma(t)
\end{equation}
with $\tilde\gamma(0)=\tilde x_0$ is not a constant solution.
Let's denote $\tilde\gamma(t)=\phi(t,\tilde x_0)$ where $\phi(t,x)$ is the flow of \eqref{eq:flowscattering}.
Consider $n=\lfloor{\mu}^{-1}\rfloor$,
where $\mu =\mu (\eps)$ is the parameter which appears in \eqref{eqn:seps}, and  $\lfloor \cdot\rfloor$ denotes the floor function.
Define two sequences:
\[
\tilde y_i=\tilde\gamma(t_i)=\phi(\Delta t,\tilde y_{i-1}),\quad \tilde x_i=\tilde \sigma_{\eps}(\tilde x_{i-1}), \quad i=1,2,\dots ,n, \quad
\tilde x_0=\tilde y_0,
\]
where $t_i=i\mu$ and $\Delta t=\mu$.
We will use two facts.

On one hand, applying Gronwall Lemma to the vector field \eqref{eq:flowscattering}, there exists a constant $K_1>0$ such that:
\begin{equation}\label{lipchitzflow}
\|\phi(\Delta t,\tilde y)-\phi(\Delta t, \tilde  y')\|\le \mathrm{e}^{K_1\mu }\|\tilde y-\tilde y'\|, \quad \mbox{for } \tilde y,
\tilde y'\in \mathcal{U}_{\tilde\gamma}.
\end{equation}

On the other hand, also by \eqref{eqn:seps}, calling
$$
\tilde g(\mu)=|g(\mu)|/\mu=o(1),
$$
there  exists a constant $K_2>0$ which is independent of $\mu, \eps$ such that
\begin{equation}\label{bounds}
\|\tilde \sigma_{\eps}(\tilde x)-\phi(\Delta t, \tilde x)\|\le K_2 \mu (\mu +\tilde g(\mu)), \quad  \mbox{for} \quad \tilde x\in \mathcal{U}_{\tilde\gamma}.
\end{equation}
Now one easily obtains that, by \eqref{bounds},
$$
\|\tilde x_1-\tilde y_1\|=\|\tilde \sigma_{\eps}(\tilde x_0)-\phi(\mu,\tilde x_0)\|\le K_2\mu (\mu +\tilde g(\mu))
$$
and, consequently,
$\tilde x_1\in  \mathcal{U}_{\tilde\gamma}$.

Now, using again \eqref{lipchitzflow} and \eqref{bounds}
\[
\begin{array}{rcl}
\|\tilde x_2-\tilde y_2\|&=&
\|\tilde \sigma_{\eps}(\tilde x_1)-\phi(\mu,\tilde y_1)\|\le \|\tilde \sigma_{\eps}(\tilde x_1)-\phi(\mu, \tilde x_1)\|
+\|\phi(\mu,\tilde x_1)-\phi(\mu,\tilde y_1)\|\\
&\le& K_2\mu (\mu +\tilde g(\mu))+\mathrm{e}^{K_1\mu}\|\tilde x_1-\tilde y_1\| \le K_2\mu (\mu +\tilde g(\mu))(1+c).
\end{array}
\]
where we denote by $c=\mathrm{e}^{K_1 \mu} >1$.

Consequently,
 $\tilde x_2\in  \mathcal{U}_{\tilde\gamma}$.
 Now we proceed by induction.
 We assume that, for some $0\le i\le n$, one has that
 \[
 \|\tilde x_i-\tilde y_i\|\le K_2\mu (\mu +\tilde g(\mu)) (1+c+c^2+\dots +c^{i-1}).
 \]

Using again  \eqref{lipchitzflow} and \eqref{bounds} we obtain:
\[
\begin{array}{rcl}
\|\tilde x_{i+1}-\tilde y_{i+1}\|&=&
\|\tilde \sigma_{\eps}(\tilde x_i)-\phi(\mu,\tilde y_i)\|\le \|\tilde \sigma_{\eps}(\tilde x_i)-\phi(\mu,\tilde x_i)\|+
\|\phi(\mu,\tilde x_i)-\phi(\mu,\tilde y_i)\|\\
&\le& K_2\mu (\mu +\tilde g(\mu))+ \mathrm{e}^{K_1\mu} \|\tilde x_i-\tilde y_i\|\\ &\le& K_2\mu (\mu +\tilde g(\mu))(1+c+c^2+\dots +c^i).
\end{array}
\]
Therefore, using that $c=\mathrm{e}^{K_1\mu}$, that $c-1=\mathrm{e}^{K_1\mu}-1\ge K_1 \mu$, and that $n=\lfloor{\mu}^{-1}\rfloor$, for $i=0,1,\dots ,n$,
we have that:
\begin{equation}\begin{split}
\|\tilde x_i-\tilde y_i\|&\le K_2\mu (\mu +\tilde g(\mu))\frac{c^i-1}{c-1}\le \frac{K_2}{K_1} (\mu +\tilde g(\mu)) \mathrm{e}^{i\,K_1\mu}
\le \frac{K_2}{K_1} (\mu +\tilde g(\mu))\mathrm{e}^{K_1}.
\end{split}\end{equation}

As  $\mu =\mu(\eps)= o(\eps)$,  there exists $\eps_1$, such that if $0< \eps \le \eps_1$,
we obtain that the sequence $\tilde x_i$ of the scattering map is also in
$\mathcal{U}_{\tilde\gamma}$ and is $(\mu  +\tilde g(\mu))$-close to the orbit $\tilde\gamma$:
\[
\tilde x_{i+1}=\tilde \sigma_\eps(\tilde x_i)\in\mathcal{U}_{\tilde\gamma}\subset \Lambda,
\quad d(\tilde\gamma(t_i),\tilde x_i)<K  (\mu (\eps) +\tilde g(\mu(\eps))), \quad i=0,\ldots, n,
\]
where $\tilde K=\frac{K_2}{K_1}\mathrm{e}^{K_1}$, and $n=\lfloor{\mu}^{-1}\rfloor$ depends on $\eps$,
for the increasing sequence of parameters $t_i=i\mu\in[0,1]$, $i=0,\ldots,n$.
The points $\tilde x_i$ represent an orbit of $\tilde \sigma_\eps$ in $\Lambda$, therefore  the points
$x_i=k_\eps(\tilde x_i)$, represent an orbit of $\sigma_\eps$ in $\Lambda_\eps$, satisfying
$d( x_i,  \gamma_\eps(t_i))<  K (\mu (\eps) +\tilde g(\mu(\eps)))$, where $\gamma_\eps= k_\eps \circ \tilde\gamma$
and $K$ is a new constant.
This orbit $\tilde x_i$ lies inside the set $\mathcal{U}_{\gamma_\eps}= k_\eps(\mathcal{U}_{\tilde\gamma})
\subseteq\Lambda_\eps$,
where a.e. point is recurrent for $(f_\eps)_{\mid\Lambda_\eps}$.
See Figure~\ref{fig:scattering_path}.
\begin{figure}
\centering
\includegraphics[width=0.5\textwidth]{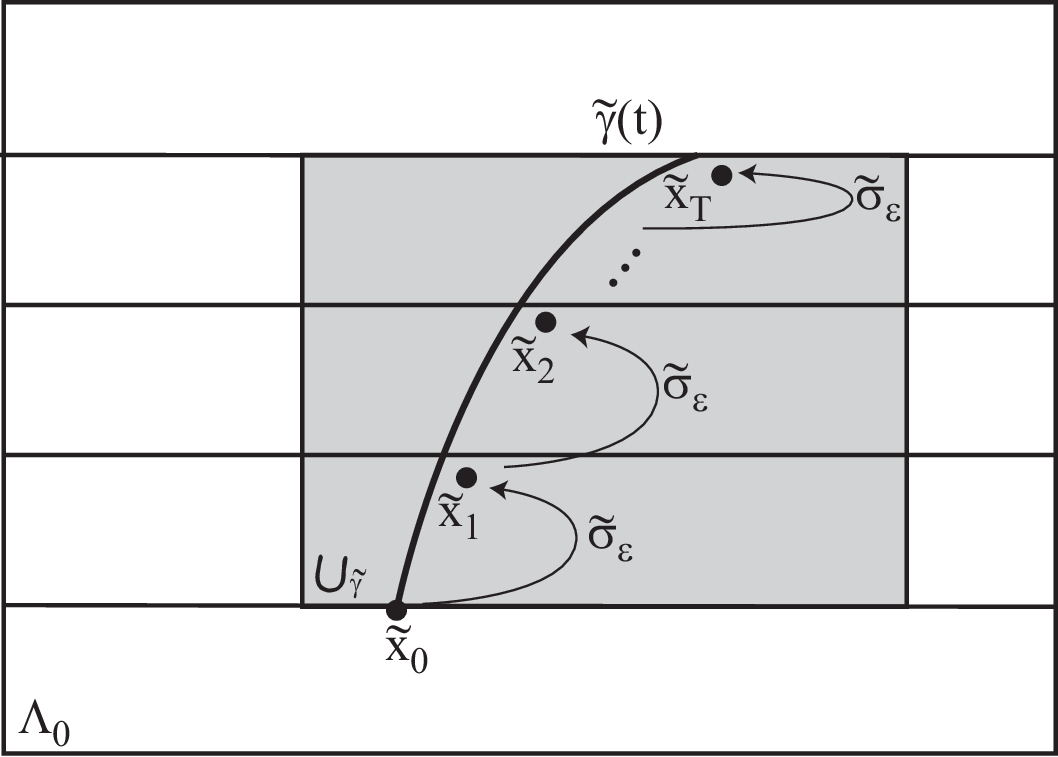}
\caption{A scattering path and a nearby orbit of the scattering map.}
\label{fig:scattering_path}
\end{figure}

We now apply Theorem  \ref{lem:main1} for the orbit $(x_i)_{i=0,\ldots,n}$  of the scattering map
$\sigma_\eps$ on $\Lambda_\eps$,  where $x_i=k_\eps(\tilde x_i)$, and we obtain that, for any  $\delta>0$ there exists
an orbit $z_{i+1}=f^{k_i}_\eps( z_i)$ of $ f_\eps$, which satisfies $d( z_i, x_i)< \delta$, $i=0,\ldots,n$.
Therefore we obtain that
\[
d( z_i, \gamma_\eps(t_i))< \delta  + K (\mu (\eps) +\tilde g(\mu(\eps))) .
\]
\qed

\subsection{Proof of Corollary  \ref{cor:main}}\label{proofcor:main}
By continuity, since $J\nabla S$  is transverse to one level set of the variable $I$ in $\Lambda$,  it is
transverse to a $O(1)$-family of level sets of the variable $I$.
More precisely, there exist  two compact disks $D^d\subseteq B^d$, $E^d\subseteq \mathbb{T}^d$, of radii independent of $\eps$,
such that $J\nabla S$ is transverse to each level set $\{I=I_a\}$ at $\tilde{\sigma}_0(I_a,\phi_a)$ for $I_a\in D^d$, $\phi_a\in E^d$.

Let $\Delta= D^d\times E^d$ and let
\[\Delta ^\infty =\bigcup_{n\geq 0}\tilde{f}^n_\eps (\Delta).\]
Note that $\Delta\subseteq \Delta ^\infty$ and  that $\Delta^\infty$ is positively invariant, i.e., $\tilde{f}_\eps(\Delta ^\infty)\subseteq \Delta ^\infty$.

We have the following dichotomy:

\begin{itemize}
\item[I.]   Either $\mu(\Delta^\infty)=\infty$,
\item[II.]  Or $\mu(\Delta^\infty)<\infty$.
\end{itemize}

Case I implies right away that for every $N>0$, there there exists an orbit $(\tilde{f}^n_\eps(\tilde x))_{n\geq 0}$ of $\tilde{f}_\eps$ in $\Lambda$ for which
$\|I(\tilde{f}^{k_N}(\tilde x))-I(\tilde x)\|>N$, for some $k_N \ge 0$.
It follows immediately that there exist orbits of $f_\eps$ as in the statement of the corollary.
Notice that in this case we obtain diffusing orbits only by applying the inner dynamics; we do not have to use the scattering map.

Now we consider Case II.
Since $\mu(\Delta^\infty)<\infty$ we can apply  the Poincar\'e Recurrence Theorem, so for every open set $\mathcal{U}\subseteq \Delta$,
almost every point of $\mathcal{U}$ is recurrent.

By the assumption on the scattering map, we have that for each $(I_0,\phi_0)\in D^d\times E^d$, the curve $\tilde\gamma(t)$, $t\in[0,1]$,
obtained by integrating the vector field $J\nabla S$ with initial condition at $(I_0,\phi_0)$ is  transverse to every level set $\{I=I_a\}$ at a point
$\tilde\gamma(t)=(I(t),\phi (t))
$,
where
$(I(t),\phi(t)) \in D^d\times E^d=\Delta$, for all $t\in[0,1]$ and  all $0<\eps<\eps_1$.
Thus, there exists $\rho_0>0$, independent of $\eps$, such that
\[\|I(\tilde\gamma(1))-I(\tilde\gamma(0))\|>\rho_0.\]

Choose an $\eps_1$ as in Theorem \ref{thm:main} and fix an $\eps\in(0,\eps_1)$.
Choose $0<\delta<\rho_0/4$, and restrict $\eps _1$ if necessary in such a way that
$K (\mu (\eps) +|g(\mu(\eps))|/\mu(\eps)) \le \delta$ and
let $\rho=\rho_0-4\delta>0$.
Theorem \ref{thm:main}
implies that there is an orbit $(z_i)_{i=0,\ldots,n}$ of $f_\eps$ such that
$d(z_0,\gamma_\eps(0))<2\delta $ and $d(z_n,\gamma_\eps(1))<2\delta$.
Thus, we have $\|I(z_n)-I(z_0)\|>\rho_0 -4\delta =\rho$.
\qed

\subsection{Proof of Lemma~\ref{lem:key1-alternative}}\label{prooflem:key1-alternative}

We provide two proofs of Lemma~\ref{lem:key1-alternative}.

The first proof uses the topological method of correctly
aligned windows  (see Appendix  \ref{section:windows}), and is constructive, in the sense that it provides an explicit algorithm to detect
orbits with prescribed itineraries. It can also be used to provide quantitative estimates for the diffusion time (see Remark \ref{rem:angle} and Remark \ref{rem:reccurencetime}).

The second proof uses the obstruction argument, and is shorter.

\subsubsection{A  proof of Lemma~\ref{lem:key1-alternative} using correctly aligned windows}
\label{proofshadowinglemma}
{$ $}
\smallskip

\paragraph{\em Outline.}\label{outline}
We will construct windows that are correctly aligned, and utilize them in two different ways:
first, to define the integers $n^*$ and $m^*_i$ that appear in the statement of the lemma, and second,  to show that,
for a given  pseudo-orbit as in the statement of the  lemma, there exists a true orbit that shadows it.

For the first part, starting with a homoclinic point, we define a pair of `prototype' windows that are correctly aligned, with one window in a neighborhood
of some negative iterate of the homoclinic point, and another window in a neighborhood of some positive iterate of the homoclinic point.
There are conditions on the number of such iterates which provide us with the integer $n^*$.
Then we consider a second homoclinic point and we construct a second pair of `prototype' windows, in a similar fashion.
To make the second window from the first pair  correctly aligned with the first window from the second pair,
we need to apply a sufficiently large number of iterates that is no less than some integer $m^*$.
When this construction is repeated $i$ times, it provides us with an integer $m^*_i$ that depends on all previous windows.

For the second part, there is given a pseudo-orbit generated by alternatively applying the scattering map and the inner dynamics;
the orders of the iterates of the inner map are required to satisfy conditions that depend on the integers $n^*$ and $m^*_i$.
Then, the above mentioned windows can be used to construct a sequence of correctly aligned windows along the pseudo-orbit.
The existence of an orbit that follows these windows, and, in particular, shadows the given pseudo-orbit, follows from the shadowing property of
correctly aligned windows (Theorem \ref{theorem:detorb}).

\smallskip

We proceed in several steps.

\smallskip

\paragraph{\em Choice of balls.}\label{section:balls}
We choose a system of linearized coordinates (see Appendix \ref{sec:linearized_coordinates}), given by $h:U_\Lambda\to V_\Lambda$,
so that $V_\Lambda$ is contained in a $\delta$-neighborhood of~$\Lambda$.

By the compactness of $\Lambda$, there exists $\delta_1>0$  such that, whenever $x_c\in\Lambda$, $\|v_u\|, \|v_s\|<\delta_1$,  and $0<\rho_c,\rho_u,\rho_s<\delta_1$, the image    of
\begin{equation}\label{eqn:balls}
B_{\rho_c}(x_c)\times B_{\rho_u}(v_u)\times B_{\rho_s}(v_s)\subseteq U_\Lambda
\end{equation}
under $h$  is contained in $V_{\Lambda}$, and has  diameter less than $\delta/2$.

We choose and fix $\rho_c,\rho_u,\rho_s$ as in \eqref{eqn:balls}.


By the normal hyperbolicity of $\Lambda$, there exist
$0<\lambda_-<\lambda_+<\mu^{-1}_+<\mu^{-1}_-<1$ such that
for each pair of balls $B_{\rho_u}(v_u)\subseteq E^u_x$ and $B_{\rho_s}(v_s)\subseteq E^s_x$, with $x\in\Lambda$, we have
\begin{equation}\label{eqn:expcontr}
\begin{split}
B_{\rho_u\mu_-}(Df_{\mid E^u}(v_u)) \subseteq Df_{\mid E^u}(B_{\rho_u}(v_u))\subseteq B_{\rho_u\mu_+}(Df_{\mid E^u}(v_u)),\\
 B_{\rho_s\lambda_-}(Df_{\mid E^s}(v_s)) \subseteq Df_{\mid E^s}(B_{\rho_s}(v_s))\subseteq B_{\rho_s\lambda_+}
 (Df_{\mid E^s}(v_s)).
\end{split}
\end{equation}

\smallskip

\paragraph{\em Lambda Lemma.}\label{sec:lambdalemma}
Consider a homoclinic channel  $\Gamma$ and the corresponding scattering map $\sigma:=\sigma^{\Gamma}:\Omega^{-}(\Gamma)\to \Omega^{+}(\Gamma)$.

Let $p\in\Gamma$ and let $p^-,p^+\in \Lambda$ be the unique points for which $W^u(p^-)\cap W^s(p^+)\cap\Gamma=\{p\}$.
For given $k^-,k^+$, denote:
\begin{eqnarray*}
f^{-k^-}(p^-)&=&h(x_c^-,0,0), \
f^{-k^-}(p)=h(x_c^-,v_u^-,0), \\
f^{k^+}(p^+)&=&h(x_c^+,0,0), \ f^{k^+}(p)=h(x_c^+,0,v_s^+).
\end{eqnarray*}

Due to the compactness  of $\Gamma$ and the exponential contraction of the unstable (stable) fibers under negative (positive) iterates of $f$, there exists $n^*>0$ sufficiently large such that for every $k^-\geq n^*$, $k^+\geq n^*$ we have
\begin{itemize}
\item [(i)]
The point  $f^{-k^-}(p)\in W^u(f^{-k^-}(p^-))$  satisfies $\|v_u^-\|<\rho_u<\delta_1$.
This implies that $f^{-k^-}(p)\in V_\Lambda$ and is $(\delta/2)$-close to $f^{-k^-}(p^-)$;

\item [(ii)]
The point $f^{k^+}(p)\in W^s(f^{k^+}(p^+))$ satisfies $\|v_s^+\|<\rho_s<\delta_1$.
This implies that   $f^{k^+}(p)\in V_\Lambda$  and is $(\delta/2)$-close to $f^{k^+}(p^+)$.
\end{itemize}

Since $\Gamma$ is an homoclinic channel,
$W^s(p^+)$ is transverse to $W^u(\Lambda)$ at $p$, and
$W^u(p^-)$ is transverse to $W^s(\Lambda)$ at $p$.

We apply two versions of the  Lambda Lemma \cite{HirschPalisPughShub1969,FontichM2000,CressonG08,CressonW15,Sabbagh14},
and derive two transversality properties.
The first  version is concerned with the asymptotic behavior of the  backwards  iterates of an $(n_s)$-dimensional manifold transverse to $W^u(\Lambda)$.
The second version is concerned  with   the asymptotic behavior of the  forward  iterates of an $(n_u+n_c)$-dimensional manifold transverse to
$W^s(x)$ where $x\in \Lambda$.

\smallskip

\subparagraph{\em First application of the Lambda Lemma.}
First we apply the  Lambda Lemma to the  $(n_s)$-dimensional manifold $W^s(p^+)$ passing through the point  $p$.
There exists a family of $(n_s)$-dimensional compact disks
$$
\mathscr{D}^s_{k^-}(p)\subset W^s(p^+)
$$
centered  at $p$, such that
$f^{-k^-}(\mathscr{D}^s_{k^-}(p))$ approaches,  in the $C^1$-topology as $k^-\to\infty$,  a disk of fixed radius  in  $W^s(f^{-k^-}(p^-))$
and centered at $f^{-k^-}(p^-)$.
Denote
\begin{equation}\label{Dsk-}
D^s(f^{-k^-}(p)):=f^{-k^-}(\mathscr{D}^s_{k^-}(p)) \subset W^s(f^{-k^-}(p^+))
\end{equation}
the disk centered at $f^{-k^-}(p)$ which is asymptotic to a moving disk of fixed size in $W^s(f^{-k^-}(p^-))$.

Choose $k^-$ large enough and $\mathscr{D}^s_{k^-}(p)$ sufficiently small so that $D^s(f^{-k^-}(p))$  is contained
in $h(B_{\rho_c}(x_c^-)\times B_{\rho_u}(0)\times B_{\rho_s}(0))$ and is $\delta/2$-close to  $W^s(f^{-k^-}(p^-))$ in the $C^1$-topology.

Due to the compactness of $\Lambda$ and $\Gamma$, the size of the disk $D^s(f^{-k^-}(p))$ can be chosen independently of $p\in\Gamma$ and of $k^-$.

\smallskip

\subparagraph{\em First transversality property.}
Since $\mathscr{D}^s_{k^-}(p)$ is transverse to $W^u(\Lambda)$ at $p$:
\begin{equation}
\label{eqn:trans0}
D^s(f^{-k^-}(p))\textrm{ is transverse to }W^u(\Lambda) \textrm{ at }f^{-k^-}(p),\textrm { for any $k^-\ge n^*$}.
\end{equation}

\smallskip

\subparagraph{\em Second application of the Lambda Lemma.}
We now apply the Lambda Lemma to the $(n_c+n_u)$ dimensional manifold $W^u(\Lambda)$ at the point $f^{-k^-}(p)$, which is
is transverse to $W^s(f^{-k^-}(p^+))$ at $f^{-k^-}(p)$.
In particular it  is transverse to $D^s(f^{-k^-}(p))$.

There exists a family of $(n_c+n_u)$-dimensional disks

$$
\mathscr{D}^{cu}_{k^-,k^+}(f^{-k^-}(p))\subseteq W^u(\Lambda),
$$
centered at $f^{-k^-}(p)$, with each disk being a  neighborhood of $f^{-k^-}(p)$ in
$W^u(\Lambda)$,
such that each
$f^{k^++k^-}(\mathscr{D}^{cu}_{k^-,k^+}(f^{-k^-}(p)))$
approaches, in the $C^1$-topology as $k^+\to\infty$,  a disk  of fixed size in the unstable directions,
contained in $W^u(\Lambda)$ and centered at $f^{k+}(p^+)$, as $k^+\to\infty$.

If we choose  $k^-$ and $k^+$ large enough and fixed,
for every disk
\begin{equation}\label{Dcuk-}
D^{cu}(f^{-k^-}(p)):=h(B_{\rho_c^-}(x_c^-)\times B_{\rho_u^-}(v_u^-)\times \{0\}) \subseteq \mathscr{D}^{cu}_{k^-,k^+}(f^{-k^-}(p))
\end{equation}
with $\rho_c^->0$, $\rho_u^->0$ small enough,  we have that
$f^{k^++k^-}(D^{cu}(f^{-k^-}(p)))$ is $\delta/2$-close, in the $C^1$-topology,
to some disk of the form $h(B_{\rho_c^+}(x_c^+))\times B_{\rho_u^+}(0)\times\{0\})$ contained in $W^u(\Lambda)$,
for some $\rho_u^+>0,\rho_c^+>0$.
Denote by:
\begin{equation}\label{Dcuk+}
D^{cu}(f^{k^+}(p)):=f^{k^++k^-}(D^{cu}(f^{-k^-}(p))).
\end{equation}
We have that $\rho_u^+>0$ depends on $\rho_u^-,\rho_c^-$, but is independent of $k^-$ and $k^+$ provided they are large enough,
and $\rho_c^+>0$ depends on  $\rho_u^-,\rho_c^-, k^-, k^+$.
For $k^-,k^+$ fixed, the smaller $\rho_u^-,\rho_c^-$, the smaller $\rho_u^+>0,\rho_c^+>0$.


Here we should note that, while $D^{cu}(f^{-k^-}(p))$ is defined via the $h$-coordinates which are only $C^0$,
it is in fact contained in $W^u(\Lambda)$, so it is embedded in a  $C^1$-disk.
Hence we can measure   its distance away from $W^u(\Lambda)$ in terms of the  $C^1$-topology.
Also, note that   $D^{cu}(f^{k^+}(p))\pitchfork W^s(f^{k^+}(p^+))$. We derive the following:

\smallskip

\subparagraph{\em  Second transversality property.}
For $k^+$  sufficiently large and fixed, there exist   $\rho_u^+>0$ and $\rho_c^+>0$  such that for each $z_c^+\in B_{\rho_c^+}(x_c^+)$, $w_u^+\in B_{\rho_u^+}(0)$,
\begin{equation}\label{eqn:toptrans}
D^{cu}(f^{k^+}(p))\textrm { is topologically transverse to }
h(\{z_c^+\}\times \{w_u^+\}\times B_{\rho^s}(0)),
\end{equation}
where $\rho^s$ is defined in \eqref{eqn:balls}.

See \cite{GideaR03} for a definition of topological transversality.
Since the linearized coordinates $h$ are $C^0$,  the $n_s$-disks
$h\left(\{z_c^+\}\times \{w_u^+\}\times B_{\rho^s}(0)\right)$ in \eqref{eqn:toptrans}  are only $C^0$.
This is why we have to use the notion of topological transversality rather than the differentiable one.
Property \eqref{eqn:toptrans} holds true for the following reasons.
The $n_s$-disks $h\left(\{z_c^+\}\times \{w_u^+\}\times B_{\rho^s}(0)\right)$ depend in a $C^0$-fashion on $z_c^+$ and $w_u^+$.
For $z_c^+=x_c^+$ and $w_u^+=0$ the corresponding $n_s$-dimensional disk is a part of the stable fiber $W^s(f^{k^+}(p^+))$,
which is differentiably transverse to the $(n_c+n_u)$-dimensional disk $D^{cu}(f^{k^+}(p))$.
Differentiable transversality implies topological transversality, and topological transversality is $C^0$-stable.

Property \eqref{eqn:toptrans} implies that
\begin{equation}\label{eqn:projection_cu}
\pi_{c,u}(h^{-1}(D^{cu}(f^{k^+}(p)))\supseteq B_{{\rho}_c^+}( x_c^+ )\times B_{ {\rho}_u^+}(0),
\end{equation}
where $\pi_{c,u}$ is the projection
onto the $(c,u)$-subspace of $(E^u\oplus E^s)_{\Lambda}$
relative to the $h$-coordinate system.

Due to the compactness of $\Lambda$ and $\Gamma$,  $\rho^+_u$  can be chosen independently of $p\in\Gamma$ and of $k^-,k^+$,
provided they are large enough, but will depend on $\rho_u^-,\rho_c^-$, whereas $\rho^+_c$  can be chosen independently of $p\in\Gamma$,
but will depend on $\rho_u^-,\rho_c^-, k^-,k^+$.

\smallskip

\paragraph{\em Choice of $n^*$.}\label{sec:n^*}
Fix $\delta >0$, and let $n^*>0$ sufficiently large so that the conditions in Section \ref{sec:lambdalemma} hold.
We impose additional conditions on~$n^*$.

Since $\Gamma$ is compact we can choose $n^*>0$ such that for every $k^-\geq n^*$
and every $p\in\Gamma$, the $n_s$-dimensional compact disk $D^s(f^{-k^-}(p))$ given in \eqref{Dsk-}
always satisfies the transversality condition  \eqref{eqn:trans0}.
In other words, $k^-$ can be chosen uniformly with respect to $p\in \Gamma$.
This $n^*$ is the number that appears in the statement of Lemma  \ref{lem:key1-alternative}.

Fix such an $n^*$ depending on $\delta$, and which is independent of $p\in\Gamma$.

For a fixed choice of $p\in\Gamma$ and of $k^->n^*$, let  $D^{cu}(f^{-k^-}(p))$ be the disk attached to $f^{-k^-}(p)$ described in \eqref{Dcuk-},
for some $\rho_c^->0$, $\rho_u^->0$.
For every $k^+\geq n^*$, the $(k^-+k^+)$-th iterate of $D^{cu}(f^{-k^-}(p))$,
denoted by  $D^{cu}(f^{k^+}(p))$ in \eqref{Dcuk+}, satisfies \eqref{eqn:toptrans} and \eqref{eqn:projection_cu} for some
$\rho_u^+$, $\rho_c^+$.
The power $k^+$ can be chosen uniformly with respect to $p\in\Gamma$, and for  $k^+$ fixed, the parameters $\rho_u^+$, and $\rho_c^+$
depend on  $\rho_c^-$ and $\rho_u^-$.

It is also important to note that $k^+$, $\rho_u^+$, $\rho_c^+$ also depend  on the  angle of the intersection between
$W^u(\Lambda)$ and $W^s(\Lambda)$ at $p\in \Gamma$.
When the angle of intersection is small, the radii
$\rho_u^+$, $\rho_c^+$  need to  be chosen sufficiently small.
However, our argument   is only qualitative, and
making quantitative estimates   on the dependence of this product of disks  on the angle of intersection  is beyond the purpose of this paper.
Since $\Gamma$ is compact,  there exists a positive lower bound for the angle of intersection, and thus we can make the choices of $k^-,k^+$
uniform  for all points $p\in\Gamma$.

\smallskip

\paragraph{\em  Prototype windows}
\label{sec:prototype}
For $\delta>0$ fixed, choose and fix $n^*>0$ as in Section \ref{sec:lambdalemma}.
Consider a  point  $p$   in the homoclinic channel $\Gamma$.
For fixed $k^-, k^+\geq n^*$ consider a  pair of disks: the $n_s$-dimensional  $D^{s}(f^{-k^-}(p))$ as in \eqref{Dsk-} and the $(n_c+n_u)$-dimensional
$D^{cu}(f^{k^+}(p))$ as in \eqref{Dcuk+}.

We make the following claim:

\smallskip

\subparagraph{\em Claim on $m^*$.}
There exists $m^*\geq 0$ depending on the size of the disks $D^{s}(f^{-k^-}(p))$ and $D^{cu}(f^{k^+}(p))$,
such that for every $m\geq m^*$, and every $k'^-\geq n^*$,  if $p'\in\Gamma$ is such that
\[
p'^-=f^{k'^-+m}(p^+),
\]
then there exists a triplet of windows $W^-$, $W^+$, $W'^-$ with the following properties:
\begin{itemize}
    \item
    $W^-$ is contained in a $\delta/2$-neighborhood of $f^{-k^-}(p)$
    and therefore in a $\delta$-neighborhood of $f^{-k^-}(p^-)$;
    \item
    $W^+$ is contained in a $\delta$-neighborhood of $f^{k^+}(p^+)$;
    \item
    $ W'^-$ is contained in a $\delta/2$-neighborhood of $f^{-k'^-}(p')$
    and therefore in a $\delta$-neighborhood of $f^{-k'^-}(p'^-)$;
\item
$W^-$ is correctly aligned with $W^+$ under $f^{k^-+k^+}$;
\item
$W^+$ is correctly aligned with $W'^-$ under $f^{m-k^+}$;
\item
the sizes of the windows $W^-$, $W^+$, $W'^-$ do  not depend on  the points $p,p'\in\Gamma$;
the size of $W^+$ depends only on the size of $W^-$ and on $k^-,k^+$;  and  the size of $W'^-$ depends only on the size of $W^+$ and on $m$ and $k^+$.
\end{itemize}

In the above, $p^-$, $p^+$ satisfy $W^u(p^-)\cap W^s(p^+)\cap \Gamma=\{p\}$, and
$p'^-$, $p'^+$ satisfy $W^u(p'^-)\cap W^s(p'^+)\cap \Gamma=\{p'\}$.
We will refer to $W^-$, $W^+$, $W'^-$ as prototype windows as we will use
them in the next section to construct an infinite sequence of correctly aligned windows,
as described in the outline.

\smallskip

\subparagraph{\em  Construction of $W^-$.}
First we construct the window $W^-$  about $f^{-{k^-}}(p)=h(x_c^-,v_u^-, 0)$, where $ \|v_u^-\|<\delta_1$ (see \eqref{eqn:balls}).
Consider the $(n_s)$-dimensional disk $D^s(f^{-k^-}(p))$  through $f^{-k^-}(p)$ given in \eqref{Dsk-},
and the $(n_c+n_u)$-dimensional disk $D^{cu} (f^{k^+}(p))$ through $f^{k^+}(p)$ given in \eqref{Dcuk+}.

To the point $f^{-k^-}(p)$ we attach the $(n_c+n_u)$-dimensional disk\break
$D^{cu}(f^{-k^-}(p))= f^{-k^+-k^-}(D^{cu} (f^{k^+}(p)))$, of fixed size independent of $p$; see~\eqref{Dcuk+}.

Then choose  a $C^0$-family of $n_s$-dimensional disks $\mathscr{D}^s(q)$ of fixed size independent of $p$, with $q\in D^{cu}(f^{-k^-}(p))$, satisfying the following conditions:
\begin{itemize}
\item
for $q= f^{-k^-}(p)$ the corresponding disk $\mathscr{D}^s(q)= \mathscr{D}^s(f^{-k^-}(p))$ is contained in $D^s(f^{-k^-}(p))$;
\item
for each  $q\in D^{cu}(f^{-k^-}(p))$, we have
\[f^{k^++k^-}(\mathscr{D}^s(q))\subset h\left(\{z_c^+\}\times \{w_u^+\}\times B_{\rho^s}(0)\right)),\]
where $z_c^+\in B_{\rho_c^+}(x_c^+)$, $w_u^+\in B_{\rho_u^+}(0)$ are defined as in  \eqref{eqn:toptrans}.
\end{itemize}
Observe that, by construction, for each $q\in D^{cu}(f^{-k^-}(p))$, $\mathscr{D}^s(q)$
is topologically transverse to $ D^{cu}(f^{-k^-}(p))$.

Thus, the $(n_c+n_u)$-dimensional disk   $D^{cu}(f^{-k^-}(p))$  is contained in $W^u(\Lambda)$, the $n_s$-dimensional disk
 $D^s(f^{-k^-}(p))$  is $(\delta/2)$-close,
in the $C^1$-topology, to $W^s(f^{-k^-}(p^-))$, and each disk $f^{k^++k^-}(\mathscr{D}^s(q))$ is topologically transverse to $D^{cu}(f^{k^+}(p))$.

We define the window $W^- $ and its exit and entry sets $(W^-)^\xt$, $(W^-)^\nt$, respectively, by:
\begin{equation}
\label{widehat1}
\begin{split}
W^-=&\bigcup_{q\in D^{cu}(f^{-k^-}(p))}\mathscr{D}^s(q),\\
(W^-) ^\xt=&\bigcup_{q\in\partial D^{cu}(f^{-k^-}(p))}\mathscr{D}^s(q),\\
(W^-) ^\nt=&\bigcup_{q\in D^{cu}(f^{-k^-}(p))}\partial\mathscr{D}^s(q).
\end{split}
\end{equation}
Furthermore, we choose the sizes of  $D^{cu}(f^{-k^-}(p))$  and of $\mathscr{D}^s(q)$, for $q \in D^{cu} (f^{-k^-}(p))$,
such that
$W^-$ is contained in a $(\delta/2)$-neighborhood of $f^{-k^-}(p)$, hence every point in $W^-$ is $\delta$-close to $f^{-k^-}(p^-)$.

We note that $W^-$ is a window; see Remark \ref{rem:windows_alternative}.

In  Section \ref{sec:$W^+$} below we will impose additional conditions on the sizes of $D^{cu}(f^{-k^-}(p))$ and of $\mathscr{D}^s(q)$.

We take a forward iterate $f^{k^++k^-}(W^-)$ of $W^-$.
The point $f^{-k^-}(p)$ is mapped by $f^{k^++k^-}$ onto $f^{k^+}(p)$.
For $k^+\geq n^*$  we have  $f^{k^+}(p)\in V_\Lambda$.
The set $f^{k^++k^-}(W^-)$ is still a window, being a homeomorphic copy of $W^-$ under $f^{k^++k^-}$,
with the exit and entry sets being defined by transporting the  exit and entry sets of $W^-$,  respectively, through  $f^{k^++k^-}$.

In fact, by construction
\begin{equation}
\label{widehat1k+k-}
\begin{split}
f^{k^-+k^+}(W^-)=&\bigcup_{q\in D^{cu}(f^{-k^-}(p))} f^{k^-+k^+}(\mathscr{D}^s(q)).\\
\end{split}
\end{equation}

\smallskip

\subparagraph{\em  Construction of $W^+$.}\label{sec:$W^+$}
We define a new window $W^+\subseteq V_\Lambda$ about
$f^{k^+}(p^+)=h(x_c^+,0,0)$ such that
$f^{k^++k^-}(W^-)$ is correctly aligned with $W^+$ under the identity map,
or, equivalently, $W^-$ is correctly aligned with $W^+$ under $f^{k^++k^-}$.
This new window will be a product of disks  in the linearized coordinates $h$.
The construction follows below.

The image set
${D}^{cu}(f^{k^+}(p)):=f^{k^++k^-}(D^{cu}(f^{-k^-}(p)))$
is a $(n_c+n_u)$-dimensional disk through $f^{k^+}(p)$ that is $(\delta/2)$-close to  $W^u(\Lambda)$; see \eqref{Dcuk+}.
This disk is transverse to $W^s(f^{k^+}(p^+))$.
Also denote:
$$ D^s(f^{k^+}(q)):=f^{k^++k^-}(\mathscr{D}^{s}(q)), \ q\in
D^{cu}(f^{-k^-}(p)).$$

For a given choice of the size of $D^{cu}(f^{-k^-}(p))$ we require $\rho_c^+>0$, $\rho_u^+>0$ to
be sufficiently small,  so that  \eqref{eqn:toptrans} and \eqref{eqn:projection_cu} hold.
We also require that $\rho_c^+, \rho_u^+ <\delta_1$.

Then we choose $0< \rho_s^+<\delta_1$, and require that all disks $\mathscr{D}^s(q)$ are small enough so that
\begin{equation}\label{widehatinc3}
\pi_s \left[h^{-1}\left(D^s(f^{k^+}(q))\right)\right]
\subseteq \textrm{int}\left [B_{\rho_s^+}(0)\right],
\end{equation}
for all $q\in D^{cu}(f^{-k^-}(p))$.

For future reference, we also have to set a lower bound for the sizes of the disks
$\mathscr{D}^s(q)$, $q\in D^{cu}(f^{-k^-}(p))$.
There exist $\delta_2>0$ defined by the property
that
\begin{equation}\label{delta2}
\textrm{int}[\pi_s(h^{-1}(\mathscr{D}^s(q)))]\supseteq B_{\delta_2}(0),
\end{equation}
for all $q\in D^{cu}(f^{-k^-}(p))$.

We now define the second `prototype' window $W^+$ around $f^{k^+}(p^+)$
to be given in the $h$-coordinates by
\[\begin{split}
W^+ =&h[B_{\rho_c^+}(x_c^+)\times B_{\rho_u^+}(0)\times B_{\rho_s^+}(0)],\\
(W^+)^\xt   =&h\left[\partial B_{\rho^+_c}(x^+_c)\times B_{\rho^+_u}(0)\times B_{\rho^+_s}(0)\right.\\
&\left.\cup B_{\rho^+_c}(x^+_c)\times \partial B_{\rho^+_u}(0)\times B_{\rho^+_s}(0)\right],\\
(W^+)^\nt   =& h[B_{\rho^+_c}(x^+_c)\times B_{\rho^+_u}(0)\times \partial  B_{\rho^+_s}(0)].
\end{split}\]

By the product property of  correct alignment Lemma \ref{ex:product}, the choices that we made imply that $W^-$ is correctly aligned with $W^+$ under $f^{k^++k^-}$.

\smallskip

It is useful  at this point to summarize the inter-dependence of the parameters involved in the construction  of the windows $W^-$ and $W^+$
so that they are correctly aligned under $f^{k^-+k^+}$.
\begin{itemize}

\item
The quantities $\rho_c^+$, $ \rho_u^+$, $ \rho_s^+$ from above can be chosen independently of the point $p\in\Gamma$,
but they depend on $k^-,k^+$  on the sizes of the disks involved in the definition of the window $W^-$.
\item
The powers $k^-,k^+$ can be chosen arbitrarily large with  $k^-\geq n^*$, $ k^+\geq n^*$, where $n^*$ depends only on $\delta$ and not on $p\in\Gamma$.
\item
The disks $D^{cu}(f^{-k^-}(p))$ and $D^{cu}(f^{k^+}(p))$, and implicitly the parameters $\rho^+_c$, $\rho^+_u$, depend on   $k^-$, $ k^+$.
In particular, for fixed  $k^-$, $k^+$,   the  parameters $\rho_c^+$, $\rho_u^+$ depend  on  the size of the disk  $D^{cu}(f^{-k^-}(p))$;
the smaller the disk $D^{cu}(f^{-k^-}(p))$ is, the smaller $ \rho_c^+$,  $ \rho_u^+$  need to be chosen.
This is due to the coupling of the center and hyperbolic dynamics, which
mixes the center and unstable directions when iterated along homoclinic orbit.
That is, the center and unstable directions of a disk are not preserved when the disk is iterated along a homoclinic orbit, as they `get mixed',
therefore, the image of a center-unstable rectangle iterated along the stable manifold of a point
does not remain a rectangle anymore, as the rectangle `gets distorted'.

\item
The disks  $\mathscr{D}^{s}(q)$, $q\in D^{cu}(f^{-k^-}(p))$, and implicitly the parameter $\delta_2$ in \eqref{delta2},
depend on $k^-$, $k^+$.
The sizes of these disks can be chosen independently of the size of the disk $D^{cu}(f^{-k^-}(p))$, provided this is sufficiently small.
That is,  if $D^{cu}(f^{-k^-}(p))$ is replaced by a smaller disk
$\tilde{D}^{cu}(f^{-k^-}(p))\subset D^{cu}(f^{-k^-}(p))$,
then we simply restrict the family of disks $\mathscr{D}^{s}(q)$ to those $q\in \tilde{D}^{cu}(f^{-k^-}(p))$,
without having to modify the size of the disks $\mathscr{D}^{s}(q)$.

\item
The parameter $\rho_s^+$ can be chosen independently of $k^-$, $k^+$, provided that   the disks
$\mathscr{D}^{s}(f^{-k^-}(q))$, $q\in D^{cu}(f^{-k^-}(p))$, are chosen  small enough,
\end {itemize}

\smallskip

\subparagraph{\em  Choice of $m^*$.}\label{sec:Choice of $m^*$}
Now we need to show that there exists  a number $m^*$ with the property that for every $m\geq m^*$ and every $k'^-\geq n^*$, and for every point $p'\in\Gamma$ with
$p'^-=f^{m+k'^-}(p^+)$, we can construct a window $W'^-$ near
$f^{-k'^-}(p')$
in a similar way in which  we  have constructed $W^-$,
such that $W^+$ is correctly aligned with $W'^-$ under $f^{m-k^+}$.

Since the power $m-k^+$ should be non-negative, we first require $m^*\geq k^+$. We also fix $k'^+=k^+ \geq n^*$.

A key observation is that, since $W^+$  is a window of product type relative to the $h$-coordinates,
the image $f^{m'}(W^+)$  is also a window of product type relative to the $h$-coordinates,  for any iterate $f^{m'}$,
provided that $f^{k}(W^+)$ remains in the domain $V_{\Lambda}$  of the map $h$
for $0\le k\le m'$.
This is due to the fact that, relative to the linearized coordinates, the map $f$ is
conjugate to $Nf$ (see Apendix \ref{sec:linearized_coordinates}).

Even in  the case when $f^{m'}(W^+)$ does not entirely remain in $V_{\Lambda}$ (e.g., it   `escapes' in the unstable directions),
 $f^{m'}(W^+)\cap V_{\Lambda}$ contains a sub-window of product type, say $\tilde W$.
If this window $\tilde W$ is correctly aligned with $W'^-$  under the identity map, it immediately  follows that  $f^{m'}(W^+)$
itself is correctly aligned with  $W'^-$.
So for all practical purposes we can assume  that $f^{m'}(W^+)$ stays in $V_{\Lambda}$.

We now take $\delta_1$ from \eqref{eqn:balls} and $\delta_2$ from \eqref{delta2}.
By \eqref{eqn:expcontr}, there exists $m^*\geq k^+$ large enough so that for $m'  \geq m^*-k^+$,
$f^{m'}(h(\{ x_c^+\}\times B_{\rho_u^+}(0)\times\{0\}))$ contains a disk in
$W^u(h(\{f^{m'}(x_c^+)\}\times \{0\}\times\{0\}))$ of radius $\delta_1$
relative to the $h$-coordinates, that is:
\begin{equation}\label{nstar1}
\textrm{int}[f^{m'}(h(\{x_c^+\}\times B_{ \rho_u^+}(0)\times\{0\}))]\supseteq h\left(\{ f^{m'}(x_c^+) \}\times B_{\delta_1}(0)\times\{0\}\right),
\end{equation}
and $f^{m'}(h(\{x_c^+\}\times \{0\}\times B_{\rho_s^+}(0)))$ is contained in a disk in
$W^s(h(\{f^{m'}(x_c^+)\}\times \{0\}\times\{0\}))$ of radius $\delta_2$, that is:

\begin{equation}\label{nstar2}
f^{m'}\left[h(\{ x_c^+ \}\times \{0\}\times B_{\rho_s^+}(0))\right]\subseteq
\textrm{int}[h(\{ f^{m'}(x_c^+) \}\times \{0\}\times B_{\delta_2}(0))],
\end{equation}
Observe that the parameter  $\delta_2$ in \eqref{delta2}
depends on $k'^-$ and $k'^+=k^+$.

Fix $m^*$ with these properties.
Note that $m^*$ depends, in particular, on the size of the unstable component $B_{\rho_u^+}(0)$ of the previous window $W^+$,
which in turn depends on the size of the disk $D^{cu}(f^{-k^-}(p))$ that is used in the construction of the first window $W^-$;
the smaller the radius  $\rho_u^+$ is, the larger $m^*$ needs to be chosen in order to satisfy~\eqref{nstar1}.

\smallskip

\subparagraph{\em Construction of $W'^-$.}
 Let $m\geq m^*$ and let $m'=m-k^+$.
 Assume that $p'\in\Gamma$ is such that $p'^-=f^{m+k'^-}(p^+)$.
We construct a third window $W'^-$ near  $f^{-k'^-}(p')$, in a similar way to the construction of $W^-$, such that  $W^+$ is correctly aligned under
$f^{m'}$ with  $W'^-$.

Consider the point $f^{-k'^-}(p')\in W^u(f^{-k'^-}(p'^-))$.

Choose a sufficiently small $(n_c+n_u)$-dimensional disk
$ \tilde{D}^{cu}(f^{-k'^-}(p'))\subseteq
{\mathscr{D}}^{cu}_{k'^-,k'^+}(f^{-k'^-}(p'))$ in
$W^u(\Lambda)$ such that it  contains the point
$f^{-k'^-}(p')$ and it satisfies   the following condition:
\begin{equation}\label{widehatprimecond1}
\pi_{c,u}[h^{-1}(\tilde{D}^{cu}(f^{-k'^-}(p')))]\subseteq
\textrm{int}[h^{-1}\circ f^{m'}\circ h(B_{\rho_c^+}(x_c^+)\times B_{\rho_u^+}(0)\times\{0\})].
\end{equation}

The size of the disk
$\tilde{D}^{cu}(f^{-k'^-}(p'))$ can be chosen to depend only on the window $W^+$, on $m'$ and
$\delta_1$ in  \eqref{eqn:balls}, and
independently of the point $p'\in\Gamma$.

Then we choose a $C^0$-family of $n_s$-dimensional disks $\tilde{\mathscr{D}}^s(q')$,
with $q'\in \tilde{D}^{cu}(f^{-k'^-}(p'))$,  such that for $\delta_2$ in \eqref{delta2}
\begin{equation}\label{delta2tilde}
\textrm{int}[\pi_s(h^{-1}(\tilde{\mathscr{D}}^s(q')))]\supseteq B_{\delta_2}(0),
\end{equation}
for all $q'\in \tilde{D}^{cu}(f^{-k'^-}(p'))$, and when $q'=f^{-k'^-}(p')$,
\[
\tilde{\mathscr{D}}^s(f^{-k'^-}(p'))\subset W^s(f^{-k'^-}(p'^+)).
\]

As we pointed out earlier, the parameter $\delta_2$ is independent
of  the choice of  the disk $\tilde{D}^{cu}(f^{-k'^-}(p'))$, provided this is sufficiently small, and only depends on $k'^-$ and $k'^+=k^+$.
For fixed $k'^-$, $k'^+$, and $\delta_2$ sufficiently small,    a family of disks $\tilde{\mathscr{D}}^s(q')$ satisfying \eqref{delta2tilde} can always
be constructed.

Conditions \eqref{nstar2} and \eqref{delta2tilde}  imply that the projection of  each  $\tilde{\mathscr{D}}^{s}(q')$,
for $q'\in \tilde{D}^{cu}(f^{-k'^-}(p'))$,
onto the stable coordinates contains  the stable component of $f^{m'}(W^+)$ inside it, that is
\begin{equation}\label{widehatprimecond2}
\textrm{int}[\pi_s(h^{-1}(\tilde{\mathscr{D}}^s(q')))]\supseteq  h^{-1}\circ f^{m'}\circ h(\{z_c^+\}\times\{0\}\times B_{\rho_s^+}(0)),
\end{equation}
for all $z_c^+\in  B_{\rho_c^+}(x_c^+)$.

The window $W'^-$ is then defined similarly to $W^-$, by
\begin{equation}
\label{widehat}
\begin{split}
W'^-=&\bigcup_{q'\in\tilde{D}^{cu}(f^{-k'^-}(p'))}\tilde{\mathscr{D}}^s(q'),\\
{W'} ^\xt=&\bigcup_{q'\in\partial\tilde{D}^{cu}(f^{-k'^-}(p'))}\tilde{\mathscr{D}}^s(q'),\\
{W'} ^\nt=&\bigcup_{q'\in\tilde{D}^{cu}(f^{-k'^-}(p'))}\partial\tilde{\mathscr{D}}^s(q').
\end{split}
\end{equation}

Conditions \eqref{widehatprimecond1} and \eqref{widehatprimecond2} imply that the  product property of  correct alignment
applies -- Lemma \ref{ex:product} --, and hence we obtain that $W^+$ is correctly aligned under $f^{m'}$ with  $W'^-$.
An important  point to keep in mind is that we have  no  control on the size of the $(n_c+n_u)$-dimensional disk
$\tilde{D}^{cu}(f^{-k'^-}(p'))$ involved in the construction
$W'^-$.
We choose this disk so that its center-unstable part is contained in the center-unstable component of $f^{m'}({W}^+)$.
Thus, the size of the disk $\tilde{D}^{cu}(f^{-k'^-}(p'))$ utilized in the construction
$W'^-$ may be
smaller than the size of the disk $ {D}^{cu}(f^{-k'^-}(p))$ utilized in the construction of $W^-$.

A schematic representation of the construction of the triplet of windows $W^-, {W}^+,  W'^-$ constructed so far is shown
in Figure \ref{fig:shadowing-pseudo}.
\begin{figure}
\centering
\includegraphics[width=0.95\textwidth]{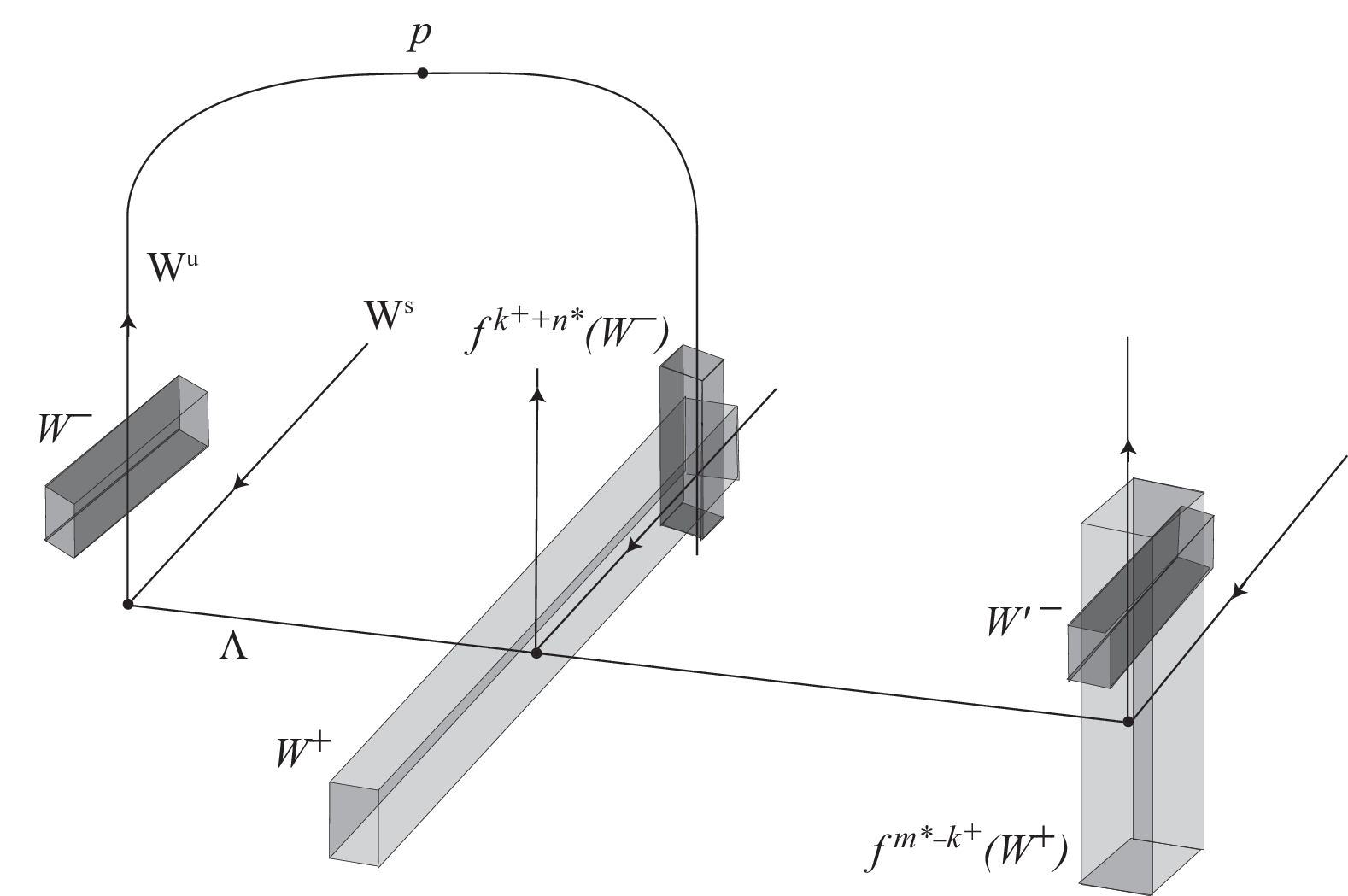}
\caption{Construction of windows.}
\label{fig:shadowing-pseudo}
\end{figure}

\smallskip

We anticipate that, in order to continue this construction of triplets of correctly aligned windows starting from $W'^-$,
the number of iterates ${m'}^*$ that we need to choose at the next step, in order to satisfy \eqref{nstar1}, may need to be   larger than $m^*$.
Without further conditions on the dynamics, we cannot guarantee a uniform choice of $m^*$ to work for all steps of the construction.
In Appendix \ref{Turaev} we show an example, which was kindly communicated to us by Dmitry Turaev, showing that a `uniform' version of
this shadowing lemma is not true in general.

\smallskip

\paragraph{\em  Definition of $m^*_i$}\label{sec:mi*}
\label{m_i_star}
Let $p_0\in\Gamma$ be an arbitrary homoclinic point, and let $n_0\geq n_*$. We construct a pair of
windows $W^-_0$ in a $\delta/2$-neighborhood of $f^{-n_0}(p_0)$, and $W^+_0$ in a $\delta$-neighborhood of $f^{k^+}(p^+_0)$,
where $k^+\geq n^*$ is fixed, such that  $W^-_0$  is correctly aligned with  $W^+_0$ under $f^{n_0+k^+}$.
Then, the procedure in   Section \ref{sec:Choice of $m^*$} provides an $m_0^*=m^*_0(n_0)$ that depends on $W^-_0$ and $W^+_0$,
and implicitly on $n_0$, and is independent of $p_0\in\Gamma$.
Inductively, if $m^*_0,\ldots m^*_{i-1}$ have been defined, let
\begin{equation}
\begin{split}
&n_0\geq n^*,\ldots,n_{i-1}\geq n^*, n_i\geq n^*,\\
&m_0\geq m^*_0(n_0),\\
&\cdots \\ &m_{i-1}\geq m^*_{i-1}(n_0,\ldots, n_{i-1},m_0,\ldots,m_{i-2}).
\end{split}
\end{equation}

Let \[p_0,\ldots,p_{i}\in \Gamma,\]
such that $f^{n_j+m_{j-1}}(p^+_{j-1})=p^-_{j}$ for $j=1,\ldots,i$.

Let
\[
W^-_0,  W^+_0, W^-_1,\ldots,  W^-_i,  W^+_i,
\]
be a sequence of correctly aligned windows, constructed as above, such that $W^-_{j-1}$ in a $\delta/2$-neighborhood of
$f^{-n_{j-1}}(p_{j-1})$,  $W^+_{j-1}$ in a $\delta$-neighborhood of $f^{k^+}(p^+_{j-1})$,
$W^-_{j-1}$  is correctly aligned with  $W^+_{j-1}$ under $f^{n_{j-1}+k^+}$,
and  $W^+_{j-1}$ is correctly aligned with $W^-_{j}$ under $f^{m_{j-1}-k^+}$, for $j=1,\ldots,i$.

Then the procedure in Section  \ref{sec:Choice of $m^*$} provides an
\[
m_i^*=m^*_i(n_0,\ldots,n_{i-1},n_i,m_0,\ldots,m_{i-1})
\]
as in the statement of Lemma \ref{lem:key1-alternative}, that depends on $W^-_0,\ldots,W^+_i$,
and implicitly on $n_0,\ldots,n_{i-1},n_i$, $m_0,\ldots,m_{i-1}$, but is independent of $p_0,\ldots,p_{i-1}\in\Gamma$.

\smallskip

\paragraph{\em  Construction of an infinite  sequence of correctly aligned windows}
\label{sec:infinite_sequence}{$ $}

Take a pseudo-orbit $\{y_i\}_{i\geq 0}$ as in the statement of Lemma \ref{lem:key1-alternative}.
We implicitly assume that $f^{n_i}(y_i)$ is in the domain $\Omega^{-}(\Gamma)$ of  $\sigma$, and hence
$\sigma\circ f^{n_i}(y_i)$ is in the range $\Omega^{+}(\Gamma)$ of $\sigma$.
Thus $W^u(f^{n_i}(y_i))\cap W^s(\sigma(f^{n_i}(y_i))\cap\Gamma=\{p_i\}$ for some uniquely defined homoclinic point $p_i\in\Gamma$.
Fix $k^+\geq n^*$.

Starting with the homoclinic point $p_0$ we construct inductively an infinite sequence of correctly aligned windows along the pseudo-orbit,
\[
W^-_0,W^+_0,  W^-_1, \ldots,  W^-_i,W^+_i,  W^-_{i+1}, W^+_{i+1},\ldots,
\]
such that for each $i\geq 0$ we have
(choosing $p_i^-=f^{n_i}(y_i)$, $p_i^+=\sigma(f^{n_i}(y_i))$, $k_i^-=n_i$, and $k^+$ fixed)
\begin{itemize}
 \item $W^-_{i}$ lies within a $\delta$-neighborhood of $y_{i}$;
 \item $W^+_{i}$  lies within a $\delta$-neighborhood of $f^{k^+}\circ \sigma\circ f^{n_{i}}(y_{i})$, where $n_i\geq n^*$;
 \item $W^-_{i}$ is correctly aligned with $W^+_{i}$ under $f^{k^++n_{i}}$,
 \item  $W^-_{i+1}$ lies within a $\delta$-neighborhood of $y_{i+1}=f^{m_i}\circ\sigma\circ f^{n_i}(y_i)$;
 \item $W^+_{i}$ is
correctly aligned with $W^-_{i+1}$ under  $f^{m_{i}-k^+}$, where $m_i\geq m^*_i$.
\end{itemize}

The shadowing property of correctly aligned windows  -- Theorem \ref{theorem:detorb} -- implies
that there exists a point $z_0\in W^-_{0}$ whose forward orbit  visits all windows in the prescribed order. In particular, the orbit points given by
$z_{i+1}=f^{m_i+n_i}(z_i)$, satisfy $z_i\in W^-_i$ for all $i\geq 0$. Since each $W^-_i$ is
contained inside a  $\delta$-neighborhood of $y_i$, it follows that $d(z_i,y_i)<\delta$ for all $i\geq 0$.
\qed

\subsubsection{A proof of  Lemma~\ref{lem:key1-alternative} using the obstruction property}\label{proofshadowinglemma2}
In this section we give an alternative proof of Theorem \ref{lem:key1-alternative}.

\smallskip
\paragraph{\em Outline}
The proof  is based on the construction of a nested sequence of  closed balls $B_{i+1}\subset B_i$ in a neighborhood of the
first point of the pseudo-orbit $y_0$, such that taking $z_0 \in B_k=\bigcap _{0\le i\le k}B_i$
one has that
$z_0\in B_\delta (y_{0})$ and  $z_{i+1}=f^{m_i+n_i}(z_0)\in B_\delta (y_{i+1})$ for $i=0,1\dots, k$, for any $k\in \N$.

Moreover, taking $z_0 \in B_\infty=\bigcap _{ i\ge 0 }B_i \ne \emptyset$,
one has that  $z_{i+1} \in B_\delta (y_{i+1})$ for    any $i\in \N$.

The argument will be done by induction.

We will define the value of $n^*$, $m^*$ at every step of the induction process. We will see that $n^*$ can be taken once and for all but
$m^*$ will depend on the previous choices, but is independent  of the given sequence $y_i$.

\smallskip

\paragraph{\em Choice of $n^*$ and $m^*$}
Consider the homoclinic channel  $\Gamma$ and the corresponding scattering map
$\sigma:\Omega^{-}(\Gamma)\to \Omega^{+}(\Gamma)$.
We will choose $\delta>0$ and consider $V_\Lambda$ and $V_\Gamma$ contained in neighborhoods  of size $\delta$ of the compact manifolds $\Lambda$ and
$\Gamma$, respectively.

We define $n^*=n^*(\delta)$ the  same number as in subsection \ref{sec:lambdalemma}.
In particular, given any point $p \in \Gamma$,
for any $n\in \N$  with $n\ge n^*$,  one has that $f^{\pm n}(p) \in V_\Lambda$.
Moreover, this property also holds for points in $W^{u,s}(\Lambda)\cap V_\Gamma$ when iterating them backwards or forward respectively.

Moreover, we will modify $n^*$ to have the following additional property.
Assume we have $p\in \Gamma$ and let $p^-,p^+ \in \Lambda$ be the unique points for which $W^u(p^-)\cap W^s(p^+) \cap \Gamma =\{p\}$.
\begin{enumerate}
 \item
Let a point $x\in W^s(f^{-k^-}(p^-))$ and $B\subset B_\delta (f^{-k^-}(p^-))$ be any ball centered at $x$ of fixed radius $\rho >0$  small enough. Then we have that
$$
B\subset V_\Lambda, \ \ x\in  B\cap W^s(f^{-k^-}(p^-)) \ne \emptyset .
$$
As $W^s(p^+)$ intersects tramsversaly $W^u(\Lambda)$ at the homoclinic point $p$,  by the Lambda Lemma
there exists a point $\bar x \in  W^s(p^+)\cap V_\Gamma$ such that $f^{-k^-}(\bar x) \in B$ if $k^- > n^*$.
The value of $n_* $ depends on $\rho$, which is fixed once for all, and also on the angle of intersection of the stable and unstable
manifolds of $\Lambda$ along $\Gamma$ which, by the hypothesis of compactness, is bounded bellow by a fixed quantity.
\item
By continuity, there exists a ball $V\subset V_\Gamma$  centered at  $\bar x$ such that $f^{-k^-}(\bar x) \in f^{-k^-}(V) \subset B$.
\end{enumerate}

The value of $n^*$ will be fixed from now on. Now we explain how we choose $m^*$ at every step of the process.

Assume  that we also have $p'\in \Gamma$ and  $p'^-$, $p'^+$ with the same properties as
$p$ and  $p^-$, $p^+$, and such that $f^{m+k'^-}(p^+)= p'^-$. Equivalently
\begin{equation}\label{eq:nexthomoclinic}
p^+= f^{-(k'^-+m)}(p'^-)
\end{equation}
Take the point $\bar x \in W^s(p^+)$ and the ball $\bar x \in V \subset V_\Gamma$  centered at $\bar x$ .
Then choose $k^+ \ge n^*$. The value of $k^+$ will be fixed along the process.
\begin{enumerate}
 \item
We know that
$f^{k^+}(\bar x) \in B_\delta (f^{k^+}(p^+))\cap V_\Lambda\cap W^s(f^{k^+}(p^+))$,  and there exists a ball $U$ centered at $f^{k^+}(\bar x) $ such that:
\begin{equation*}\begin{split}
U\subset B_\delta (f^{k^+}(p^+))\subset V_\Lambda, \\ f^{k^+}(\bar x) \in U\cap W^s(f^{k^+}(p^+)) \ne \emptyset ,\\  f^{-k^+}(U)\subset V.
\end{split}\end{equation*}
\item
As, by \eqref{eq:nexthomoclinic}, $f^{k^+}(p^+)= f^{-(k'^-+m-k^+)}(p'^-))$, the ball $U$ satisfies
$$
f^{k^+}(\bar x)\in U\cap W^s(f^{-(k'^-+m-k^+)}(p'^-))) \ne \emptyset.
$$
\item
Now we apply the Lambda Lemma to $U$;
we know that $W^s(p'^+)$ intersects transversally $W^u(\Lambda)$ at $p'$, and therefore, if $k'^-+m-k^+>m^*$ big enough (depending of the size of $U$),
there exists $ \bar x' \in  W^s(p'^+)$ such that:
 $f^{-(k'^-+m-k^+)}( \bar  x')\in U$.
\item
By continuity, there exists a ball centered at  $ \bar x'\in  V'\subset V_\Gamma$, such that
${f^{-(k'^-+m-k^+)}}(V')\subset U$.
\end{enumerate}
Summarizing: Given  a point $x\in W^s(f^{-k^-}(p^-))$ and a ball $B$  centered at $x$ of fixed radius $\rho >0$ small enough  with the property that
\begin{equation*}\begin{split}
B\subset B_\delta (f^{k^-}(p^-))\subset V_\Lambda, \\ x\in B\cap W^s(f^{-k^-}(p^-)) \ne \emptyset ,
\end{split}\end{equation*}
we have produced:
\begin{enumerate}
 \item
For $k^- \ge n^*$, a ball $V \subset V_\Gamma$, centered at a point  $\bar x \in  W^s(p^+)\cap V_\Gamma$ such that  $f^{-k^-}(V) \subset B$.
 \item
For $k^+ \ge n^*$, and fixed,  a ball $U\subset B_\delta (f^{k^+}(p^+))\subset V_\Lambda$ centered at the point  $f^{k^+}(\bar x) \in  W^s(f^{k^+}(p^+))\cap U $
such that $f^{-k^+}(U)\subset V$.
\item
For $k'^-+m-k^+ \ge m^*$, a ball  $V'\subset V_\Gamma $, centered at a point  $\bar x' \in  W^s(p'^+)\cap V_\Gamma$ such that
${f^{-(k'^-+m-k^+)}}( V')\subset U$.
\item
Moreover, as $k'^- \ge n^*$ we can also ensure  $f^{-k'^-}(V')\subset B_\delta (f^{k'^-}(p'^-))$.
\end{enumerate}
The values of $k^+, k^-, k'^-$ are taken bigger than $n^*$, which is already fixed, but the value of $m^*$ depends on the size of $U$ and $m^* >n^*$,
but it is independent of the points $p, p', p^\pm, (p')^\pm$.
As the balls  $U$, $V$ will decrease in size during the induction process, the value of $m^*$ will incresase depending
of the previous iterates.

\smallskip

\paragraph{\em Inductive construction}
Now we begin the construction of the shadowing orbit $\{z_i\}$ once the pseudo-orbit $\{y_i\}$ is given.
The required values of $n^*$, $k^+$, are fixed (one can use, for instance, $k^+=n^*$) and $m^*_i$ does not
depend of the given pseudo-orbit, but only on the numbers $n_i,m_j$.

The first step in the induction procedure is done separately because it requires a slightly different reasoning.
In this first step, $p^-= f^{n_0}(y_0)$, $p^+=\sigma (f^{n_0}(y_0))$, and $k^-=n_0$.

\begin{enumerate}
\item Choose  $x_0 \in W^s(y_0)$ and $B_0$ be any ball centered at $x_0$ of fixed radius $\rho>0$ such that
\begin{equation*}\begin{split}
B_0\subset B_\delta (y_0)\subset V_\Lambda, \\ x_0 \in B_0\cap W^s(y_0) \ne \emptyset .
\end{split}\end{equation*}
As $W^u(\Lambda)\pitchfork  W^s(\sigma(f^{n_0}(y_0)))$ at an homoclinic point that we call $p_0$,  by the Lambda Lemma
there exists a point $\bar x_0 \in  W^s(\sigma(f^{n_0}(y_0)))\cap V_\Gamma$ such that $f^{-n_0}(\bar x_0) \in B_0$ if $n_0 \ge n^*$.

\item
By continuity, there exists a ball $V_0\subset V_\Gamma$ centered at $\bar x_0$ such that
\begin{equation}\label{V0}
f^{-n_0}(V_0) \subset B_0\subset B_\delta (y_0)\subset V_\Lambda.
\end{equation}
\end{enumerate}

Now we proceed with the second step of the induction procedure:
\begin{enumerate}
\item
 By the definition of $n^*$, as $\bar x_0 \in  W^s(\sigma(f^{n_0}(y_0)))\cap V_\Gamma$,  as $k^+\ge n^*$, we know that
\[
f^{k^+}(\bar x_0) \in  W^s(f^{k^+}(\sigma(f^{n_0}(y_0))))\cap B_\delta (f^{k^+}(\sigma(f^{n_0}(y_0))))\subset V_\Lambda.
\]
\item
By continuity, there is a ball $U_1$ centered at $f^{k^+}(\bar x_0)$ such that:
\begin{equation}\label{U1V0}\begin{split}
U_1 \subset B_\delta (f^{k^+}(\sigma (f^{n_0}(y_0))))\subset V_\Lambda,\\
f^{k^+}(\bar x_0)\in U_1 \cap W^s(f^{k^+}(\sigma(f^{n_0}(y_0)))),\\
f^{-k^+}(U_1)\subset V_0.
\end{split}\end{equation}
\item
Recall that $y_1= f^{m_0}(\sigma (f^{n_0}(y_0)))$, and therefore
$f^{k^+}(\sigma (f^{n_0}(y_0))) = f^{k^+-m_0}(y_1)$.
\item
The next step is the application of the Lambda Lemma. Now $p'^-= f^{n_1}(y_1)$, $p'^+= \sigma (f^{n_1}(y_1))$ and $k'^-=n_1$.
As $W^u(\Lambda)$ intersects transversally $W^s(\sigma(f^{n_1}(y_1)))$ at an homoclinic point that we will call $p_1$, if we take $n_1 >k^+\ge n^*$ and
$m_0>m_0^*$, where $m_0^*$ is the value $m^*$ given in the general step and depends
on the size of $U_1$ and therefore on $n_0$, one has that $n_1+m_0-k^+ > k^++m_0-k^+=m_0 >m_0^*$ and  there exists
$x_1 \in W^s(\sigma(f^{n_1}(y_1)))$ and a ball $ V_1$ centered at $x_1$ such that:
\begin{eqnarray}
f^{-{n_1}}(x_1)\in f^{-n_1}(V_1)&\subset& B_\delta(y_1), \label{V1subset} \\
f^{-(n_1+m_0-k^+)}(x_1) \in f^{-(n_1+m_0-k^+)}(V_1)&\subset& U_1. \label{V1U1}
\end{eqnarray}
\item
If we now take
$B_1 = f^{-(n_0+n_1+m_0)}(V_1)$, we have, using \eqref{V1U1}, \eqref{U1V0}, \eqref{V0},  that:
\begin{equation}\label{B1B0}
\begin{split}
B_1 &= f^{-(n_0+n_1+m_0)}(V_1) = f^{-n_0-k^+}\circ f^{-n_1-m_0+k^+}(V_1) \\&\subset   f^{-n_0-k^+}(U_1) \subset  f^{-n_0}(V_0) \subset B_0.
\end{split}
\end{equation}
Moreover,  if we take $z_0 \in B_1$ it satisfies, by \eqref{B1B0} and \eqref{V1subset} and using that $B_0\subset B_\delta (y_0)$:
\begin{eqnarray*}
z_0 &\in & B_\delta(y_0),\\
f^{n_0+m_0}(z_0) &\in& f^{-n_1}(V_1)\subset B_{\delta}(y_1).
\end{eqnarray*}

\end{enumerate}

Once we have done the two first steps, we can proceed with the general induction step.

Assume we have built the sequence $\bar x_i\in W^s(\sigma (f^{n_i}(y_i)))\cap V_\Gamma$, a ball $V_{i}\subset V_\Lambda$  centered at $\bar x_i$, $i=0,\dots j$
and $U_{i+1}$ a ball centered at $f^{k^+}(\bar x_i)$, for $n_i >k^+\ge n^*$ and $m_i \ge  m_i^*$, for $i=0,\ldots ,j$, with the properties:
\begin{itemize}
 \item
 $f^{-n_{i}}(V_{i})\subset B_\delta(y_{i})$,
 \item
 $f^{-(n_{i}+m_{i-1}-k^+)}(V_{i})\subset U_{i}$,
\item $U_{i+1} \subset B_\delta (f^{k^+}(\sigma (f^{n_i}(y_i))))$, \item
$f^{k^+}(\bar x_i)\in U_{i+1}$, \item $f^{-k^+}(U_{i+1})\subset V_i$.

\end{itemize}

We also assume that we have $\bar x_{j+1}\in V_{j+1}\cap W^s(\sigma (f^{n_{j+1}}(y_{j+1})))$, such that
\begin{itemize}
 \item
 $f^{-n_{j+1}}(V_{j+1})\subset B_\delta(y_{j+1})$,
 \item
 $f^{-(n_{j+1}+m_{j}-k^+)}(V_{j+1})\subset U_{j+1}$.

  \end{itemize}

Let
 $$
 B_{j+1}= f^{-n_{j+1}}\circ f^{-\sum _{i=0}^{j}m_i+n_i} (V_{j+1}),
 $$
 and we have that $B_{j+1} \subset B_{j}\subset B_{j-1}\subset\dots\subset B_0$.

 To proceed, first we look for a ball $U_{j+2}$ centered at $f^{k^+}(\bar x_{j+1})$ such that:
\begin{eqnarray*}
&&U_{j+2} \subset B_\delta (f^{k^+}(\sigma (f^{n_{j+1}}(y_{j+1})))),\\
&&f^{k^+}(\bar x_{j+1})\in U_{j+2}\cap W^s(f^{k^+}\sigma (f^{n_{j+1}}(y_{j+1}))), \\
 &&f^{-k^+}(U_{j+2})\subset V_{j+1}.
\end{eqnarray*}
The value of $k^+$ and $n^*$ are fixed, but the size of $U_{j+2}$ depends on the size of $V_{j+1}$ and therefore on the previous steps.
Then, applying the Lambda Lemma, using that $y_{j+2} = f^{m_{j+1}}(\sigma (f^{n_{j+1}}(y_{j+1}))$,
$W^u(\Lambda))\pitchfork W^s(\sigma(f^{n_{j+2}}(y_{j+2})))$
at a point
$p_{j+2}$,
we will find $\bar x_{j+2}\in W^s(\sigma(f^{n_{j+2}}(y_{j+2})))$
and a ball $V_{j+2}\subset V_\Gamma$ centered at  $\bar x_{j+2}$ such that, if $n_{j+2} >k^+\ge n^*$, and $m_{j+1}>m^*_{j+1}$, then
\begin{itemize}
 \item
 $f^{-n_{j+2}}(V_{j+2})\subset B_\delta(y_{j+2})$,
 \item
 $f^{-(n_{j+2}+m_{j+1}-k^+)}(V_{j+2})\subset U_{j+2}$.
\end{itemize}

Observe that the value $m^*_{j+1}$ is the
general value $m^*$ that now depends on the size of $U_{j+2}$, and therefore of all the previous steps.

Finally, define
$$
B_{j+2}= f^{-\sum _{k=0}^{j}m_k+n_k}\circ f^{-n_{j+2}-m_{j+1}-n_{j+1}}  (V_{j+2}).
$$

Then, we have:
\begin{equation}
\begin{split}
f^{-n_{j+2}-m_{j+1}-n_{j+1}}  (V_{j+2})&=f^{-n_{j+1}-n^*}\circ f^{-n_{j+2}-m_{j+1}+n^*}(V_{j+2})\\
&\subset   f^{-n_{j+1}-n^*}(U_{j+2}) \subset  f^{-n_{j+1}}(V_{j+1}).
\end{split}
\end{equation}

Therefore
$$
B_{j+2}\subset  f^{-\sum _{k=0}^{j}m_k+n_k}\circ f^{-n_{j+1}}  (V_{j+1})=B_{j+1}.
$$

This finishes the induction procedure.
Observe that if $z_0 \in \bigcap _{0\le i\le j}B_j$ and we consider the orbit $z_{i+1}=f^{m_i+m_i}(z_i)$ we have that:
\begin{itemize}
 \item
 $z_0 \in B_0 \subset B_\delta (y_0)$.
 \item
For all $i=0,\dots j$, $z_0 \in B_i$, and therefore, by the definition of $B_i$,
$z_i= f^{n_0+m_0+\dots +n_{i-1}+m_{i-1}}(z_0)\in f^{-n_i}(V_i) \subset B_\delta (y_i)$.
 \end{itemize}
To finish the proof we just point out that, the definition of $m_{j+1}^*$ depends of the size of the balls $U_{j+1}$ but not on the points $y_j$ themselves.
Therefore, if another pseudo-orbit is given with the same indexes $n_i,m_j$ the same choices of $n^*$ and $m_i^*$ will work.
\qed

\subsubsection{Remarks}
\begin{rem}
In the proof of Lemma \ref{lem:key1-alternative} given in Section \ref{proofshadowinglemma},  we have constructed windows $W^-_i, W^+_i,W^-_{i+1}$ in $V_\Lambda$ such that  $f^t(W^+_i) \subseteq V_\Lambda$ for all $0\leq t\leq m_i-k^+$,   so the corresponding
segment  of the shadowing orbit of $z_i$ stays in $V_\Lambda$ for this entire time.
Thus, the construction in the proof of the lemma enables one to find shadowing orbits that stay
close to $\Lambda$ for any sufficiently long time intervals, between two consecutive homoclinic excursions.
\end{rem}

\begin{rem}
Lemma \ref{lem:key1-alternative} provides a true forward orbit that shadows a given forward pseudo-orbit.
The current proofs do not allow for  immediately extending this result for bi-infinite orbits.
We remark that there is no assumption on the inner dynamics given by $f_{\mid \Lambda}$.
In the proof given in Section \ref{proofshadowinglemma}, the alignment of windows in the center directions was achieved by defining, at
each step of the construction, the center component of  $W^-_{i+1}$  as a ball inside some forward image of the center-component of $W^+_{i}$.
Thus, the consecutive balls in the center direction can get smaller and smaller in size as $i$ increases.
So if we try to continue the procedure in backwards time, the center-components of the windows $W^+_i$, $i\leq 0$, may get bigger and bigger in size. Thus, we may loose control on the shadowing trajectory, that is, the resulting shadowing orbit does not follow $\delta$-closely the prescribed pseudo-orbit.
\end{rem}

\begin{rem}\label{rem:comparisons} Statements related  to Lemma \ref{lem:key1-alternative} appear in \cite{DelshamsGR2016,DelshamsGR13,GideaR12}.
The main difference is that the statements in these papers assume certain geometric conditions on the inner dynamics.

There is also a related version of the Shadowing Lemma in \cite{GelfreichT2014}, but only for finite pseudo-orbits; moreover, those pseudo-orbits are subject to certain conditions that are very different from ours.
\end{rem}

\begin{rem}
It is interesting to note that the geometric proof of
Lemma~\ref{lem:key1-alternative} given in Section \ref{proofshadowinglemma2} works in infinite dimensions.
One only needs to substitute the compactness assumptions
by the assumption that the regularity of the maps -- and hence of
the manifolds are uniform.

Indeed,  infinite dimensional versions of the theory of
normally hyperbolic manifolds appear in \cite{BatesLZ08,ShatahZ03}.
An infinite dimensional version of the inclination lemma
appears in  \cite{LlaveOP}.
Note also that the nested balls arguments
also works in infinite dimensions when the space we consider
is reflexive (or the dual of Banach space).
It suffices to note that by Banach-Alaoglu
theorem, balls are compact in the weak${}^*$-topology.
\end{rem}

\appendix
\label{section:background}
\section{Normally hyperbolic invariant manifolds and the scattering map.}\label{subsection:NHIM}
In this section we recall the
background on normally hyperbolic invariant  manifolds
and  the  definition of
the scattering map and its geometric properties.

The main references for normally hyperbolic manifolds
are \cite{Fenichel71,Fenichel74,HirschPS77,Pesin04}.
Even if the definitions of \cite{Fenichel71,Fenichel74}
and \cite{HirschPS77} are not completely equivalent, the results
that we use are very basic and appear in both treatments
as well as in several subsequent treatments
\cite{BatesLZ00, BatesLZ08}.
The properties of the scattering map
appear  in \cite{DelshamsLS08a}.

Let $f:M\to M$ a $C^r$ map on a $C^r$-differentiable manifold $M$.
Assume that there exists a manifold $\Lambda\subseteq M$ that
is a normally hyperbolic invariant manifold for $f$.
We will assume that the derivatives of $f$ are uniformly continuous
and uniformly bounded
in a neighborhood of $\Lambda$. This is, of course, automatic if
$\Lambda$ is a compact manifold and many of the results
are stated only for compact manifolds, but as remarked in
\cite{HirschPS77,BatesLZ00,BatesLZ08}, only the uniform continuity and uniform
boundedness  is needed.

We recall that, following
\cite{Fenichel71,Fenichel74,HirschPS77, Pesin04} we say that
a  $\Lambda \subset M$ is a hyperbolic manifold if
exists a splitting of the tangent bundle of $TM$ into $Df$-invariant
sub-bundles
\[TM=E^u\oplus E^s\oplus T\Lambda,\] and there exist a constant $C>0$ and
rates \begin{equation}\label{rates0} 0<\lambda_+ < \eta_- \le 1 \le \eta_+ \le   \mu_-,\end{equation}
such that for all
 $x\in\Lambda$ we have
\begin{equation}\label{characterization}
\begin{split}
v\in E^s_x  &\Leftrightarrow \|Df^k_x(v)\|\leq C\lambda_+^k\|v\|  \textrm{ for all } k\geq 0,\\
v\in E^u_x  &\Leftrightarrow \|Df^k_x(v)\|\leq C\mu_-^{-k}\|v\|  \textrm{ for all } k\leq 0,\\
v\in T_x\Lambda &\Leftrightarrow \|Df^k_x(v)\|\leq C\eta_+^k \|v\|, \quad
\|Df^{-k}_x(v)\| \leq C\eta_-^{-k} \|v\|,
\textrm{ for all } k \geq 0.
\end{split}
\end{equation}

If $Df(x), Df^{-1}(x)$ are uniformly bounded, we have
that there are opposite inequalities, namely there exist $\lambda_-\leq \lambda_+$ and $\mu_+\geq\mu_-$ such that
\begin{equation} \label{complementary}
\begin{split}
v\in E^s_x  &\implies  \|Df^k_x(v)\|\geq C\lambda_-^k \|v\|  \textrm{ for all } k\geq 0,\\
v\in E^u_x  &\implies  \|Df^{k}_x(v)\|\geq  C\mu_+^{-k} \|v\|  \textrm{ for all } k\leq 0.\\
\end{split}
\end{equation}

Note that, of course, if the inequalities \eqref{characterization},
\eqref{complementary} hold for some rates, they also
hold for other rates $\tilde \lambda_\pm$, $\tilde \mu_\pm$,  $\tilde \eta_\pm$ satisfying \eqref{rates0}
such that
\begin{equation} \label{worsebounds}
[\lambda_-, \lambda_+] \subset  [\tilde \lambda_-, \tilde \lambda_+],\quad
[\mu_-, \mu_+] \subset  [\tilde \mu_-, \tilde \mu_+],\quad
[\eta_-, \eta_+] \subset  [\tilde \eta_-, \tilde \eta_+].
\end{equation}
Clearly, the bounds for the $\tilde \lambda_\pm, \tilde \mu_\pm.
\tilde \eta_\pm$ are less sharp than those for the original values.

If we change the metric in the manifold $M$ by an equivalent metric,
the rates $\lambda_\pm, \eta_\pm,   \mu_\pm $ are not altered, but
the constant $C$ can be modified. A standard construction
\cite{HirschPS77} shows that, for any rates that
satisfy \eqref{worsebounds} with strict inclusions, we can find a metric
(called \emph{adapted metric} ) equivalent to and as smooth as
the original one   in such a way that
$C = 1$ both in \eqref{characterization} and in
\eqref{complementary}.
 See \cite{CabreFL03b} for
a discussion of adapted metrics for \eqref{complementary}.
Hence, for  theoretical purposes (including in this paper)
one can assume that $C=1$ in both \eqref{characterization}  and \eqref{complementary}.

In the case when $f$ is symplectic, it is natural to consider
hyperbolic manifolds with the property that
\begin{equation} \label{restriction}
\begin{split}
& \eta_- = 1/\eta_+,\,\quad \lambda_+ = 1/\mu_-, \textrm{ and also } \\
&\lambda_- = 1/\mu_+.
\end{split}
\end{equation}
As shown in \cite{DelshamsLS08a}, normally hyperbolic invariant manifolds for
symplectic maps with the restricted exponents as in \eqref{restriction}
enjoy many geometric properties (e.g. the  map restricted to the manifold
is symplectic).
Note however that, even for symplectic
maps  there are normally hyperbolic invariant  manifolds
that satisfy the general definition but not \eqref{restriction}.
A notable example  is the stable manifold of  a NHIM, which is
normally hyperbolic according to the general definition
(this plays an important role in \cite{Fenichel71})
but does not satisfy \eqref{restriction}  and, indeed,
the map restricted to it is not symplectic.

Assume that there exists an integer $\ell>0$ such that
\[\ell\leq\min(r, \log\lambda_-^{-1} / \log \eta_+^{-1} ,  \log \eta_- /\log \mu_+ ).\]
Then $\Lambda$
is $C^\ell$-differentiable, and its stable and unstable manifolds
$W^s(\Lambda)$, $W^u(\Lambda)$ are $C^{\ell}$-differentiable
manifolds. See \cite{Robinson1971}.

The manifolds $W^s(\Lambda)$, $W^u(\Lambda)$  are foliated by stable and  unstable manifolds of points $W^s(z)$, $W^u(z')$ respectively,
with $z,z'\in\Lambda$, which are $C^r$-differentiable manifolds. The foliations are $C^{\ell-1}$-differentiable.
For each $x\in W^s(\Lambda)$
 there exists a unique $x^+ \in \Lambda$ such that $x \in W^s(x^+)$, and for each $x\in W^u(\Lambda)$ there exists a unique
$x^- \in \Lambda$  such that $x \in W^u(x^-)$.
We define the wave maps:
$$
\begin{array}{rcl}
\Omega^+ &:& W^s(\Lambda) \to \Lambda \ \mbox{by} \ \Omega^+(x) = x^+ \\
\Omega^- &:& W^u(\Lambda) \to \Lambda \ \mbox{by} \ \Omega^-(x) = x^-.
\end{array}
$$
The maps $\Omega^+$ and $\Omega^-$ are     $C^{\ell -1}$-smooth.

We assume that there exists a  transverse homoclinic  manifold  $\Gamma \subseteq M$, which is $C^{\ell-1}$-differentiable. This means that
  $\Gamma \subseteq W^u(\Lambda) \cap W^s(\Lambda)$
and, for each $x\in\Gamma $, we have
\begin{equation}\label{eqn:channel1}
\begin{split}
T_xM=T_xW^u(\Lambda)+T_xW^s(\Lambda),\\
T_x\Gamma =T_xW^u(\Lambda)\cap T_xW^s(\Lambda).
\end{split}
\end{equation}
We assume the additional conditions that for each  $x\in\Gamma $ we
have
\begin{equation}\label{eqn:channel2}
\begin{split}
T_xW^s(\Lambda)=T_xW^s(x^+)\oplus T_x \Gamma,\\
T_xW^u(\Lambda)=T_xW^u(x^-)\oplus T_x \Gamma,
\end{split}
\end{equation}
where $x^-,x^+$ are the uniquely defined points in $\Lambda$ corresponding to $x$.
Following \cite{DelshamsLS08a}, we say that $\Gamma$ is transverse to the foliations and therefore
$\Omega^-,\Omega^+$ restricted to $\Gamma $
are diffeomorphisms. We call $\Gamma $ an a homoclinic channel.
Hence,  we can define a scattering map
\[\sigma :\Omega^{-}(\Gamma) \to \Omega^{+}(\Gamma),\quad \sigma=\Omega^+\circ (\Omega^-)^{-1},\]
which is a diffeomorphism from $\Omega^{-}(\Gamma)$ to $\Omega^{+}(\Gamma)$.

If $\sigma(x^-)=x^+$, then there exits a unique $x\in\Gamma $ such that $W^u(x^-)\cap W^s(x^+)\cap \Gamma =\{x\}$.
Note that the backwards orbit $f^{-n}(x)$ of $x$ in $M$ is asymptotic to the backwards orbit $f^{-n}(x^-)$ in $\Lambda$, and the forward orbit $f^{m}(x)$ of $x$ in $M$ is asymptotic to the forward orbit $f^{m}(x^+)$ in $\Lambda$.

\section{Linearized coordinates}\label{sec:linearized_coordinates}
We will construct all windows used in Section \ref{proofshadowinglemma} in linearized coordinates, which will be recalled below, following  \cite{PughS70}.

Let $\Lambda$ be a normally hyperbolic invariant manifold for $f$ in $M$.
There exists an open neighborhood $V_{\Lambda}$ of $\Lambda$ in $M$, an open neighborhood $U_{\Lambda}$ of the zero section of
$(E^u\oplus E^s)_{ \mid \Lambda}$,
and a  homeomorphism $h$ from $U_{\Lambda}$ to $V_{\Lambda}$ such that
for every $(x^c,v^u, v^s)\in  (E^u\oplus E^s)_{ \mid \Lambda}$
\[
(h^{-1}\circ f\circ h)(x^c,v^u, v^s)=Nf(x^c,v^u, v^s)=(f_{\mid\Lambda}(x^c),Df(x^c)_{\mid E^u\oplus E^s}(v^u, v^s)).
\]
Via this coordinate system, each point $p\in V_{\Lambda}$ can be written uniquely through $(x^c,v^u,v^s)$ for some $x^c\in\Lambda$, $v^u\in E^u$, $v^s\in E^s$,
as $p= h(x^c,v^u,v^s)$.

In  the linearized coordinates,  the map $f$ is conjugate with the normal mapping $Nf_{ \mid E^u\oplus E^s }$ of $f$ in a neighborhood of $\Lambda$.
Hence, iterating a rectangle in these coordinates by the map $f$, for an arbitrary number of times, is equivalent to iterating the rectangle by the normal mapping $Nf$.

\section{Correctly aligned windows}\label{section:windows}
We review briefly the topological method of correctly aligned windows. We follow \cite{GideaZ04a} (see also \cite{GideaR03,GideaL06}).

\begin{defn}
An $(m_1,m_2)$-window in an $m$-dimensional manifold $M$, where $m_1+m_2=m$, is a
a $C^0$-homeomorphism $\chi$ from some open neighborhood $\textrm{dom}(\chi)$ of $[0,1]^{m_1}\times [0,1]^{m_2}$ in
$\mathbb{R}^{m_1}\times \mathbb{R}^{m_2}$ to an open subset $\textrm{im}(\chi)$  of $M$,
together with the homeomorphic image $W=\chi([0,1]^{m_1}\times [0,1]^{m_2})$, and with a choice of an `exit set' \[W^{\rm
exit} =\chi \left(\partial[0,1]^{m_1}\times [0,1]^{m_2} \right )\]
and  of an `entry set'  \[W^{\rm entry}
=\chi \left([0,1]^{m_1}\times
\partial[0,1]^{m_2}\right ).\]
\end{defn}

\begin{rem}\label{rem:windows_alternative}
Alternatively, we can define a window as a $C^0$-family of $m_1$-dimensional disks attached to an $m_2$-dimensional disk, i.e.,
\[W=\bigcup_{q\in D^{m_1}} D^{m_2}(q),\] with $D^{m_1}$ being some fixed $m_2$-dimensional disk, and  $D^{m_2}(q)$
being $m_1$-dimensional disks  depending in a $C^0$-fashion on $q\in  D^{m_1}$. In which case
 \begin{equation*}\begin{split}W^{\rm
exit} =\bigcup_{q\in D^{m_1}} \partial D^{m_2}(q)
\\ W^{\rm entry}=\bigcup_{q\in \partial D^{m_1}} D^{m_2}(q)
.\end{split}\end{equation*}
\end{rem}

In the sequel, when we refer to a window we mean the set $W$ together with the underlying local parametrization $\chi$.

\begin{defn}\label{defn:corr}
Let  $W_1$ and $W_2$ be $(m_1,m_2)$-windows, let $\chi_1$ and $\chi_2$ be the corresponding local parametrizations.
Let $f$ be a continuous
map on $M$ with $f(\textrm {im}(\chi_1))\subseteq \textrm
{im}(\chi_2)$, and let $f_\chi= \chi_2^{-1}\circ f\circ \chi_1$.  We say that $W_1$ is correctly aligned with
$W_2$ under $f$ if the following conditions are satisfied:
\begin{itemize}
\item[(i)] There exists a continuous homotopy $h:[0,1]\times ([0,1]^{m_1}\times [0,1]^{m_2})
 \to {\mathbb R}^{m_1} \times {\mathbb R}^{m_2}$,
   such that the following conditions hold true
   \begin{eqnarray*}
      h_0&=&f_\chi, \\
      h([0,1],\partial[0,1]^{m_1}\times [0,1]^{m_2}) \cap ([0,1]^{m_1}\times [0,1]^{m_2})&=& \emptyset, \\
      h([0,1],[0,1]^{m_1}\times [0,1]^{m_2}) \cap ([0,1]^{m_1}\times \partial[0,1]^{m_2})&=& \emptyset,
   \end{eqnarray*}

\item[(ii)] There exists $y_0\in[0,1]^{m_2}$ such that
the map $A_{y_0}:[0,1]^{m_1}\to\mathbb{R}^{m_1}$ defined by
$A_{y_0}(x)=\pi _{m_1}\left(h_{1}(x, y_0)\right )$ satisfies
\begin{eqnarray*}
A_{y_0}\left ( \partial[0,1]^{m_1}\right )\subseteq \mathbb
{R}^{m_1}\setminus [0,1]^{m_1},\\\deg({A_{y_0}},0)\neq 0,\end{eqnarray*}
where $\pi_{m_1}: \mathbb{R}^{m_1}\times \mathbb{R}^{m_2}\to
\mathbb{R}^{m_1}$ is the  projection onto the first
component, and $\deg(\cdot, 0)$ is the Brouwer degree of a map  at
$0$.
\end{itemize}
\end{defn}

The following is a shadowing lemma type of result for correctly aligned windows.
\begin{thm}
\label{theorem:detorb} Let $f:M\to M$ be a homeomorphism, $W_i$ be a collection of
$(m_1,m_2)$-windows in $M$,  and $\{t_i\}$ be a collection of positive integers, where $i\in\mathbb{Z}$.
If $W_i$ is correctly aligned with $W_{i+1}$ under $f^{t_i}$ for each $i$,  then
there exists a point $p\in W_0$ such that
\[(f^{t_i}\circ \ldots\circ f^{t_0})(p)\in W_{i+1} \textrm { for all } i.\]
Moreover, if for some $k>0$ we have $t_{i+k}=t_i$ and $W_{i+k}=W_{i}$ for all $i$, then the
point $p$ can be chosen periodic of period  $t_0+\ldots+t_{k-1}$.
\end{thm}

The correct alignment satisfies a natural product property.
Given two windows and a map, if each window can be written as a product of window components, and if the components of the first
window are correctly aligned with the corresponding components of
the second window under the appropriate components of the map, then
the first window is correctly aligned with the second window under
the given map. The details can be found in \cite{GideaL06}.

We describe the product property in a special case,  which corresponds to the situation considered in the paper.

Let $f:M\to M$ be a homeomorphism of the $m$-dimensional manifold $M$.
Denote by $B^k_\rho(x)$ the $k$-dimensional closed ball of radius $\rho$ centered at the point $x$ in $\mathbb{R}^k$.
Assume that $c,u,s\in\mathbb{N}$ are such that $c+u+s=m$, and write each $x\in\mathbb{R}^m$ as $x=(x^c,x^u,x^s)$, with $x^c\in\mathbb{R}^c$, $x^u\in\mathbb{R}^u$, and $x^s\in\mathbb{R}^s$. Let $p_1$, $p_2$ be two points in $M$, and let $\chi_1,\chi_2$ be two systems of local coordinates about $p_1,p_2$, respectively. Relative to these coordinate systems, we write $p_1=(p^c_1,p^u_1,p^s_1)$ and $p_2=(p^c_2,p^u_2,p^s_2)$.

\begin{lem}\label{ex:product}
Given two sets, $W_1$ in the local chart around $p_1$, and $W_2$ in the local chart around $p_2$, such  that,
in the corresponding local coordinates, we have
\[
\begin{split}W_1=B^{c}_{\rho^c_1}(p_1^c)\times B^{u}_{\rho^u_1}(p^u_1)\times B^{s}_{\rho^s_1}(p^s_1),\\
W_2=B^{c}_{\rho^c_2}(p^c_2)\times B^{u}_{\rho^u_2}(p^u_2)\times B^{s}_{\rho^s_1}(p^s_2),
\end{split}
\]
for some $\rho^c_1, \rho^u_1, \rho^s_1, \rho^c_2, \rho^s_2, \rho^u_2>0$.
Let
\[
\begin{split}
W_1^\xt=&{\quad}\partial B^{c}_{\rho^c_1}(p_1^c)\times B^{u}_{\rho^u_1}(p^u_1)\times B^{s}_{\rho^s_1}(p^s_1)\\
& \cup  B^{c}_{\rho^c_1}(p_1^c)\times \partial B^{u}_{\rho^u_1}(p^u_1)\times B^{s}_{\rho^s_1}(p^s_1),\\
W_1^\nt=&B^{c}_{\rho^c_1}(p_1^c)\times B^{u}_{\rho^u_1}(p^u_1)\times \partial B^{s}_{\rho^s_1}(p^s_1),\\
W_2^\xt=&{\quad}\partial B^{c}_{\rho^c_2}(p_2^c)\times B^{u}_{\rho^u_2}(p^u_2)\times B^{s}_{\rho^s_2}(p^s_2)\\
& \cup  B^{c}_{\rho^c_2}(p_2^c)\times \partial B^{u}_{\rho^u_2}(p^u_2)\times B^{s}_{\rho^s_2}(p^s_2),\\
W_2^\nt=&B^{c}_{\rho^c_2}(p_2^c)\times B^{u}_{\rho^u_2}(p^u_2)\times \partial B^{s}_{\rho^s_2}(p^s_2).
\end{split}
\]

Assume that the map $f$, written in local coordinates,  satisfies the following conditions relative to $W_1$ and $W_2$:
\[
\begin{split}
\pi_c\circ f(B^{c}_{\rho^c_1}(p_1^c)\times\{p^u_1\}\times \{p^s_1\})\supseteq   B^{c}_{\rho^c_2}(p_2^c),\\
\pi_u\circ f(\{p_1^c\}\times B^{u}_{\rho^u_1}(p^u_1)\times \{p^s_1\})\supseteq   B^{u}_{\rho^u_2}(p^u_2),\\
\pi_s\circ f(\{p_1^c\}\times \{p^u_1\}\times B^{s}_{\rho^s_1}(p^s_1))\subseteq   B^{s}_{\rho^s_2}(p^s_2),
\end{split}
\]
where $\pi_c,\pi_u,\pi_s$ denote the standard projections onto $\mathbb{R}^c$, $\mathbb{R}^u$, $\mathbb{R}^s$ respectively.

Then $W_1$ and $W_2$ are $(c+u,s)$-windows, and $W_1$ is correctly aligned with $W_2$ under~$f$.
\end{lem}

This lemma is an immediate consequence of Proposition 3 in  \cite{GideaL06}.

\section{An example of D. Turaev}\label{Turaev}
We are very grateful to  Dmitry Turaev who provided to us an example that shows that
a  `uniform-time' version of the shadowing lemma analogue of Lemma \ref{lem:key1-alternative}
is not true in general.

This example shows that the requirement  that \[m_i \ge m_i^*(n_0, \ldots,n_{i-1}, m_0,\ldots,m_{i-1})\] in Lemma \ref{lem:key1-alternative} cannot be
replaced by $m_i \ge m^*$, where $m^*$ is a constant.

\begin{ex}\label{ex:Turaev}
Let $M=\mathbb{R}^3$, $f:M\to M$ be a $C^1$-map,   and $\Lambda$ be a straight line in $M$ that is a normally hyperbolic invariant manifold for $f$ as follows.
There exists a system of coordinates  $(x,u,v)$ in a neighborhood $V$ of $\Lambda$ in $M$, with $x\in \mathbb{R}$ representing the coordinate on $\Lambda$,
and $u\in \mathbb{R}$ and $v\in \mathbb{R}$  the  contracting and expanding directions, respectively, and a corresponding open set
$U\subseteq \mathbb{R}^3$ of the form
\[
U=\{(x,u,v)\,:\,x\in\mathbb{R},  |u|<3/2,|v|<3/2\},
\]  such that for $p=(x,u,v)\in U$, the map $f$ is of the form $f(x,u,v,x)= (x',u',v')$ where
\begin{equation} \label{turaev1}
x'=x+(uv)^2, \,  u'=u/2,\,       v'= 2v.
\end{equation}
Thus $\Lambda$ corresponds to $u=v=0$,  and for each point $p=(x,0,0)\in\Lambda$, $W^s(p)=\{(x,u,0)\}$ and $W^u(p)=\{(x,0,v)\}$.
Moreover, $f_{\mid\Lambda}=\textrm{Id}$, and $f(W^{s,u}(p))=W^{s,u}(p)$ for every $p\in\Lambda$, that is, $f$ leaves invariant the stable and unstable fibers.

Assume that $W^u(\Lambda)$ and $W^s(\Lambda)$ intersect transversally along a homoclinic manifold
$$
\Gamma^-=\{(x,u,v)\,: \,u=0, v=1\},
$$
which is a line, and that for some power $q>0$ the map $f^q$ is of the form $f^q(x,u,v)=(x'',u'',v'')$,
\begin{equation} \label{turaev2}
x''=x+1,\, u''= 1+u, \,     v''=v-1.
\end{equation}
Thus  $f^q (\Gamma^-)=\Gamma^+ =\{(x,u,v)\,:\, u=1,v=0\}$, and the
corresponding scattering map $\sigma:\Lambda\to\Lambda$ is of the form
\begin{equation} \label{turaev3}
\sigma(x,0,0)=(x+1,0,0).
\end{equation}

Assume that for every $\delta>0$ there exists $n^*$ such that for every pseudo-orbit $y_{i+1}=f^{n_i}\circ\sigma(y_i)$ with $n_i\geq n^*$
there exists a true orbit $z_{i+1}=f^{n_i}(z_i)$ such that $d(z_i,y_i)<\delta$ for all $i$.
Take $\delta>0$  small and a corresponding $n^*$  sufficiently large.
Choose and fix a pseudo-orbit
$y_{i+1}=f^{n_i} \circ \sigma(y_i)$   with $n_i=n=n^*$ for all $i\geq 0$.
Let $y_i=(x_i,0,0)$.
Since $f_{\mid\Lambda}=\textrm{id}$ and $\sigma$ shifts the $x$-coordinate by $1$-unit, we have
\begin{equation} \label{turaev4}
x_{i+1}= x_i+1,
\end{equation}

Assume that there is an orbit $z_{i+1}=f^{n}(z_i)$ whose points are $\delta$-close to the corresponding points $y_i$. Let $z_i=(x'_i, u_i,  v_i)$.
Shadowing would imply that for all  $i\geq 0$ we have
\begin{equation} \label{turaev5}
|x'_i - x_i|<\delta,  \,   |u_i|<\delta, \,    |v_i|<\delta.
\end{equation}
First we show  that one of the iterations $f^k(z_i)$ with $0<k <n$ must lie outside the neighborhood $U$
of $\Lambda$  (i.e., the orbit between $z_i$  and $z_{i+1}$ makes an excursion along the homoclinic).
Indeed, note that as long as the orbit of $z_i$ lies in the neighborhood $U$ the product $uv$  stays constant by \eqref{turaev1},
so if the orbit between $z_i$ and $z_{i+1}$ stays in $U$ for all time, then
\[
x'_{i+1}=x'_i+n(u_i v_i)^2.
\]
Also by \eqref{turaev1}, we must have
$|v_i|< \frac{3}{2^{n+1}}$, so
$|x'_{i+1}-x'_i| < \frac{9n}{4^{n+1}}u_i^2$
which contradicts   \eqref{turaev5} and \eqref{turaev5} if $\delta$
is small. Thus, the orbit between $z_i$  and $z_{i+1}$ must leave at some point $f^{k+1}(z_i)$ the neighborhood $U$.


Thus, for each $i\geq 0$ there exists a positive integer $k_i<n-q$ such that the first $k_i$ iterations of $z_i$ stay in $U$, the next $q$ iterations stay outside $U$ following the homoclinic and returning to $U$, and the last $n-q-k_i$ iterations stay in $U$ again.
For the first $k_i$ iterations
the product $uv$ stays equal to $u_i v_i$, and for the last $n-q-k_i$ iterations the product stays constant and equals
to $u_{i+1} v_{i+1}$ (see \eqref{turaev1}).
Thus, by \eqref{turaev1}, \eqref{turaev2}
\begin{equation}\label{turaev6}
x'_{i+1} = x'_i +1 + k_i (u_i v_i)^2 +  (n-q-k_i) (u_{i+1} v_{i+1})^2
\end{equation}

Using \eqref{turaev1}, since $z_i$ leaves $U$ after $k_i<n-q$   forward iterations  we have
\[
|v_i| > {3}/{2^{n-q+2}} \textrm { for all }   i,
\]
and since   $z_i$ leaves $U$ after $n-q-k_i\leq n-q$ negative iterations
\[
|u_i| > {3}/{2^{n-q+2}} \textrm { for all }   i.
\]

Therefore,
\[
(u_i v_i)^2  > {1}/{2^{4n}} \textrm { for all }   i,
\]
hence \eqref{turaev6} implies
\[
x'_{i+1} > x'_i +1 + {(n-q)}/{2^{4n}}\geq  x'_i +1 + {1}/{2^{4n}},
\]
thus
\[
x'_{i+1} > x'_0 + i  + i({1}/{2^{4n}}).
\]

By  \eqref{turaev4}  \[x_i= x_0 +i,\]  the distance between  $x'_i$  and $x_i$ grows without bound as $i$ grows,
for any choice of $x_0$ and $x'_0$, so the shadowing property will be broken after finitely many iterations, for any choice of the constant $n$.
\end{ex}

\begin{rem}
The idea of this counter-example is that the dynamics off $\Lambda$ differs from the dynamics restricted to $\Lambda$
by some fixed amount of shift that depends on the $u$-, $v$-coordinates of a point.
Thus the shift between the points of the pseudo-orbit (which lie on $\Lambda$) and the points of a shadowing orbit that takes the
same number of iterates between successive points keeps increasing by the  fixed amount of shift at every step,
and the pseudo-orbit and the shadowing orbit end up  being far apart.
If we allow that the number of  iterates between successive points of the pseudo-orbit to vary, as it is the case in
Lemma  \ref{lem:key1-alternative}, we can arrange that the shadowing orbit gets closer and closer to $\Lambda$,
which makes the amount of shift between its points and the points of the pseudo-orbit getting smaller and smaller at every step.
More precisely, in the above example   we consider a shadowing orbit $z_{i+1}=f^{n_i}(z_i)$ with $n_i$ sufficiently large and depending on $i$,
the above estimates  yield an error term between $x'_{i+1}$ and $x_{i+1}$ of the order
\[\sum_{k=0}^{i} {n_i}/{2^{4n_i}},\]
which can be made arbitrarily small by choosing,  for instance,   $n_0$ sufficiently large and $n_{i}$ increasing at a linear rate.\end{rem}

\begin{rem} The proof of Lemma  \ref{lem:key1-alternative} uses the existence of a linearized system of coordinates $h$ in a neighborhood of
$\Lambda$ (cf. \cite{PughS70}). In the Example \ref{ex:Turaev}, such a system of coordinates is given
by \[h(x,u,v)=\left\{
                \begin{array}{ll}
                  \left(x+\frac{\ln|u|-\ln|v|}{2\ln 2}(uv)^2,u,v\right), & \hbox{$\textrm { for } u,v\neq 0$;} \\
                  (x,u,0), & \hbox{$\textrm { for } v\neq 0$}\\
                  (x,0,v), & \hbox{$\textrm { for } u\neq 0$},\\
                  (x,0,0), & \hbox{$\textrm { for } u=v=0$.}
                \end{array}
              \right.\]
Indeed, note that \begin{equation*}\begin{split}h\circ f(x,u,v)=&h(x+(uv)^2,u/2, 2v)\\=&\left(x+(uv)^2+\left(\frac{\ln |u|-\ln|v|}{2\ln2}-1\right)(uv)^2,u/2,2v\right)\\=&\left(x+\left(\frac{\ln |u|-\ln|v|}{2\ln2}\right)(uv)^2,u/2,2v\right)\\=&Nf\circ h(x,u,v).\end{split}\end{equation*}.
\end{rem}

\bibliographystyle{alpha}
\bibliography{diff}

\def\cprime{$'$} \def\cprime{$'$} \def\cprime{$'$}
\begin{thebibliography}{DGdlLS08b}

\bibitem[BCV01]{BessiCV01}
Ugo Bessi, Luigi Chierchia, and Enrico Valdinoci.
\newblock Upper bounds on {A}rnold diffusion times via {M}ather theory.
\newblock {\em J. Math. Pures Appl. (9)}, 80(1):105--129, 2001.

\bibitem[Ber08]{Bernard08}
Patrick Bernard.
\newblock The dynamics of pseudographs in convex {H}amiltonian systems.
\newblock {\em J. Amer. Math. Soc.}, 21(3):615--669, 2008.

\bibitem[Bes96]{Bessi96}
Ugo Bessi.
\newblock An approach to {A}rnol$\prime$d's diffusion through the calculus of
  variations.
\newblock {\em Nonlinear Anal.}, 26(6):1115--1135, 1996.

\bibitem[Bes97]{Bessi97}
Ugo Bessi.
\newblock Arnold's example with three rotators.
\newblock {\em Nonlinearity}, 10(3):763--781, 1997.

\bibitem[BFGS12]{BaldomaFGS12}
Inmaculada Baldom{\'a}, Ernest Fontich, Marcel Guardia, and Tere~M. Seara.
\newblock Exponentially small splitting of separatrices beyond {M}elnikov
  analysis: rigorous results.
\newblock {\em J. Differential Equations}, 253(12):3304--3439, 2012.

\bibitem[BKZ11]{KaloshinBZ11}
P.~Bernard, V.~Kaloshin, and K.~Zhang.
\newblock Arnold diffusion in arbitrary degrees of freedom and 3-dimensional
  normally hyperbolic invariant cylinders, 2011.

\bibitem[BLZ00]{BatesLZ00}
Peter~W. Bates, Kening Lu, and Chongchun Zeng.
\newblock Invariant foliations near normally hyperbolic invariant manifolds for
  semiflows.
\newblock {\em Trans. Amer. Math. Soc.}, 352(10):4641--4676, 2000.

\bibitem[BLZ08]{BatesLZ08}
Peter~W. Bates, Kening Lu, and Chongchun Zeng.
\newblock Approximately invariant manifolds and global dynamics of spike
  states.
\newblock {\em Invent. Math.}, 174(2):355--433, 2008.

\bibitem[BM06]{BolotinM06}
S.~Bolotin and R.~S. MacKay.
\newblock Nonplanar second species periodic and chaotic trajectories for the
  circular restricted three-body problem.
\newblock {\em Celestial Mech. Dynam. Astronom.}, 94(4):433--449, 2006.

\bibitem[Bol06]{Bolotin06}
Sergey Bolotin.
\newblock Symbolic dynamics of almost collision orbits and skew products of
  symplectic maps.
\newblock {\em Nonlinearity}, 19(9):2041--2063, 2006.

\bibitem[Bou12]{Bounemoura12}
Abed Bounemoura.
\newblock An example of instability in high-dimensional {H}amiltonian systems.
\newblock {\em Int. Math. Res. Not. IMRN}, (3):685--716, 2012.

\bibitem[BT99]{BolotinTreschev1999}
S.~Bolotin and D.~Treschev.
\newblock Unbounded growth of energy in nonautonomous {H}amiltonian systems.
\newblock {\em Nonlinearity}, 12(2):365--388, 1999.

\bibitem[CDMR06]{CanaliasDMR06}
E.~Canalias, A.~Delshams, J.~J. Masdemont, and P.~Rold{\'a}n.
\newblock The scattering map in the planar restricted three body problem.
\newblock {\em Celestial Mech. Dynam. Astronom.}, 95(1-4):155--171, 2006.

\bibitem[CFL03]{CabreFL03b}
Xavier Cabr{\'e}, Ernest Fontich, and Rafael de~la Llave.
\newblock The parameterization method for invariant manifolds. {II}.
  {R}egularity with respect to parameters.
\newblock {\em Indiana Univ. Math. J.}, 52(2):329--360, 2003.

\bibitem[CG94]{ChierchiaG94}
L.~Chierchia and G.~Gallavotti.
\newblock Drift and diffusion in phase space.
\newblock {\em Ann. Inst. H. Poincar\'e Phys. Th\'eor.}, 60(1):144, 1994.

\bibitem[CG08]{CressonG08}
Jacky Cresson and Christophe Guillet.
\newblock Hyperbolicity versus partial-hyperbolicity and the
  transversality-torsion phenomenon.
\newblock {\em J. Differential Equations}, 244(9):2123--2132, 2008.

\bibitem[CGDlL17]{CapinskiGL14}
M.~Capinski, M.~Gidea, and R.~De~la Llave.
\newblock Arnold diffusion in the planar elliptic restricted three-body
  problem: mechanism and numerical verification.
\newblock {\em Nonlinearity}, 30(1):329--360, 2017.

\bibitem[{Che}12]{Cheng12}
C.-Q. {Cheng}.
\newblock {Arnold diffusion in nearly integrable Hamiltonian systems}.
\newblock {\em ArXiv e-prints}, July 2012.

\bibitem[Chi79]{Chirikov79}
Boris~V. Chirikov.
\newblock A universal instability of many-dimensional oscillator systems.
\newblock {\em Phys. Rep.}, 52(5):264--379, 1979.

\bibitem[CW15]{CressonW15}
Jacky Cresson and Stephen Wiggins.
\newblock A {$\lambda$}-lemma for normally hyperbolic invariant manifolds.
\newblock {\em Regul. Chaotic Dyn.}, 20(1):94--108, 2015.

\bibitem[CX15]{ChengX2015}
C.-Q. {Cheng} and J.~{Xue}.
\newblock Arnold diffusion in nearly integrable hamiltonian systems of
  arbitrary degrees of freedom, 2015.

\bibitem[CY04]{ChengY2004}
Chong-Qing Cheng and Jun Yan.
\newblock Existence of diffusion orbits in a priori unstable {H}amiltonian
  systems.
\newblock {\em J. Differential Geom.}, 67(3):457--517, 2004.

\bibitem[CY09]{ChengY2009}
Chong-Qing Cheng and Jun Yan.
\newblock Arnold diffusion in {H}amiltonian systems: a priori unstable case.
\newblock {\em J. Differential Geom.}, 82(2):229--277, 2009.

\bibitem[dCNGM97]{MorrisonCNG1997}
D.~del Castillo-Negrete, J.~M. Greene, and P.~J. Morrison.
\newblock Renormalization and transition to chaos in area preserving nontwist
  maps.
\newblock {\em Phys. D}, 100(3-4):311--329, 1997.

\bibitem[DdlL00]{DelshamsL2000}
Amadeu Delshams and Rafael de~la Llave.
\newblock K{AM} theory and a partial justification of {G}reene's criterion for
  nontwist maps.
\newblock {\em SIAM J. Math. Anal.}, 31(6):1235--1269, 2000.

\bibitem[DdlLS00]{DelshamsLS00}
Amadeu Delshams, Rafael de~la Llave, and Tere~M. Seara.
\newblock A geometric approach to the existence of orbits with unbounded energy
  in generic periodic perturbations by a potential of generic geodesic flows of
  {${\bf T}^2$}.
\newblock {\em Comm. Math. Phys.}, 209(2):353--392, 2000.

\bibitem[DdlLS03]{DLS03}
Amadeu Delshams, Rafael de~la Llave, and Tere~M. Seara.
\newblock A geometric mechanism for diffusion in {H}amiltonian systems
  overcoming the large gap problem: announcement of results.
\newblock {\em Electron. Res. Announc. Amer. Math. Soc.}, 9:125--134, 2003.

\bibitem[DdlLS06a]{DelshamsLS2006}
Amadeu Delshams, Rafael de~la Llave, and Tere~M Seara.
\newblock A geometric mechanism for diffusion in {H}amiltonian systems
  overcoming the large gap problem: heuristics and rigorous verification on a
  model.
\newblock {\em Mem. Amer. Math. Soc.}, 179(844):viii+141, 2006.

\bibitem[DdlLS06b]{DelshamsLS2006b}
Amadeu Delshams, Rafael de~la Llave, and Tere~M Seara.
\newblock Orbits of unbounded energy in quasi-periodic perturbations of
  geodesic flows.
\newblock {\em Adv. Math.}, 202(1):64--188, 2006.

\bibitem[DdlLS08]{DelshamsLS08a}
Amadeu Delshams, Rafael de~la Llave, and Tere~M. Seara.
\newblock Geometric properties of the scattering map of a normally hyperbolic
  invariant manifold.
\newblock {\em Adv. Math.}, 217(3):1096--1153, 2008.

\bibitem[DdlLS16a]{DelshamsLS16}
Amadeu Delshams, Rafael de~la Llave, and Tere~M. Seara.
\newblock Instability of high dimensional {H}amiltonian systems: multiple
  resonances do not impede diffusion.
\newblock {\em Adv. Math.}, 294:689--755, 2016.

\bibitem[DdlLS16b]{DelshamsLS2013}
Amadeu Delshams, Rafael de~la Llave, and Tere~M. Seara.
\newblock Instability of high dimensional {H}amiltonian systems: multiple
  resonances do not impede diffusion.
\newblock {\em Adv. Math.}, 294:689--755, 2016.

\bibitem[DGdlLS08a]{DGLS08}
Amadeu Delshams, Marian Gidea, Rafael de~la Llave, and Tere~M. Seara.
\newblock Geometric approaches to the problem of instability in {H}amiltonian
  systems. {A}n informal presentation.
\newblock In {\em Hamiltonian dynamical systems and applications}, NATO Sci.
  Peace Secur. Ser. B Phys. Biophys., pages 285--336. Springer, Dordrecht,
  2008.

\bibitem[DGdlLS08b]{DelshamsGLS2008}
Amadeu Delshams, Marian Gidea, Rafael de~la Llave, and Tere~M. Seara.
\newblock Geometric approaches to the problem of instability in {H}amiltonian
  systems. {A}n informal presentation.
\newblock In {\em Hamiltonian dynamical systems and applications}, NATO Sci.
  Peace Secur. Ser. B Phys. Biophys., pages 285--336. Springer, Dordrecht,
  2008.

\bibitem[DGR13]{DelshamsGR13}
A.~Delshams, M.~Gidea, and P.~Roldan.
\newblock Transition map and shadowing lemma for normally hyperbolic invariant
  manifolds.
\newblock {\em Discrete and Continuous Dynamical Systems. Series A.},
  3(33):1089--1112, 2013.

\bibitem[DGR16]{DelshamsGR2016}
A.~Delshams, M.~Gidea, and P.~Roldan.
\newblock Arnold's mechanism of diffusion in the spatial circular restricted
  three-body problem: A semi-analytical argument.
\newblock {\em Physica D: Nonlinear Phenomena}, 2016.

\bibitem[DH09]{DelshamsH2009}
Amadeu Delshams and Gemma Huguet.
\newblock Geography of resonances and {A}rnold diffusion in a priori unstable
  {H}amiltonian systems.
\newblock {\em Nonlinearity}, 22(8):1997--2077, 2009.

\bibitem[DKdlRS14]{DelshamsKRS14}
A.~Delshams, V.~Kaloshin, A.~de~la Rosa, and T.~Seara.
\newblock Parabolic orbits in the restricted three body problem.
\newblock Preprint, 2014.

\bibitem[dlL04]{Llave04}
Rafael de~la Llave.
\newblock Orbits of unbounded energy in perturbations of geodesic flows by
  periodic potentials. a simple construction.
\newblock Preprint, 2004.

\bibitem[dlLOP11]{LlaveOP}
Rafael de~la Llave, Arturo Olvera, and Nikola~P. Petrov.
\newblock Combination laws for scaling exponents and relation to the geometry
  of renormalization operators.
\newblock {\em Journal of Statistical Physics}, 143(5):889, 2011.

\bibitem[Fen72]{Fenichel71}
N.~Fenichel.
\newblock Persistence and smoothness of invariant manifolds for flows.
\newblock {\em Indiana Univ. Math. J.}, 21:193--226, 1971/1972.

\bibitem[Fen74]{Fenichel74}
N.~Fenichel.
\newblock Asymptotic stability with rate conditions.
\newblock {\em Indiana Univ. Math. J.}, 23:1109--1137, 1973/74.

\bibitem[FGKR16]{FejozGKR11}
Jacques F{\'e}joz, Marcel Guardia, Vadim Kaloshin, and Pablo Rold{\'a}n.
\newblock Kirkwood gaps and diffusion along mean motion resonances in the
  restricted planar three-body problem.
\newblock {\em J. Eur. Math. Soc.}, 18:2315 -- 2403, 2016.

\bibitem[FM00]{FontichM2000}
E.~Fontich and P.~Mart{\'{\i}}n.
\newblock Differentiable invariant manifolds for partially hyperbolic tori and
  a lambda lemma.
\newblock {\em Nonlinearity}, 13(5):1561--1593, 2000.

\bibitem[GdlL06a]{GideaL06b}
M.~Gidea and R.~de~la Llave.
\newblock Arnold diffusion with optimal time in the large gap problem.
\newblock 2006.

\bibitem[GdlL06b]{GideaL06}
Marian Gidea and Rafael de~la Llave.
\newblock Topological methods in the instability problem of {H}amiltonian
  systems.
\newblock {\em Discrete Contin. Dyn. Syst.}, 14(2):295--328, 2006.

\bibitem[GDlL16]{GideaL16}
M.~Gidea and R.~De~la Llave.
\newblock Global melnikov potential and homoclinic intersections in higher
  dimensional hamiltonian systems.
\newblock 2016.

\bibitem[GdlLar]{GideaL13}
Marian Gidea and Rafael de~la Llave.
\newblock Perturbations of geodesic flows by recurrent dynamics.
\newblock {\em Jour. of the EMS}, To appear.
\newblock arXiv:1307.1617.

\bibitem[GHS14]{GranadosHS14}
A.~Granados, S.~J. Hogan, and T.~M. Seara.
\newblock The scattering map in two coupled piecewise-smooth systems, with
  numerical application to rocking blocks.
\newblock {\em Phys. D}, 269:1--20, 2014.

\bibitem[GR03]{GideaR03}
Marian Gidea and Clark Robinson.
\newblock Topologically crossing heteroclinic connections to invariant tori.
\newblock {\em J. Differential Equations}, 193(1):49--74, 2003.

\bibitem[GR07]{GideaR07}
Marian Gidea and Clark Robinson.
\newblock Shadowing orbits for transition chains of invariant tori alternating
  with {B}irkhoff zones of instability.
\newblock {\em Nonlinearity}, 20(5):1115--1143, 2007.

\bibitem[GR12]{GideaR12}
M.~Gidea and C.~Robinson.
\newblock Diffusion along transition chains of invariant tori and aubry-mather
  sets.
\newblock {\em Ergodic Theory Dynam. Systems}, 2012.

\bibitem[GT08]{GelfreichT2008}
Vassili Gelfreich and Dmitry Turaev.
\newblock Unbounded energy growth in {H}amiltonian systems with a slowly
  varying parameter.
\newblock {\em Comm. Math. Phys.}, 283(3):769--794, 2008.

\bibitem[GT14]{GelfreichT2014}
V.~{Gelfreich} and D.~{Turaev}.
\newblock {Arnold Diffusion in a priory chaotic Hamiltonian systems}.
\newblock {\em ArXiv e-prints}, June 2014.

\bibitem[Hal97]{Haller97}
G.~Haller.
\newblock Universal homoclinic bifurcations and chaos near double resonances.
\newblock {\em J. Statist. Phys.}, 86(5-6):1011--1051, 1997.

\bibitem[Hal99]{Haller99}
G.~Haller.
\newblock {\em Chaos near resonance}.
\newblock Springer-Verlag, New York, 1999.

\bibitem[HPPS70]{HirschPalisPughShub1969}
M.~Hirsch, J.~Palis, C.~Pugh, and M.~Shub.
\newblock Neighborhoods of hyperbolic sets.
\newblock {\em Invent. Math.}, 9:121--134, 1969/1970.

\bibitem[HPS77]{HirschPS77}
M.W. Hirsch, C.C. Pugh, and M.~Shub.
\newblock {\em Invariant manifolds}, volume 583 of {\em Lecture Notes in Math.}
\newblock Springer-Verlag, Berlin, 1977.

\bibitem[Kal03]{Kaloshin03}
V.~Kaloshin.
\newblock Geometric proofs of {M}ather's connecting and accelerating theorems.
\newblock In {\em Topics in dynamics and ergodic theory}, volume 310 of {\em
  London Math. Soc. Lecture Note Ser.}, pages 81--106. Cambridge Univ. Press,
  Cambridge, 2003.

\bibitem[KZ12a]{KaloshinZ12a}
V.~Kaloshin and K.~Zhang.
\newblock Normally hyperbolic invariant manifolds near strong double resonance,
  2012.

\bibitem[KZ12b]{KaloshinZ12b}
V.~Kaloshin and K.~Zhang.
\newblock A strong form of {A}rnold diffusion for two and a half degrees of
  freedom, 2012.

\bibitem[KZ14]{KaloshinZ14}
V.~Kaloshin and K.~Zhang.
\newblock A strong form of {A}rnold diffusion for three and a half degrees of
  freedom, 2014.

\bibitem[LT83]{Lieberman83}
M.~A. Lieberman and Jeffrey~L. Tennyson.
\newblock Chaotic motion along resonance layers in near-integrable
  {H}amiltonian systems with three or more degrees of freedom.
\newblock In {\em Long-time prediction in dynamics ({L}akeway, {T}ex., 1981)},
  volume~2 of {\em Nonequilib. Problems Phys. Sci. Biol.}, pages 179--211.
  Wiley, New York, 1983.

\bibitem[Mar08]{Marco2008}
Jean-Pierre Marco.
\newblock Mod\`eles pour les applications fibr\'ees et les polysyst\`emes.
\newblock {\em C. R. Math. Acad. Sci. Paris}, 346(3-4):203--208, 2008.

\bibitem[Mar13]{Marco13}
Jean-Pierre Marco.
\newblock Generic hyperbolic properties of classical systems on the torus.
\newblock Preprint, 2013.

\bibitem[Mat04]{Mather04}
John~N. Mather.
\newblock Arnol\cprime d diffusion. {I}. {A}nnouncement of results.
\newblock {\em J. Math. Sci. (N. Y.)}, 124(5):5275--5289, 2004.

\bibitem[Mat10]{Mather10}
John~N. Mather.
\newblock Order structure on action minimizing orbits.
\newblock In {\em Symplectic topology and measure preserving dynamical
  systems}, volume 512 of {\em Contemp. Math.}, pages 41--125. Amer. Math.
  Soc., Providence, RI, 2010.

\bibitem[Mat12]{Mather12}
John~N. Mather.
\newblock Arnold diffusion by variational methods.
\newblock In {\em Essays in mathematics and its applications}, pages 271--285.
  Springer, Heidelberg, 2012.

\bibitem[Pes04]{Pesin04}
Ya~B Pesin.
\newblock {\em Lectures on partial hyperbolicity and stable ergodicity}.
\newblock European Mathematical Society, 2004.

\bibitem[Pif06]{Piftankin2006}
G.~N. Piftankin.
\newblock Diffusion speed in the {M}ather problem.
\newblock {\em Dokl. Akad. Nauk}, 408(6):736--737, 2006.

\bibitem[PS70]{PughS70}
Charles Pugh and Michael Shub.
\newblock Linearization of normally hyperbolic diffeomorphisms and flows.
\newblock {\em Invent. Math.}, 10:187--198, 1970.

\bibitem[PT07]{PiftankinT07}
G.~N. Piftankin and D.~V. Treshch{\"e}v.
\newblock Separatrix maps in {H}amiltonian systems.
\newblock {\em Uspekhi Mat. Nauk}, 62(2(374)):3--108, 2007.

\bibitem[Rob71]{Robinson1971}
Clark Robinson.
\newblock Differentiable conjugacy near compact invariant manifolds.
\newblock {\em Bol. Soc. Brasil. Mat.}, 2(1):33--44, 1971.

\bibitem[Sab15]{Sabbagh14}
Lara Sabbagh.
\newblock An inclination lemma for normally hyperbolic manifolds with an
  application to diffusion.
\newblock {\em Ergodic Theory Dynam. Systems}, 35(7):2269--2291, 2015.

\bibitem[SZ03]{ShatahZ03}
Jalal Shatah and Chongchun Zeng.
\newblock Orbits homoclinic to centre manifolds of conservative {PDE}s.
\newblock {\em Nonlinearity}, 16(2):591--614, 2003.

\bibitem[Ten82]{Tennyson82}
Jeffrey Tennyson.
\newblock Resonance transport in near-integrable systems with many degrees of
  freedom.
\newblock {\em Phys. D}, 5(1):123--135, 1982.

\bibitem[Tre02a]{Treschev02b}
D.~Treschev.
\newblock Multidimensional symplectic separatrix maps.
\newblock {\em J. Nonlinear Sci.}, 12(1):27--58, 2002.

\bibitem[Tre02b]{Treschev02a}
D.~Treschev.
\newblock Trajectories in a neighbourhood of asymptotic surfaces of a priori
  unstable {H}amiltonian systems.
\newblock {\em Nonlinearity}, 15(6):2033--2052, 2002.

\bibitem[Tre02c]{Treschev02c}
D.~Treschev.
\newblock Trajectories in a neighbourhood of asymptotic surfaces of a priori
  unstable {H}amiltonian systems.
\newblock {\em Nonlinearity}, 15(6):2033--2052, 2002.

\bibitem[Tre04]{Treschev04}
D.~Treschev.
\newblock Evolution of slow variables in a priori unstable {H}amiltonian
  systems.
\newblock {\em Nonlinearity}, 17(5):1803--1841, 2004.

\bibitem[Tre12]{Treschev12}
D.~Treschev.
\newblock Arnold diffusion far from strong resonances in multidimensional {\it
  a priori} unstable {H}amiltonian systems.
\newblock {\em Nonlinearity}, 25(9):2717--2757, 2012.

\bibitem[ZG04]{GideaZ04a}
Piotr Zgliczy{\'n}ski and Marian Gidea.
\newblock Covering relations for multidimensional dynamical systems.
\newblock {\em J. Differential Equations}, 202(1):32--58, 2004.

\bibitem[Zha11]{Zhang11}
Ke~Zhang.
\newblock Speed of {A}rnold diffusion for analytic {H}amiltonian systems.
\newblock {\em Invent. Math.}, 186(2):255--290, 2011.

\end{thebibliography}
\end{document}